\documentclass[reqno,11pt]{amsart}

\usepackage[colorlinks, linkcolor=blue, citecolor=red, urlcolor=cyan,pagebackref]{hyperref}
\usepackage{breakurl}

\usepackage{cleveref}

\usepackage[rgb]{xcolor} 
\definecolor{mygreen}{rgb}{0,0.7,0.3}
\definecolor{myblue}{rgb}{0,0.50,1.20}
\definecolor{orange}{rgb}{2.55,1.65,0}

\definecolor{fillred}{rgb}{1,0.9,0.9}
\definecolor{fillgreen}{rgb}{0.9,1,0.9}

 \definecolor{refkey}{rgb}{0,0.7,0.3}
 \definecolor{labelkey}{rgb}{1,0,0}

\usepackage{etex}
\usepackage{amsthm,amssymb,amscd}
\usepackage{mathrsfs}
\usepackage{bbm}
\usepackage
{graphicx}
\usepackage{epic,eepic}
\usepackage{tikz-cd}
\usepackage{stmaryrd}

\usepackage{here}
\usepackage{bm}
\usepackage[all]{xy}

\usepackage{tikz}
\usetikzlibrary{calc}
\usetikzlibrary{positioning}
\usetikzlibrary{patterns}
\usetikzlibrary{decorations.pathmorphing, arrows, decorations.pathreplacing}

\usepackage{version}
\newenvironment{NB}{\color{red}{\bf N.B.} \footnotesize}{}

\excludeversion{NB}
\excludeversion{NB2}
\numberwithin{equation}{section}

\newtheorem{thm}{Theorem}[section]
\newtheorem{cor}[thm]{Corollary}
\newtheorem{lem}[thm]{Lemma}
\newtheorem{prop}[thm]{Proposition}

\newtheorem{introthm}{Theorem}
\newtheorem{introconj}[introthm]{Conjecture}
\newtheorem{introcor}[introthm]{Corollary}

\theoremstyle{definition}
\newtheorem{dfn}[thm]{Definition}
\newtheorem{ex}[thm]{Example}

\newtheorem{rem}[thm]{Remark}
\newtheorem{conv}[thm]{Notation}

\crefname{thm}{Theorem}{Theorems}
\crefname{cor}{Corollary}{Corollaries}
\crefname{lem}{Lemma}{Lemmas}
\crefname{prop}{Proposition}{Propositions}
\crefname{dfn}{Definition}{Definitions}
\crefname{ex}{Example}{Examples}
\crefname{claim}{Claim}{Claims}
\crefname{conj}{Conjecture}{Conjectures}
\crefname{rem}{Remark}{Remarks}
\crefname{figure}{Figure}{Figures}
\crefname{section}{Section}{Sections}
\crefname{subsection}{Section}{Sections}
\crefname{appendix}{Appendix}{Appendices}
\crefname{assum}{Assumption}{Assumptions}
\crefname{conv}{Notation}{Notations}

\crefname{introthm}{Theorem}{Theorems}
\crefname{introcor}{Corollary}{Corollaries}
\crefname{introconj}{Conjecture}{Conjectures}

\def\C{{\mathbb C}}
\def\Z{{\mathbb Z}}
\def\ve{{\varepsilon}}
\def\P{{\mathcal{P} }}
\renewcommand{\mathbf}{\boldsymbol}

\newcommand{\std}{\mathrm{std}}
\newcommand{\Hom}{\mathop{\mathrm{Hom}}\nolimits}

\newcommand{\inn}{\mathrm{in}}
\newcommand{\out}{\mathrm{out}}
\newcommand{\uf}{\mathrm{uf}}
\newcommand{\tr}{\mathop{\mathrm{tr}}\nolimits}
\newcommand{\pr}{\mathop{\mathrm{pr}}\nolimits}

\newcommand\vw{{\overline{w_0}}}
\newcommand{\lie}{\mathfrak{g}}
\newcommand{\op}{\mathrm{op}}

\newcommand{\weight}{\mathop{\mathrm{wt}}}
\newcommand{\wtil}[1]{\widetilde{#1}}
\newcommand{\Rat}{{ \mathcal{K} }}

\newcommand\A{{\mathcal{A} }}
\newcommand\B{{\mathcal{B} }}
\newcommand\X{{\mathcal{X} }}
\newcommand{\cF}{\mathcal{F}}

\newcommand{\cU}{\mathcal{U}}
\newcommand{\fu}{\mathfrak{u}}
\newcommand{\cO}{\mathcal{O}}
\newcommand{\bA}{\mathbb{A}}
\newcommand\bs{{\boldsymbol{s}}}
\newcommand{\bD}{\boldsymbol{\Delta}}
\newcommand{\bJ}{\mathbf{J}}
\newcommand{\bT}{\mathbb{T}}
\newcommand{\bX}{\mathbf{X}}
\newcommand{\bG}{\mathbb{G}}
\newcommand{\bP}{\mathbb{P}}
\newcommand{\bL}{\mathbb{L}}
\newcommand{\sfb}{\mathsf{b}}
\newcommand{\sfS}{\mathsf{S}}
\newcommand{\sfT}{\mathsf{T}}
\newcommand{\sfu}{\mathsf{u}}
\newcommand{\sfh}{\mathsf{h}}
\newcommand{\wC}{\mathcal{C}}
\newcommand{\pF}{\mathbb{F}_{\mathrm{pos}}}
\newcommand{\ptF}[1]{\widetilde{\mathbb{F}}_{\mathrm{pos}, #1}}
\newcommand{\Grep}{$(\mathrm{G}_{\mathrm{rep}})$}

\def\L{{\mathcal{L}}}%
\newcommand\Conf{{\mathrm{Conf}}}%
\DeclareMathOperator{\Ad}{{\mathrm{Ad}}}

\newcommand{\GSuniv}{GS-universally}
\DeclareMathOperator{\Spec}{\mathrm{Spec}}

\makeatletter 
\newcommand{\subalign}[1]{%
  \vcenter{%
    \Let@ \restore@math@cr \default@tag
    \baselineskip\fontdimen10 \scriptfont\tw@
    \advance\baselineskip\fontdimen12 \scriptfont\tw@
    \lineskip\thr@@\fontdimen8 \scriptfont\thr@@
    \lineskiplimit\lineskip
    \ialign{\hfil$\m@th\scriptstyle##$&$\m@th\scriptstyle{}##$\hfil\crcr
      #1\crcr
    }%
  }%
}
\makeatother

\newcommand{\dprod}{\mathop{\overrightarrow{\prod}}\limits}

\newcommand\dnode[1]{\draw(#1) circle[radius=0.15] node{\scalebox{0.8}{$2$}}}

\newcommand\qarrow[2]{\draw[->,shorten >=2pt,shorten <=2pt] (#1) -- (#2) [thick];} 
\newcommand\qsarrow[2]{\draw[->,shorten >=3pt,shorten <=3pt] (#1) -- (#2) [thick];} 
\newcommand\qdarrow[2]{\draw[->,dashed,shorten >=2pt,shorten <=2pt] (#1) -- (#2) [thick];} 
\newcommand\qsharrow[2]{\draw[->,shorten >=4pt,shorten <=2pt] (#1) -- (#2) [thick];} 
\newcommand\qstarrow[2]{\draw[->,shorten >=2pt,shorten <=4pt] (#1) -- (#2) [thick];} 

\newcommand\qshdarrow[2]{\draw[->,dashed,shorten >=4pt,shorten <=2pt] (#1) -- (#2) [thick];} 
\newcommand\qstdarrow[2]{\draw[->,dashed,shorten >=2pt,shorten <=4pt] (#1) -- (#2) [thick];} 

\tikzset{
  mid arrow/.style={postaction={decorate,decoration={
        markings,
        mark=at position .5 with {\arrow[#1]{stealth}}
      }}},
}

\tikzset{pics/.cd,
handle/.style={code={
\draw (-0.72,0) to[bend left] (0.72,0);
\draw (-0.9,0.1) to[bend right] (0.9,0.1);
}}}

\newcommand{\quiverAthree}[3]{
    \path(#1) coordinate(x1);
    \path(#2) coordinate(x2);
    \path(#3) coordinate(x3);
    \foreach \j in {1,2,3}
    {
        \foreach \k in {1,2,3}
        {
            \foreach \l in {1,2,3}
            {
            \path($(x\j)!0.25*\l!(x\k)$) coordinate(x\j\k\l);
            }
        }
    }
    \foreach \i in {0,1} 
        \draw($(x233)!\i!(x133)$) circle(2pt) coordinate(V3\i);
    \foreach \i in {0,1,2} 
        \draw($(x232)!0.5*\i!(x132)$) circle(2pt) coordinate(V2\i);       
    \foreach \i in {0,1,2,3} 
        \draw($(x231)!0.333*\i!(x131)$) circle(2pt) coordinate(V1\i);    
    {\color{myblue}
        \foreach \j in {1,2,3}
            \draw(x21\j) circle(2pt) coordinate(Y\j);
    }
    \foreach \s in {1,2,3}
        {
        {\color{red}
        \pgfmathsetmacro{\t}{3-\s}
        \foreach \i in {0,...,\t}
            {
            \pgfmathsetmacro{\j}{\i+1}
            \qarrow{V\s\j}{V\s\i};
            }
        }
        }
    \qarrow{V30}{V21}
    \qarrow{V21}{V31}
    
    \qarrow{V20}{V11}
    \qarrow{V11}{V21}
    \qarrow{V21}{V12}
    \qarrow{V12}{V22}
    {\color{myblue}
    \qarrow{V10}{Y1}
    \qarrow{Y1}{V11}
    \qarrow{V11}{Y2}
    \qarrow{Y2}{V12}
    \qarrow{V12}{Y3}
    \qarrow{Y3}{V13}
    }
}

\newcommand{\quiverCthree}[3]{
    \path(#1) coordinate(x1);
    \path(#2) coordinate(x2);
    \path(#3) coordinate(x3);
    \foreach \j in {1,2,3}
    {
        \foreach \k in {1,2,3}
        {
            \foreach \l in {1,2,3}
            {
            \path($(x\j)!0.25*\l!(x\k)$) coordinate(x\j\k\l);
            }
        }
    }
    \foreach \i in {0,1,2,3}
    	\dnode{$(x233)!0.333*\i!(x133)$} coordinate(V3\i);
    \foreach \i in {0,1,2,3}
    {
        \foreach \s in {1,2}
        \draw($(x23\s)!0.333*\i!(x13\s)$) circle(2pt) coordinate(V\s\i); 
    }
    {\color{myblue}
        \foreach \j in {1,2}
            \draw(x21\j) circle(2pt) coordinate(Y\j);
        \dnode{x213} coordinate(Y3);
    }
    {\color{red}
    \foreach \i in {0,1,2}
        {\qsarrow{$(x233)!0.333*\i+0.333!(x133)$}{$(x233)!0.333*\i!(x133)$};
        \foreach \s in {1,2}
         \qarrow{$(x23\s)!0.333*\i+0.333!(x13\s)$}{$(x23\s)!0.333*\i!(x13\s)$};
        }
    }
    \qstarrow{V30}{V21}
    \qsharrow{V21}{V31}
    \qstarrow{V31}{V22}
    \qsharrow{V22}{V32}
    \qstarrow{V32}{V23}
    
    \qarrow{V20}{V11}
    \qarrow{V11}{V21}
    \qarrow{V21}{V12}
    \qarrow{V12}{V22}
    \qarrow{V22}{V13}
    {\color{myblue}
    \qarrow{V10}{Y1}
    \qarrow{Y1}{V11}
    \qarrow{V11}{Y2}
    \qarrow{Y2}{V12}
    \qsarrow{V32}{Y3}
    \qsarrow{Y3}{V33}
    }
}

\newcommand{\quiverCthreeY}[3]{
    \path(#1) coordinate(x1);
    \path(#2) coordinate(x2);
    \path(#3) coordinate(x3);
    \foreach \j in {1,2,3}
    {
        \foreach \k in {1,2,3}
        {
            \foreach \l in {1,2,3}
            {
            \path($(x\j)!0.25*\l!(x\k)$) coordinate(x\j\k\l);
            }
        }
    }
    \foreach \i in {0,1,2,3}
    	\dnode{$(x233)!0.333*\i!(x133)$} coordinate(V3\i);
    \foreach \i in {0,1,2,3}
    {
        \foreach \s in {1,2}
        \draw($(x23\s)!0.333*\i!(x13\s)$) circle(2pt) coordinate(V\s\i); 
    }
    {\color{myblue}
        \foreach \j in {1,2}
        {   \pgfmathsetmacro{\k}{4-\j}
            \draw(x21\k) circle(2pt) coordinate(Y\j);
        }
        \dnode{x211} coordinate(Y3);
    }
    {\color{red}
    \foreach \i in {0,1,2}
        {\qsarrow{$(x233)!0.333*\i+0.333!(x133)$}{$(x233)!0.333*\i!(x133)$};
        \foreach \s in {1,2}
         \qarrow{$(x23\s)!0.333*\i+0.333!(x13\s)$}{$(x23\s)!0.333*\i!(x13\s)$};
        }
    }
    \qstarrow{V30}{V21}
    \qsharrow{V21}{V31}
    \qstarrow{V31}{V22}
    \qsharrow{V22}{V32}
    \qstarrow{V32}{V23}
    
    \qarrow{V20}{V11}
    \qarrow{V11}{V21}
    \qarrow{V21}{V12}
    \qarrow{V12}{V22}
    \qarrow{V22}{V13}
    {\color{myblue}
    \qarrow{V10}{Y1}
    \qarrow{Y1}{V11}
    \qarrow{V11}{Y2}
    \qarrow{Y2}{V12}
    \qsarrow{V32}{Y3}
    \qsarrow{Y3}{V33}
    }
}

\begin{document}
\title[Wilson lines and their Laurent positivity]
{Wilson lines and their Laurent positivity}

\author[Tsukasa Ishibashi]{Tsukasa Ishibashi}
\address{Tsukasa Ishibashi, Mathematical Institute, Tohoku University, 
6-3 Aoba, Aramaki, Aoba-ku, Sendai, Miyagi 980-8578, Japan.}
\email{tsukasa.ishibashi.a6@tohoku.ac.jp}

\author[Hironori Oya]{Hironori Oya}
\address{Hironori Oya, Department of Mathematics, Tokyo Institute of Technology, 2-12-1 Ookayama, Meguro-ku, Tokyo 152-8551, Japan.}
\email{hoya@math.titech.ac.jp}

\date{\today}

\maketitle

\begin{abstract}
    For a marked surface $\Sigma$ and a semisimple algebraic group $G$ of adjoint type, we study the Wilson line morphism $g_{[c]}:\P_{G,\Sigma} \to G$ associated with the homotopy class of an arc $c$ connecting boundary intervals of $\Sigma$, which is the comparison element of pinnings via parallel-transport. The matrix coefficients of the Wilson lines give a generating set of the function algebra $\cO(\P_{G,\Sigma})$ when $\Sigma$ has no punctures. 
    The Wilson lines have the multiplicative nature with respect to the gluing morphisms introduced by Goncharov--Shen \cite{GS19}, hence can be decomposed into triangular pieces with respect to a given ideal triangulation of $\Sigma$. 
    We show that the matrix coefficients $c_{f,v}^V(g_{[c]})$ give Laurent polynomials with positive integral coefficients in the Goncharov--Shen coordinate system associated with any decorated triangulation of $\Sigma$, for suitable $f$ and $v$. 
\end{abstract}

\setcounter{tocdepth}{1}
\tableofcontents

\section{Introduction} 
The moduli space of $G$-local systems on a topological surface is a classical object of study, which has been investigated both from mathematical and physical viewpoints. \emph{Wilson loops} give a class of important functions (or gauge-invariant observables), which are obtained as the traces of the monodromies of $G$-local systems in some finite-dimensional representations of $G$.

For a marked surface $\Sigma$, Fock--Goncharov \cite{FG03} introduced two extensions $\A_{\widetilde{G},\Sigma}$ and $\X_{G,\Sigma}$ of the moduli space of local systems, each of which admits a natural \emph{cluster structure}. Here $\widetilde{G}$ is a simply-connected semisimple algebraic group, and $G=\widetilde{G}/Z(\widetilde{G})$ is its adjoint group. 
The cluster structures of these moduli spaces are distinguished collections of open embeddings of algebraic tori accompanied with weighted quivers, related by two kinds of \emph{cluster transformations}. The collection of weighted quivers is shared by $\A_{\widetilde{G},\Sigma}$ and $\X_{G,\Sigma}$, and thus they form a \emph{cluster ensemble} in the sense of \cite{FG09}. 
Such a cluster structure is first constructed by Fock--Goncharov \cite{FG03} when the gauge groups are of type $A_n$, by Le \cite{Le16} for type $B_n,C_n,D_n$ (and further investigated in \cite{IIO19}), and by Goncharov--Shen \cite{GS19} for all semisimple gauge groups, generalizing all the works mentioned above and giving a uniform construction. 

In \cite{GS19}, Goncharov--Shen introduced a new moduli space $\P_{G,\Sigma}$ closely related to the moduli space $\X_{G,\Sigma}$, which possesses the frozen coordinates that are missed in the latter. When $\partial \Sigma=\emptyset$, we have $\P_{G,\Sigma}=\X_{G,\Sigma}$, and otherwise the former includes additional data called the \emph{pinnings} assigned to boundary intervals. The supplement of frozen coordinates turns out to be crucial in the quantum geometry of moduli spaces: for example, it is manifestly needed in the relation with the quantized enveloping algebra in their work. 
The data of pinnings also allow one to glue the $G$-local systems along boundary intervals in an unambiguous way, which leads to a \emph{gluing morphism}
\begin{align*}
    q_{E_1,E_2}: \P_{G,\Sigma} \to \P_{G,\Sigma'}.
\end{align*}
Here $\Sigma'$ is obtained from $\Sigma$ by gluing two boundary intervals $E_1$ and $E_2$ of $\Sigma$. 

\subsection{The Wilson lines}
Using the data of pinnings, we introduce a new class of $G$-valued morphisms
\begin{align*}
    g_{[c]}: \P_{G,\Sigma} \to G,
\end{align*}
which we call the \emph{Wilson line} 
along the homotopy classes $[c]$ of a curve connecting two boundary intervals called an \emph{arc class}. 
Roughly speaking, the Wilson line $g_{[c]}$ is defined to be the comparison element of the two pinnings assigned to the initial and terminal boundary intervals under the parallel-transport along the curve $c$. Our aim in this paper is a detailed study of these morphisms. 
Here are main features:
\begin{description}
\item[Multiplicativity] We will see that the Wilson lines have the multiplicative nature for the gluing morphisms. If we have two arc classes $[c_1]:E_1 \to E_2$ and $[c_2]:E'_2 \to E_3$ on $\Sigma$, then by gluing the boundary intervals $E_2$ and $E'_2$ we obtain another marked surface $\Sigma'$ equipped with an arc class $[c]:=[c_1]\ast [c_2]$, which is the concatenation of the two arcs. Then we will see that the Wilson line $g_{[c]}$ is given by the product of the Wilson lines $g_{[c_1]}$ and $g_{[c_2]}$. 
See \cref{prop:multiplicativity} and \cref{fig:multiplicativity}. 
\item[Open analogue of Wilson loops] 
Let
\begin{align*}
    \rho_{|\gamma|}:\P_{G,\Sigma} \to [G/\Ad G]
\end{align*}
be the morphism given by the monodromy along a free loop $|\gamma|$, which we call the \emph{Wilson loop} in this paper\footnote{In literature, the composition of this function with the trace in a finite-dimensional representation of $G$ is called a Wilson line. We call them the \emph{trace functions} in this paper. 
}.
Using the multiplicativity above, one can compute the Wilson loop from the Wilson line along the arc obtained by cutting the loop $\gamma$ along an edge. See \cref{prop:Wilson line-loop} and \cref{fig:multiplicativity_connected_case}. In this sense, the Wilson lines are \lq\lq open analogues'' of the Wilson loops. 
\item[Generation of the function algebra] The matrix coefficients of Wilson lines give rise to regular functions on $\P_{G,\Sigma}$. Moreover, we will see in \cref{subsec:generate} that the function algebra $\cO(\P_{G,\Sigma})$ is generated by these matrix coefficients when $\Sigma$ has no punctures. Therefore Wilson lines provide enough functions to study the function algebra $\cO(\P_{G,\Sigma})$. 
\item[Universal Laurent property] Shen \cite{Shen20} proved that the algebra $\cO(\P_{G,\Sigma})$ of regular functions on this moduli stack is isomorphic to the cluster Poisson algebra $\cO_{\mathrm{cl}}(\P_{G,\Sigma})$, which is by definition the algebra of regular functions on the corresponding cluster Poisson variety. Hence the matrix coefficients of Wilson lines belong to $\cO_{\mathrm{cl}}(\P_{G,\Sigma})$. In other words, they are \emph{universally Laurent polynomials}, meaning that they are expressed as Laurent polynomials in \emph{any} cluster chart (including those not coming from decorated triangulations). 
\end{description}

We remark here that the essential notion of Wilson lines has been appeared in many related works including \cite{FG03,GMN14,GS14,SS17,GS19,CS20} (mainly as a tool for the computation of Wilson loops), while our work would be the first on its systematic study in the setting of the moduli space $\P_{G,\Sigma}$. Via their coordinate expressions as we discuss below, the Wilson lines (loops) have been recognized as related to the spectral networks \cite{GMN14} and certain integrable systems \cite{SS17}.

\subsection{Laurent positivity of Wilson lines}
Our goal in this paper is a detailed study of the Laurent expressions of the matrix coefficients of Wilson lines in cluster charts on $\P_{G,\Sigma}$. Moreover, it will turn out that a special class of matrix coefficients give rise to Laurent polynomials with non-negative coefficients. 

\smallskip
\paragraph{\textbf{Fock--Goncharov's snake formula.}}
Coordinate expressions of Wilson loops (or the trace functions) have been studied by several authors. 
In the $A_1$ case, a combinatorial formula for the expressions of Wilson loops in terms of the cross ratio coordinates is given by Fock \cite{Fock94} (see also \cite{Penner,FG07}). It expresses the Wilson loop along a free loop $|\gamma|$ as a product of the elementary matrices 
\begin{align*}
    L=\begin{pmatrix}1 & 1\\ 0 & 1 \end{pmatrix},\quad 
    R=\begin{pmatrix}1 & 0\\ 1 & 1 \end{pmatrix},\quad 
    H(x)=\begin{pmatrix}x^{1/2} & 0\\ 0 & x^{-1/2}\end{pmatrix}\quad \in PGL_2,
\end{align*}
which are multiplied according to the turning pattern after substituting the cross ratio coordinates into $x$.  

In the $A_n$ case, Fock--Goncharov \cite{FG06} 
gave a similar formula called the \emph{snake formula}, which expresses the Wilson loops in the cluster coordinates associated with ideal triangulations (called the \emph{special coordinate systems}). In particular, the trace functions are positive Laurent polynomials (with fractional powers) in any special coordinate systems. 

\smallskip
\paragraph{\textbf{Generalizations of the snake formula.}}
Generalizing the special coordinate systems, Goncharov--Shen \cite{GS19} gave a uniform construction of coordinate systems on $\P_{G,\Sigma}$ associated with \emph{decorated triangulations}\footnote{Actually, they described more coordinate systems geometrically: those along flips of ideal triangulations and along rotations of dots. We do not investigate these additional coordinate systems in this paper.}.
Let us call them the \emph{Goncharov--Shen coordinate systems} (\emph{GS coordinate systems} for short). 
The special coordinate systems in the type $A_n$ case are special instances of the GS coordinate systems, where the choice of reduced words are the \lq\lq standard'' one (see \eqref{eq:std_word}). Unlike the special coordinate systems, however, a general GS coordinate system no longer have the cyclic symmetry on each triangle. The data of \lq\lq directions'' of coordinates is encoded in the data of decorated triangulations, as well as the choice of reduced words on each triangle. 

Locally, a natural generalization of the snake formula is given by the \emph{evaluation map} \cite{FG06} parametrizing the \emph{double Bruhat cells} of $G$. We will see that the ``basic'' Wilson lines $b_L,b_R$ on the configuration space $\Conf_3 \P_G$ ,which models the moduli space on the triangle, can be expressed using the evaluation maps. 
Since the multiplicativity allows one to decompose Wilson lines into those on triangles, one can write the Wilson lines as a product of evaluation maps when the direction of GS coordinates agree with the direction that the arc class traverses on each triangle. This is basically the same strategy as Fock--Goncharov \cite{FG06}, but manipulations in the recently-innovated moduli space $\P_{G,\Sigma}$ makes the computation much clearer, thanks to the nice properties of the gluing morphism \cite{GS19}. 


\smallskip
\paragraph{\textbf{Transformations by cyclic shifts.}}
In general, we need to transform the evalutation maps in the expression of Wilson lines by the \emph{cyclic shift automorphism} on $\Conf_3 \P_G$,
in order to match the directions of a given GS coordinate system with the direction of the arc class on each triangle. The cyclic shifts are known to be written as a composite of cluster transformations, which is computable in nature but rather a complicated rational transformation. 
While the matrix coefficients of $g_{[c]}$ are at least guaranteed to be Laurent polynomials as we discussed above, it is therefore non-trivial whether their coefficients are non-negative integers. 

Let us further clarify the problem which we will deal with. 
A function $f \in \cO(\P_{G,\Sigma})$ is said to be \emph{\GSuniv\ positive Laurent} if it is expressed as a Laurent polynomial with non-negative integral coefficients in the GS coordinate system associated with any decorated triangulation $\boldsymbol{\Delta}$. This is a 
straightforward generalization of \emph{special good positive Laurent polynomials} on $\X_{PGL_{n+1},\Sigma}$ in \cite{FG03}. 
Moreover, a morphism $F:  \P_{G,\Sigma}\to G$ is said to be \emph{\GSuniv\ positive Laurent} if for any finite-dimensional representation $V$ of $G$, there exists a basis $\mathbb{B}$ of $V$ such that 
 \[
 F^\ast c_{f, v}^V \in \cO(\P_{G,\Sigma})
 \]
is \GSuniv\ positive Laurent for all $v\in \mathbb{B}$ and $f\in \mathbb{F}$, where $\mathbb{F}$ is the basis of $V^{\ast}$ dual to $\mathbb{B}$. Our result is the following:
\begin{introthm}[\cref{t:Wilson_line_positivity}]\label{introthm:positivity_line}
	Let $G$ be a semisimple algebraic group of adjoint type, and assume that our marked surface $\Sigma$ has non-empty boundary. 
	Then, for any arc class $[c]:E_\inn \to E_\out$, the Wilson line $g_{[c]}:  \P_{G,\Sigma}\to G$ is a \GSuniv\ positive Laurent morphism.  
\end{introthm}
Since the Wilson loops $\rho_{|\gamma|}: \P_{G,\Sigma} \to [G/\Ad G]$ can be computed from the Wilson lines by \cref{prop:Wilson line-loop}, it immediately implies the following: 

\begin{introcor}[\cref{t:monodromy_positivity}]\label{introcor:positivity_loop}
	Let $G$ be a semisimple algebraic group of adjoint type, and $|\gamma| \in \widehat{\pi}(\Sigma)$ a free loop. Then, for any finite dimensional representation $V$ of $G$, the trace function $\tr_V(\rho_{|\gamma|}):=\rho_{|\gamma|}^\ast\tr_V \in \cO(\P_{G,\Sigma})$ is \GSuniv\ positive Laurent. 
\end{introcor}
\cref{introcor:positivity_loop} is a generalization of \cite[Theorem 9.3, Corollary 9.2]{FG03}.

Here we briefly comment on the proof of \cref{introthm:positivity_line}. 
By the construction of the GS coordinate system on $\P_{G,\Sigma}$ associated with a decorated triangulation, the Laurent positivity of a regular function on $\P_{G,\Sigma}$ can be deduced from the Laurent positivity of its pull-back via the gluing morphism $q_\Delta: \prod_{T \in t(\Delta)} \P_{G,T}\to  \P_{G,\Sigma}$ associated with the underlying ideal triangulation $\Delta$. 
In other words, we can investigate the Laurent positivity of a regular function on $\P_{G,\Sigma}$ by a local argument on triangles. Indeed, a key to the proof of \cref{introthm:positivity_line} is a construction of 
a basis $\ptF{T}$ of $\cO(\P_{G,T})$ consisting of \GSuniv\ positive Laurent elements, which is invariant under the cyclic shift and compatible with certain matrix coefficients. 

We show that such a nice basis is constructed whenever we have a nice basis $\pF$ of the coordinate ring $\cO(U^+_{\ast})$ of the \emph{unipotent cell} $U^+_{\ast}$ of $G$. In particular, the invariance of $\ptF{T}$ under the cyclic shift on $\P_{G,T}$ comes from the invariance of $\pF$ under \emph{the Berenstein-Fomin-Zelevinsky twist automorphism on $U^+_{\ast}$} \cite{BFZ96,BZ97}. 
An example of a basis of $\cO(U^+_{\ast})$ which satisfies the list of desired properties (see \cref{t:positivechoice}) is obtained from the theory of \emph{categorification of $\cO(U^{+}_{\ast})$ via quiver Hecke algebras}, which has been investigated, for example, in  \cite{KL:I,Rou08,KL:II,Rou12,KK:hw,KKKO:monoidal,KKOP:strata,KKOP:loc}. 

\bigskip

The \GSuniv\ positive Laurent property is weaker than the \emph{universal positive Laurent property} \cite{FG09}, which requires a similar positive Laurent property for all cluster charts. By replacing the \GSuniv\ positive Laurent property with universal positive Laurent property, we have the notion of \emph{universally positive Laurent morphisms}. Then, it would be natural to expect the following:

\begin{introconj}\label{conj:universally Laurent}
For any arc class $[c]:E_\inn \to E_\out$, the Wilson line $g_{[c]}:  \P_{G,\Sigma}\to G$ is a universally positive Laurent morphism. 
Moreover, the trace function $\tr_V(\rho_{|\gamma|}) \in  \cO(\P_{G,\Sigma})$ is universally positive Laurent.
\end{introconj}
Indeed, it is known that this conjecture on the trace functions holds true for type $A_1$ case \cite{FG03}. In our continuing work \cite{IOS} with Linhui Shen, it is shown that the generalized minors of Wilson lines are cluster monomials. In particular, they are known to be universally positive Laurent \cite{GHKK}. 

\subsection{Future directions}

\subsubsection*{Poisson brackets of Wilson lines.}
The absolute values $|\tr_V(\rho_{|\gamma|})|$ of the trace functions in the vector representation are well-defined smooth functions on the positive-real part $\P_{G,\Sigma}(\mathbb{R}_{>0})$ (or $\X_{G,\Sigma}(\mathbb{R}_{>0})$), in spite of the fact that $V$ is not a representation of the adjoint group $G$. 
In the type $A_n$ case, Chekhov--Shapiro \cite{CS20} proved that the cluster Poisson brackets of these functions reproduce the Goldman brackets \cite{Go86}. 
Their argument is local in nature and seems to be applicable also to Wilson lines, and it can be expected that (absolute values of) certain matrix coefficients of the Wilson lines form an open analogue of the Goldman algebra. 


\subsubsection*{Quantum lifts of Wilson lines.}
Any cluster Poisson variety $\X$ admits a canonical quantization, namely a one-parameter deformation $\cO_q(\X)$ of the cluster Poisson algebra $\cO(\X)$ and its representation on a certain Hilbert space as self-adjoint operators \cite{FG08}. It will be an interesting problem to consider a quantum analogue of the matrix coefficients of the Wilson lines, which belong to $\cO_q(\P_{G,\Sigma})$ and recovers $c^V_{f,v}(g_{[c]})$ in the classical limit $q\to 1$. A special example is the Goncharov--Shen's realization of the quantum enveloping algebra $U_q(\mathfrak{b}^+)$ inside the quantum cluster Poisson algebra \cite[Section 11]{GS19}, whose generators are quantum lifts of certain matrix coefficients of the Wilson line along an arc class that encircles exactly one special point (see \cref{lem:Wilson_line_Borel}). In the type $A_n$ case, quantum lifts are also studied by Douglas \cite{Douglas21} and Chekhov--Shapiro \cite{CS20}.

A comparison with the quantization of the moduli stacks in terms of the factorization homology studied by \cite{JLSS} will also be an important problem.

\subsection*{Organization of the paper}
\begin{description}
\item[Geometric study of Wilson lines (Sections \ref{sec:coordinates}--\ref{sec:monodromy})]
After recalling basic notations in \cref{sec:coordinates}, we introduce the Wilson line morphisms in \cref{sec:monodromy} and study their properties from the geometric point of view. Some basic facts on the quotient stacks are summarized in \cref{sec:stacks}. 
In \cref{subsec:generate}, we prove the generation of the function algebra $\cO(\P_{G,\Sigma})$ by the matrix coefficients of  Wilson lines when $\Sigma$ has no punctures. 
We give the decomposition formulae for the Wilson lines in \cref{subsec:regularity_of_Wilson_lines_and_loops}, as a preparation for the study on the coordinate expressions.

\item[Coordinate expressions and Laurent positivity (Sections \ref{sec:coordinate expressions}--\ref{sec:positivity})] After recalling the Goncharov--Shen coordinates on the moduli space $\P_{G,\Sigma}$ and relevant coordinate systems on unipotent cells and double Bruhat cells in \cref{sec:coordinate expressions}, we study we study the coordinate expressions of the Wilson lines and prove \cref{introthm:positivity_line} in \cref{sec:positivity}. 
In the course of the proof, we construct a basis of $\cO(\P_{G,T})$ for a triangle $T$ consisting of \GSuniv\ positive Laurent elements, which is invariant under the cyclic shift. 
Some basic notions on the cluster varieties, weighted quivers and their amalgamation procedure are recollected in \cref{sec:quivers}. 
\end{description}

\subsection*{Acknowledgments}
The authors' deep gratitude goes to Linhui Shen for his insightful comments on this paper at several stages and explaining his works with Alexander Goncharov. 
They are grateful to Tatsuki Kuwagaki and Takuma Hayashi for explaining some basic notions and backgrounds on Artin stacks, and giving valuable comments on a draft of this paper. They also wish to thank Ryo Fujita for helpful discussions on quiver Hecke algebras. The authors are grateful to the anonymous referee for the careful reading of the paper and the significant comments. 
T. I. would like to express his gratitude to his former supervisor Nariya Kawazumi for his continuous guidance and encouragement in the earlier stage of this work. 
T. I. is partially supported by JSPS KAKENHI Grant Numbers 18J13304 and 20K22304, and the Program for Leading Graduate Schools, MEXT, Japan.
H. O. is supported by JSPS Grant-in-Aid for Early-Career Scientists (No. 19K14515).

\section{Configurations of pinnings}\label{sec:coordinates}

Denote by $\bG_m=\Spec \C[t,t^{-1}]$ the multiplicative group scheme over $\C$. For an algebraic torus $T$ over $\C$, let $X^*(T):=\Hom(T,\mathbb{G}_m)$ be the lattice of characters, $X_*(T):=\Hom(\mathbb{G}_m, T)$ the lattice of cocharacters, and $\langle -,-\rangle$ the natural pairing 
\[
\langle -,-\rangle\colon~ X_*(T)\times X^*(T)\longrightarrow \Hom(\mathbb{G}_m, \mathbb{G}_m)\simeq \Z. 
\]
For $t\in T$ and $\mu\in X^*(T)$, the evaluation of $\mu$ at $t$ is denoted by $t^{\mu}$.

\subsection{Notations from Lie theory}
In this subsection, we briefly recall basic terminologies in Lie theory. See \cite{Jan} for the details.

Let $\widetilde{G}$ be a simply-connected connected simple algebraic group over $\mathbb{C}$. Let $\wtil{B}^+$ be a Borel subgroup of $\widetilde{G}$ and $\wtil{H}$ a maximal torus (a.k.a. Cartan subgroup) contained in $\wtil{B}^+$, respectively. Let $U^+$ be the unipotent radical of $\wtil{B}^+$. 
Let 
\begin{itemize}
    \item $X^*(\wtil{H})$ be the weight lattice and $X_*(\wtil{H})$ the coweight lattice; 
    \item $\Phi\subset X^*(\wtil{H})$ the root system of $(\widetilde{G}, \wtil{H})$; 
    \item $\Phi_+\subset \Phi$ the set of positive roots consisting of the $\wtil{H}$-weights of the Lie algebra of $U^+$;
    \item $\{\alpha_s\mid s\in S\}\subset \Phi_+$ the set of simple roots, where $S$ is the index set with $|S|=r$;
    \item $\{\alpha_s^\vee \mid s \in S\} \subset X_\ast(\wtil{H})$ the set of simple coroots.
\end{itemize}
For $s\in S$, let $\varpi_s\in X^*(\wtil{H})$ be the $s$-th fundamental weight defined by  $\langle \alpha_t^{\vee}, \varpi_s\rangle=\delta_{st}$. Then we have $\alpha_t=\sum_{u \in S} C_{ut} \varpi_u$ for $t \in S$, where $C_{st}:=\langle \alpha_s^\vee, \alpha_t\rangle\in \mathbb{Z}$. 
We have \[X^*(\wtil{H})=\sum_{s\in S}\mathbb{Z}\varpi_s \quad\mbox{and}\quad X_*(\wtil{H})=\sum_{s\in S}\mathbb{Z}\alpha^\vee_s.\] 
The sub-lattice 
generated by $\alpha_s$ for $s\in S$ is  called the root lattice.

For $s\in S$, we have a pair of root homomorphisms $x_{s}, y_{s}\colon \bA^1\to \wtil{G}$ such that 
\[
hx_{s}(t)h^{-1}=x_{s}(h^{\alpha_s}t), \qquad hy_s(t)h^{-1}=y_s(h^{-\alpha_s}t)
\]
for $h\in \wtil{H}$. After a suitable normalization, we obtain a homomorphism 
$\varphi_{s}\colon SL_2\to \wtil{G}$
such that 
\[
\varphi_{s}\left( \begin{pmatrix}
    1&a\\
    0&1
\end{pmatrix}\right)=x_{s}(a), \qquad
\varphi_{s} \left(\begin{pmatrix}
1&0\\
a&1
\end{pmatrix}\right)= y_s(a), \qquad
\varphi_{s} \left(\begin{pmatrix}
a&0\\
0&a^{-1}
\end{pmatrix}\right)= \alpha_s^\vee(a).
\]

The group $G=\widetilde{G}/Z(\widetilde{G})$ is called the \emph{adjoint group}, where $Z(\widetilde{G})$ denotes the center of $\widetilde{G}$. Then $B^+:=\widetilde{B}^+/Z(\widetilde{G})$ is a Borel subgroup of $G$ and $H:=\widetilde{H}/Z(\widetilde{G})$ is a Cartan subgroup of $G$. Moreover the unipotent radical of $B^+$ is isomorphic to $U^+$ through the natural map $\wtil{G}\to G$, which we again denote by $U^+$. Then we have $B^+=HU^+$. The natural map $\wtil{H}\to H$ induces $\Z$-module homomorphisms 
\[
X^*(H)=\sum_{s\in S}\mathbb{Z}\alpha_s \hookrightarrow X^*(\wtil{H}) \quad\mbox{and}\quad X_*(\wtil{H})\twoheadrightarrow X_*(H)=\sum_{s\in S}\mathbb{Z}\varpi^\vee_s,
\]
where $\varpi_s^\vee\in X_*(H)$ is the $s$-th fundamental coweight defined by $\langle \varpi_s^{\vee}, \alpha_t\rangle=\delta_{st}$; we tacitly use the same notations for the elements related by these maps. 
The above mentioned one-parameter subgroups $x_s, y_s$ descend to the homomorphisms $x_{s}, y_{s}\colon \bA^1\to G$ with the same notation. There exists an anti-involution 
\[
\sfT:G\to G,\ g\mapsto g^\mathsf{T}
\]
of the algebraic group $G$ given by $x_s(t)^\mathsf{T}=y_s(t)$ and $h^\mathsf{T}=h$ for $s\in S, t\in\bA^1, h\in H$. This is called the \emph{transpose} in $G$. Let 
\begin{itemize}
    \item $B^-:=(B^+)^\mathsf{T}$ be the opposite Borel subgroup of $B^+$, and $U^-:=(U^+)^\mathsf{T}$; 
    \item $G_0:=U^-HU^+ \subset G$ the open subvariety of triangular-decomposable elements. 
\end{itemize}

\begin{dfn}
In $G$, define $\mathbb{E}^s:=x_s(1) \in U^+$ and $\mathbb{F}^s:=y_s(1) \in U^-$ for each $s \in S$. Let $H^s: \bG_m \to H$ be the one-parameter subgroup given by $H^s(a) = \varpi_s^\vee(a)$. 
\end{dfn}

\paragraph{\textbf{Weyl groups}}
Let $W(\widetilde{G}):=N_{\widetilde{G}}(\widetilde{H})/\widetilde{H}$ denote the Weyl group of $\widetilde{G}$, where $N_{\widetilde{G}}(\widetilde{H})$ is the normalizer subgroup of $\widetilde{H}$ in $\widetilde{G}$. For $s \in S$, we set \[\overline{r}_{s}:=\varphi_s\left(
\begin{pmatrix}
    0&-1\\
    1&0
\end{pmatrix}
\right) \in N_{\widetilde{G}}(\widetilde{H}).\]
The elements $r_s:=\overline{r}_s\wtil{H} \in W(\wtil{G})$ have order $2$, and give rise to a Coxeter generating set for $W(\wtil{G})$ with the following presentation:
\begin{align*}
    W(\wtil{G})= \langle r_s\ (s \in S) \mid (r_sr_t)^{m_{st}}=1\ (s,t \in S) \rangle,
\end{align*}
where $m_{st} \in \Z$ is given by the following table
\[
\begin{tabular}{rccccc}
$C_{st}C_{ts}:$ & $0$ & $1$ & $2$ & $3$  \\
$m_{st}:$         & $2$ & $3$ & $4$ & $6$ 
\end{tabular}.
\]
For a reduced word $\bs=(s_1,\dots,s_\ell)$ of $w \in W(\wtil{G})$, let us write $\overline{w}:=\overline{r}_{s_1}\dots\overline{r}_{s_\ell} \in N_{\widetilde{G}}(\widetilde{H})$, which does not depend on the choice of the reduced word. 
We have a left action of $W(\widetilde{G})$ on $X^*(\widetilde{H})$ induced from the (right) conjugation action of $N_{\widetilde{G}}(\widetilde{H})$ on $\widetilde{H}$. 
The action of $r_s$ is given by 
\begin{align*}
    r_s.\mu:=\mu - \langle \alpha_s^\vee, \mu \rangle \alpha_s
\end{align*}
for $s \in S$ and $\mu\in X^*(\wtil{H})$. 

For $w\in W(\wtil{G})$, write the length of $w$ as $l(w)$. Let $w_0\in W(\wtil{G})$ be the longest element of $W(\wtil{G})$, and set $s_G:=\overline{w_0}^2 \in N_{\widetilde{G}}(\widetilde{H})$. It turns out that $s_G\in Z(\widetilde{G})$, and $s_G^2=1$ (cf.~\cite[\S 2]{FG03}). We define an involution $S\to S, s\mapsto s^{\ast}$ by 
\begin{align*}
\alpha_{s^{\ast}}=-w_0\alpha_s. 
\end{align*}
We note that the Weyl group $W(G):=N_{G}(H)/H$ of $G$ is naturally isomorphic to the Weyl group $W(\widetilde{G})$ of $\widetilde{G}$, and we will frequently 
regard $\overline{w}$ as an element of $N_{G}(H)$ by abuse of notation. Remark that $s_G=\overline{w_0}^2=1$ in $G$. 

\paragraph{\textbf{Irreducible modules and matrix coefficients}}
Set $X^*(\widetilde{H})_+:=\sum_{s\in S}\mathbb{Z}_{\geq 0}\varpi_s\subset X^*(\widetilde{H})$ and $X^*(H)_+:=X^*(H)\cap X^*(\widetilde{H})_{+}$. For $\lambda\in X^*(H)_+$, let $V(\lambda)$ be the rational irreducible $G$-module of highest weight $\lambda$. 
A fixed highest weight vector of $V(\lambda)$ is denoted by $v_{\lambda}$. Set 
\begin{align*}
v_{w\lambda}:=\overline{w}.v_{\lambda}
\end{align*}
for $w\in W(G)$. There exists a unique non-degenerate symmetric bilinear form $(\ ,\ )_{\lambda}\colon V(\lambda)\times V(\lambda)\to \bA^1$ satisfying
	\begin{align*}
	(v_{\lambda}, v_{\lambda})_{\lambda}=1,\quad (g.v, v')_{\lambda}=(v, g^{\mathsf{T}}.v')_{\lambda}
	\end{align*}
	for $v, v'\in V(\lambda)$ and $g\in G$. For $v\in V(\lambda)$, we set 
\begin{align}
v^{\vee}:=(v'\mapsto (v, v')_{\lambda})\in V(\lambda)^{\ast}, \quad 
f_{w\lambda}:=v_{w\lambda}^{\vee}. \label{eq:vee}
\end{align}
Note that $(v_{w\lambda}, v_{w\lambda})_{\lambda}=1$ for all $w\in W(G)$. 

For a $G$-module $V$, the dual space $V^{\ast}$ is considered as a (left) $G$-module by 
\[
\langle g.f, v\rangle:=\langle f, g^{\mathsf{T}}.v\rangle
\]
for $g\in G$, $f\in V^{\ast}$ and $v\in V$. Note that, under this convention, the correspondence $v\mapsto v^{\vee}$ for $v\in V(\lambda)$ gives a $G$-module isomorphism $V(\lambda)\to V(\lambda)^{\ast}$ for $\lambda\in X^*(H)_+$. For $f\in V^{\ast}$ and $v\in V$, define the element $c_{f, v}^V\in \cO(G)$ by  
\begin{align}
g\mapsto \langle f, g.v\rangle\label{eq:mat_coeff}
\end{align}
for $g\in G$. An element of this form is called a \emph{matrix coefficient}. For $\lambda\in X^*(H)_+$, we simply write $c_{f, v}^{\lambda}:=c_{f, v}^{V(\lambda)}$. Moreover, for $w,w'\in W(G)$, the matrix coefficient
\begin{align}
\Delta_{w\lambda,w'\lambda}:=c_{f_{w\nu},v_{w'\lambda}}^{\lambda}.    \label{eq:minor}
\end{align}
is called a \emph{generalized minor}.

\paragraph{\textbf{The $\ast$-involutions}}
We conclude this subsection by recalling an involution on $G$ associated with a certain Dynkin diagram automorphism (cf. \cite[(2)]{GS16}). 
\begin{lem}\label{l:Dynkininv}
	Let $\ast: G\to G, g\mapsto g^{\ast}$ be a group automorphism defined by 
	\begin{align*}
	g\mapsto \overline{w}_0(g^{-1})^{\mathsf{T}}\overline{w}_0^{-1}.
	\end{align*}
Then $(g^{\ast})^{\ast}=g$ for all $g\in G$, and $x_s(t)^{\ast}= x_{s^{\ast}}(t)$, $y_s(t)^{\ast}= y_{s^{\ast}}(t)$ for $s\in S$. 
\end{lem}
For a proof, see \cite[Lemma 5.3]{IIO19}.

\subsection{The configuration space $\Conf_k \P_G$}\label{subsec:config}

Let $G$ be an adjoint group. Here we introduce the configuration space $\Conf_k \P_G$ based on \cite{GS19}, which models the moduli space $\P_{G,\Pi}$ for a $k$-gon $\Pi$.

\begin{dfn}
The homogeneous spaces $\A_G:=G/U^+$ and $\B_G:=G/B^+$ are called the \emph{principal affine space} and the \emph{flag variety}, respectively. An element of $\A_G$ (resp. $\B_G$) is called a \emph{decorated flag} (resp. \emph{flag}). We have a canonical projection $\pi: \A_G \to \B_G$.
\end{dfn}
The principal affine space can be identified with the moduli space of pairs $(U,\psi)$, where $U \subset G$ is a maximal unipotent subgroup and $\psi: U \to \mathbb{A}^1$ is a non-degenerate character. See \cite[Section 1.1.1]{GS14} for a detailed discussion. The basepoint of $\A_G$ is denoted by $[U^+]$. 
The flag variety $\B_G$ will be identified with the set of connected maximal solvable subgroups of $G$ via $g.B^+\mapsto gB^+g^{-1}$. 

The Cartan subgroup $H$ acts on $\A_G$ from the right by $g.[U^+].h:=gh.[U^+]$ for $g \in G$ and $h \in H$, which makes the projection $\pi: \A_G \to \B_G$ a principal $H$-bundle.



A pair $(B_1,B_2)$ of flags is said to be \emph{generic} if there exists $g \in G$ such that
\[
g.(B_1,B_2) :=(g.B_1,g.B_2) =(B^+,B^-).
\]
Using the Bruhat decomposition $G = \bigcup_{w \in W(G)} U^+H \overline{w} U^+$, it can be verified that the $G$-orbit of any pair $(A_1,A_2) \in \A_G\times \A_G$ contains a point of the form $(h.[U^+], \overline{w}.[U^+])$ for unique $h \in H$ and $w \in W(G)$. 
Then the parameters 
\[
h(A_1,A_2):=h\quad \mbox{and} \quad w(A_1,A_2):=w
\]
are called the \emph{$h$-invariant} and the \emph{$w$-distance} of the pair $(A_1,A_2)$, respectively. Note that the $w$-distance only depends on the underlying pair $(\pi(A_1),\pi(A_2))$ of flags, and the pair is generic if and only if $w(A_1,A_2) = w_0$. The following lemma justifies the name ``$w$-distance'' and provides us a fundamental technique to define Goncharov--Shen coordinates.

\begin{lem}[{\cite[Lemma 2.3]{GS19}}]
Let $u, v \in W(G)$ be two elements such that $l(uv) = l(u) + l(v)$. Then the followings hold.
\begin{enumerate}
\item
If a pair $(B_1,B_2)$ of flags satisfies $w(B_1,B_2) = uv$, then there exists a unique flag $B'$ such that 
\[
w(B_1,B') = u, \quad w(B',B_2) = v.
\]
\item
Conversely, if we have $w(B_1,B') = u$ and $w(B',B_2) = v$, then $w(B_1,B_2) = uv$.
\end{enumerate}
\end{lem}

\begin{cor}\label{c:flagchain}
Let $(B_l,B_r)$ be a pair of flags with $w(B_l,B_r) = w$. Every reduced word $\bs = (s_1,\dots, s_p)$ of $w$ gives rise to a unique chain of flags $B_l= B_0,B_1, \dots, B_p=B_r$ such that $w(B_{k-1},B_k) = r_{s_k}$.
\end{cor}
Next we define an enhanced configuration space by adding extra data called \emph{pinnings}.

\begin{dfn}[pinnings]
A \emph{pinning} is a pair $p=(\widehat{B}_1,B_2) \in \A_G \times \B_G$ of a decorated flag and a flag such that the underlying pair $(B_1,B_2) \in \B_G \times \B_G$ is generic, where $B_1:=\pi(\widehat{B}_1)$. We say that $p$ is a pinning over $(B_1,B_2)$. 
\end{dfn}
An important feature is that the set $\P_G$ of pinnings is a principal $G$-space, and in particular $\P_G$ is an affine variety. 
In this paper, we fix the basepoint to be $p_\std:=([U^+], B^-)$, so that any pinning can be writen as $g.p_\std$ for a unique $g \in G$. 
The right $H$-action of $\A_G$ induces a right $H$-action on $\P_G$, which is given by $(g.p_\std).h=gh.p_\std$ for $g \in G$ and $h \in H$. 
Each fiber of the projection 
\begin{align}\label{eq:proj_pinning}
    (\pi_+,\pi_-): \P_G \to \B_G \times \B_G, \quad p=(\widehat{B}_1,B_2) \mapsto (B_1,B_2)
\end{align}
is a principal $H$-space. 

For $p=g.p_\std$, we define the \emph{opposite pinning} to be $p^\ast:=g\vw.p_\std$. We have $(g.p_\std.h)^\ast=g.p_\std^\ast.w_0(h)$ for $g \in G$ and $h \in H$. 

\begin{rem}
We have the following equivalent descriptions of a pinning. See \cite{GS19} for details.
\begin{enumerate}
\item
A pair $p=(\widehat{B}_1,\widehat{B}_2) \in \A_G \times \A_G$ of decorated flags such that $h(\widehat{B}_1,\widehat{B}_2)=e$ and the underlying pair of flags is generic (\emph{i.e.}, $w(\widehat{B}_1,\widehat{B}_2)=w_0$). The opposite pinning is given by $p^*=(\widehat{B}_2,\widehat{B}_1)$.
\item
A data $p=(B,B^{\op};(\xi^+_s(t))_{s \in S}, (\xi^-_s(t))_{s \in S})$, where $(B,B^{\op})$ is a pair of opposite Borel subgroups of $G$ and $(\xi^+_s(t))_s$, $(\xi^-_s(t))_s$ are one-parameter subgroups determined by a fundamental system for the root data with respect to the maximal torus $B \cap B^{\op}$. The opposite pinning is given by $p^*=(B^{\op},B;(\xi^-_s(-t))_{s \in S}, (\xi^+_s(-t))_{s \in S})$.

\end{enumerate}
\end{rem}

For $k \in \Z_{\geq 2}$, we consider the configuration space
\[
\Conf_k \P_G:= [G\backslash \{ (B_1,\dots, B_k;p_{12},\dots,p_{k-1,k},p_{k,1})\}],
\]
where $B_i \in \B_G$, and $p_{i,i+1}$ is a pinning over $(B_i,B_{i+1})$ for cyclic indices $i \in \Z_k$. Here we use the notation for a quotient stack. See \cref{sec:stacks}. By \cref{lem:geometric_quotient}, $\Conf_k \P_G$ is in fact a geometric quotient, whose points are $G$-orbits of the data $(B_1,\dots, B_k;p_{12},\dots,p_{k-1,k},p_{k,1})$.   

We will sometimes write an element of $\Conf_k \P_G$ (\emph{i.e.} a $G$-orbit) as $[p_{12},\dots,p_{k-1,k},p_{k,1}]$, since the remaining data of flags can be read off from it via projections. However, the reader is reminded that the tuples of pinnings must satisfy the constraints $\pi_-(p_{i-1,i})=\pi_+(p_{i,i+1})$ for $i \in \Z_k$.


\section{Wilson lines on the moduli space \texorpdfstring{$\P_{G,\Sigma}$}{}}\label{sec:monodromy}
In this section, we first recall the definition of the moduli space $\P_{G,\Sigma}$ for a marked surface $\Sigma$.
We give an explicit description of the structure of $\P_{G,\Sigma}$ as a quotient stack
as an algebraic basis for the arguments in the subsequent sections. 
Then we introduce the Wilson line and Wilson loop morphisms on the stack $\P_{G,\Sigma}$ and study their basic properties. 
Finally we give their decomposition formula for a given ideal triangulation (or an ideal cell decomposition) of $\Sigma$. 

\subsection{The moduli space $\P_{G,\Sigma}$}\label{subsec:moduli}

\paragraph{\textbf{Topological setting.}} 
A \emph{marked surface} $(\Sigma,\mathbb{M})$ consists of a (possibly disconnected) compact oriented surface $\Sigma$ and a fixed non-empty finite set $\mathbb{M} \subset \Sigma$ of \emph{marked points}. See \cref{fig:marked surface} for an example. 
\begin{itemize}
    \item A marked point is called a \emph{puncture} if it lies in the interior of $\Sigma$, and \emph{special point} if it lies on the boundary. Let $\mathbb{P}=\mathbb{P}(\Sigma)$ (resp. $\mathbb{S}=\mathbb{S}(\Sigma)$) denote the set of punctures (resp. special points), so that $\mathbb{M}=\mathbb{P} \sqcup \mathbb{S}$.
    \item We call a connected component of the set $\partial \Sigma \setminus \mathbb{S}$ a \emph{boundary interval}. Let $\mathbb{B}=\mathbb{B}(\Sigma)$ denote the set of boundary intervals. 
By convention, we endow each boundary interval with the orientation induced from $\partial \Sigma$. 
\end{itemize}

\begin{figure}[ht]
\scalebox{0.7}{
\begin{tikzpicture}

    \draw(0,0) circle[x radius=5cm, y radius=3cm];
    {\begin{scope}[yshift=-1cm]
        \draw(1,0) arc[start angle=0, end angle=-180,     radius=1cm];
        \draw(0.707,-0.707) to[out=135, in=45] (-0.707,-0.707);
    \end{scope}}
    {\begin{scope}[xshift=-2cm, yshift=0.5cm]
        \fill[gray!30](0,0) circle(0.8cm);
        \draw(0,0) circle(0.8cm);
        \draw[color=red,ultra thick](0:0.8) arc[start angle=0, end angle=120, radius=0.8] node[midway,above]{\scalebox{1.2}{$E$}};
        \foreach \i in {0,120,240}
	        \fill(\i:0.8) circle(4pt); 
	    \draw(0:0.8) node[right]{\scalebox{1.2}{$m_1$}};
	    \draw(120:0.8) node[above left]{\scalebox{1.2}{$m_2$}};
	    \draw(240:0.8) node[below left]{\scalebox{1.2}{$m_3$}};
    \end{scope}}
\begin{scope}[xshift=2cm, yshift=0.5cm]
    \draw(0,0) circle(4pt) coordinate(B) node[above=0.3em]{\scalebox{1.2}{$m_0$}};
\end{scope}
\end{tikzpicture}
}
    \caption{A marked surface $\Sigma$ with $\mathbb{P}=\{m_0\}$ and $\mathbb{S}=\{m_1,m_2,m_3\}$. A boundary interval $E \in \mathbb{B}$ is shown in red.}
    \label{fig:marked surface}
\end{figure}

Let $\Sigma^\ast:=\Sigma \setminus \mathbb{P}$. 
We always assume the following conditions:
\begin{enumerate}
\item Each boundary component has at least one marked point.
\item $n(\Sigma):=-2\chi(\Sigma^*)+|\mathbb{S}|>0$.
\end{enumerate}
These conditions ensure that the marked surface $\Sigma$ has an ideal triangulation with $n(\Sigma)$ triangles, which is the isotopy class $\Delta$ of a triangulation of $\Sigma$ by a collection of mutually disjoint simple arcs connecting marked points. 
Each boundary interval belongs to any ideal triangulation of $\Sigma$. 
Denote the set of triangles of $\Delta$ by $t(\Delta)$, and the set of edges by $e(\Delta)$. Let $e_{\mathrm{int}}(\Delta) \subset e(\Delta)$ be the subset of internal edges, so that $e(\Delta)=e_{\mathrm{int}}(\Delta) \sqcup \mathbb{B}$. 

In this paper, we only consider an ideal triangulation having no self-folded triangle (\emph{i.e.} a triangle one of its edges is a loop) for simplicity. Indeed, thanks to the condition $(2)$, our marked surface admits such an ideal triangulation. See, for instance, \cite{FST08} \footnote{Note that the number $n$ \emph{loc. cit.} is the number of \emph{interior} edges of an ideal triangulation.}. 
More generally, one can consider an \emph{ideal cell decomposition}: it is the isotopy class of a collection of mutually disjoint simple arcs connecting marked points such that each complementary region is a polygon.  

\bigskip

\paragraph{\textbf{Framed $G$-local systems with pinnings.}}

Recall that a $G$-local system on a manifold $M$ is a principal $G$-bundle over $M$ equipped with a flat connection. 

Let $\L$ be a $G$-local system on $\Sigma^\ast$. A \emph{framing} of $\L$ is a flat section $\beta$ of the associated bundle $\L_\B:=\L \times_G \B_G$ on a small neighborhood of $\mathbb{M}$. Notice that a choice of any element $B \in \mathbb{B}_G$ determines a flat section $\beta_m$ near $m \in \mathbb{S}$, while only a $G$-invariant element $B$ corresponds to a flat section $\beta_m$ near $m \in \mathbb{P}$ (defined over a circular domain). 

\begin{dfn}[Fock-Goncharov \cite{FG03}]
Let $\X_{G,\Sigma}$ denote the set of gauge-equivalence classes of framed $G$-local systems $(\L,\beta)$. 
\end{dfn}

A framing $\beta$ of $\L$ is said to be \emph{generic} if for each boundary interval $E=(m_E^+,m_E^-)$ with initial (resp. terminal) special point $m_E^+$ (resp. $m_E^-$), the associated pair $(\beta_E^+,\beta_E^-)$ is generic. Here $\beta_E^\pm$ is the section defined near $m_E^\pm$, and such a pair is said to be generic if the pair of flags obtained as the value at any point on $E$ is generic.  

Let $(\L,\beta)$ be a $G$-local system equipped with a generic framing $\beta$. A \emph{pinning} over $(\L,\beta)$ is a section $p$ of the associated bundle $\L_\P:=\L \times_G \P_G$ on the set $\partial \Sigma \setminus \mathbb{S}$ such that for each boundary interval $E \in \mathbb{B}$, the corresponding section $p_E$ is a pinning over $(\beta_E^+,\beta_E^-)$. Here the last sentence means that $p_E$ is projected to the pair $(\beta_E^+,\beta_E^-)$ via the bundle map 
\begin{align*}
    \L_\P|_E \xrightarrow{(\pi_+,\pi_-)} \L_\B|_E \times \L_\B|_E \to (\L_\B)_{m_E^+} \times (\L_\B)_{m_E^-},
\end{align*}
where the former map is induced by the projection \eqref{eq:proj_pinning}, and the latter is the evaluation at the points $m_E^\pm$. 
Since $\L_\P$ is a principal $G$-bundle, a pinning of $(\L,\beta)$ determines a trivialization of $\L$ near each boundary interval.
\begin{dfn}[Goncharov--Shen \cite{GS19}]
Let $\P_{G,\Sigma}$ denote the set of the gauge-equivalence classes $[\L,\beta;p]$ of the triples $(\L,\beta;p)$ as above.
\end{dfn}

If the marked surface $\Sigma$ has empty boundary, we have $\P_{G,\Sigma} \cong \X_{G,\Sigma}$. In general we have a map $\P_{G,\Sigma} \to \X_{G,\Sigma}$ forgetting pinnings, which turns out to be a dominant morphism. The image $\X_{G,\Sigma}^0$ consists of the $G$-local systems with generic framings. 
For each boundary interval $E$, we have a natural action $\alpha_E: \P_{G,\Sigma} \times H \to \P_{G,\Sigma}$ given by the rescaling of the pinning $p_E$. Here recall that the set of pinnings over a given pair of flags is a principal $H$-space. Thus the dominant map $\P_{G,\Sigma} \to \X_{G,\Sigma}$ coincides with the quotient by these actions. 

The following variant of the moduli space is also useful. Let $\Xi \subset \mathbb{B}$ be a subset. A framed $G$-local system is said to be \emph{$\Xi$-generic} if the pair of flags associated with any boundary interval in $\Xi$ is generic. Then we define the notion of \emph{$\Xi$-pinning} over a $\Xi$-generic framed $G$-local system, where we only assign pinnings to the boundary intervals in $\Xi$. 

\begin{dfn}\label{def:moduli_partial_pinnings}
Let $\P_{G,\Sigma;\Xi}$ denote the set of gauge-equivalence classes of the triples $(\L,\beta,p)$, where $(\L,\beta)$ is a $\Xi$-generic framed $G$-local system and $p$ is a $\Xi$-pinning. 
\end{dfn}
Obviously we have $\P_{G,\Sigma;\emptyset}=\X_{G,\Sigma}$ and $\P_{G,\Sigma;\mathbb{B}}=\P_{G,\Sigma}$.

\subsubsection{The moduli space $\P_{G,\Sigma}$ as a quotient stack.}\label{subsec:alg_str_stack}
For simplicity, consider a connected marked surface $\Sigma$. 
Fix a basepoint $x \in \Sigma^\ast$. 
A \emph{rigidified framed $G$-local system with pinnings} consists of a triple $(\L,\beta;p)$ together with a choice of $s \in \L_x$. The group $G$ acts on the isomorphism classes of rigidified framed $G$-local systems with pinnings $(\L,\beta,p;s)$ by fixing $(\L,\beta,p)$ and by $s \mapsto s.g$ for $g \in G$. 

A rigidified $G$-local system $(\L;s)$ determines its \emph{monodromy homomorphism} $\rho: \pi_1(\Sigma^\ast,x) \to G$. Given a based loop $\gamma$ at $x$, the element $\rho(\gamma) \in G$ gives the comparison of the fiber point $s \in \L_x$ with its parallel-transport along $\gamma$. 
It is a classical fact that the conjugacy class 
\begin{align*}
    [\rho] \in \Hom(\pi_1(\Sigma^\ast,x),G)/G 
\end{align*}
depends only on and uniquely determines the isomorphism class of $\L$. The quotient stack $\mathrm{Loc}_{G,\Sigma}:=[\Hom(\pi_1(\Sigma^\ast,x),G)/G]$ is called the (Betti) moduli stack of $G$-local systems on $\Sigma$.

In order to parametrize the isomorphisms classes of rigidified framed $G$-local systems, let us prepare some notations. 
See \cref{fig:choice of paths}. 
\begin{itemize}
    \item For each puncture $m \in \mathbb{P}$, let $\gamma_m \in \pi_1(\Sigma^\ast,x)$ denote a based loop encircling $m$. 
    \item Enumerate the connected components of $\partial\Sigma^\ast$ as $\partial_1,\dots,\partial_b$, and let $\delta_k \in \pi_1(\Sigma^\ast,x)$ be a based loop freely homotopic to $\partial_k$ and following its orientation for $k=1,\dots,b$. 
    \item For $k=1,\dots,b$, choose a distinguished marked point $m_k$ on the boundary component $\partial_k$. Let $E_1^{(k)},\dots,E_{N_k}^{(k)}$ be the boundary intervals on $\partial_k$ in this clockwise ordering so that $m_k$ is the initial marked point of $E_1^{(k)}$.
    \item Take a path $\epsilon_1^{(k)}=\epsilon_{E_1^{(k)}}$ from $x$ to a point on the boundary interval $E_1^{(k)}$ for $k=1,\dots,b$, and let $\epsilon_j^{(k)}=\epsilon_{E_j^{(k)}}$ be a path from $x$ to $E_{j}^{(k)}$ such that the concatenation $\epsilon_{j,j-1}^{(k)}:=(\epsilon_{j}^{(k)})^{-1}\ast \epsilon_{j-1}^{(k)}$ is based homotopic to a boundary arc which contains exactly one marked point, the initial vertex of $E_j^{(k)}$ for $j=2,\dots,N_k$. 
\end{itemize}
In the pictures, the location of distinguished marked points is indicated by dashed lines. We will use the notation $\epsilon_{E,E'}:=\epsilon_E^{-1}\ast \epsilon_{E'}$ for two boundary intervals $E \neq E'$. 

\begin{figure}[h]
\begin{tikzpicture}

\draw(2.5,0) ellipse (3.8 and 2);
\fill[gray!30](0,0) circle(0.4cm);

\draw(0,0) circle(0.4cm);
	\draw(150:0.9) node{$\partial_k$};
\draw(5,0) circle(2pt) node[above]{$m$};
\draw[dashed] (120:0.4) -- (120:1);;
\fill(0.4,0) circle(2pt);
\fill(120:0.4) circle(2pt);
\fill(240:0.4) circle(2pt);
\pic at (2.5,-0.3) {handle};

\node[red] at (60:0.6) {};
	\draw (50:2)  coordinate(x) node[above]{$x$};
	
{\color{myblue}
	\draw (x) ..controls (50:1) and (50:0.8).. (60:0.4) node[above]{\scalebox{0.8}{$\epsilon_1^{(k)}$}};
	\draw (x) ..controls (-50:1) and (-50:0.6).. (-60:0.4) node[below]{\scalebox{0.8}{$\epsilon_2^{(k)}$}};
	\draw (x) ..controls (-80:3) and (-180:1).. (180:0.4) node[left]{\scalebox{0.8}{$\epsilon_3^{(k)}$}};
}
\draw[red] (x) ..controls ++(0,-4) and (-100:1.1).. (-120:1.1) ..controls (-140:1.1) and (180:2.3).. (x);
\draw[red] (-60:1.7) node{$\delta_k$};

\begin{scope}[xshift=5cm]
\draw[red] (x) ..controls ++(4,0) and (0:1.2).. (-60:0.6) ..controls (-120:1.2) and (130:1).. (x);
\draw[red] (-90:0.8) node{$\gamma_m$};
\end{scope}
\draw[red,fill] (50:2) circle(1.5pt) coordinate(x) node[above]{$x$};
\end{tikzpicture}
    \caption{Some of the curves in the defining data of the atlas of $\P_{G,\Sigma}$.}
    \label{fig:choice of paths}
\end{figure}
Notice that given a rigidified framed $G$-local system, 
\begin{itemize}
    \item the flat section of $\L_\B$ around $m \in \mathbb{P}$ gives an element $\lambda_m \in \B_G$ via the parallel-transport along the path from $m$ to $x$ surrounded by the loop $\gamma_m$ and the isomorphism $\B_G \xrightarrow{\sim} \L_x$, $g.B \mapsto s.g^{-1}$. 
    \item Similarly, the flat section of $\L_\P$ defined on a boundary interval $E$ gives an element $\phi_E \in\P_G$ via the parallel transport along the path $\epsilon_E$. 
\end{itemize}

Let $m_j^{(k)} \in \mathbb{M}$ denote the initial marked point of the boundary interval $E_j^{(k)}$. By convention, $m_1^{(k)}=m_k$ for $k=1,\dots,b$. 
Then we have:

\begin{lem}\label{lem:rigidified-P}
There is a bijection between the set of isomorphism classes of the rigidified framed $G$-local systems with pinnings on $(\Sigma,x)$ and the set of points of the complex quasi-projective variety $P_{G,\Sigma}^{(\{m_k\})}$ consisting of triples $(\rho,\lambda,\phi) \in \Hom(\pi_1(\Sigma^\ast,x),G)\times (\B_G)^{\mathbb{M}} \times (\P_G)^{\mathbb{B}}$ which satisfy the following conditions:
\begin{itemize}
    \item $\rho(\gamma_m).\lambda_{m} = \lambda_{m}$ for all $m \in \mathbb{P}$.
    \item $\pi_+(\phi_{E_j^{(k)}}) = \lambda_{m_j^{(k)}}$ and $\pi_-(\phi_{E_j^{(k)}}) = \lambda_{m_{j+1}^{(k)}}$ for $k=1,\dots,b$ and $j=1,\dots,N_k$, where we set $\lambda_{m_{N_k+1}^{(k)}}:=\rho(\delta_k).\lambda_{m_{1}^{(k)}}$.
\end{itemize}
Moreover the correspondence is $G$-equivariant, where the group $G$ acts on $P_{G,\Sigma}^{(\{m_k\})}$ by $(\rho,\lambda,\phi) \mapsto (g\rho g^{-1},g.\lambda,g.\phi)$ for $g \in G$. 
\end{lem}

\begin{dfn}\label{d:Betti_P-space}
The \emph{moduli stack of framed $G$-local systems with pinnings on $\Sigma$} is defined to be the quotient stack
\begin{align*}
    \P_{G,\Sigma}:=[P_{G,\Sigma}^{(\{m_k\})}/G]
\end{align*}
over $\C$. 
\end{dfn}

\begin{lem}\label{lem:boundary_shift}
Suppose we replace the distinguished marked points as $m'_k:=m_2^{(k)}$ by a shift on a boundary component $\partial_k$, and $m'_{k'}:=m_{k'}$ for $k'\neq k$. Then we have a $G$-equivariant isomorphism
\begin{align*}
    P_{G,\Sigma}^{(\{m_k\})} \xrightarrow{\sim} P_{G,\Sigma}^{(\{m'_k\})}
\end{align*}
given by sending 
\begin{align*}
    (\lambda_1,\dots,\lambda_{N_k}) \mapsto (\lambda_2,\dots,\lambda_{N_k},\rho(\delta_k).\lambda_1), \quad (\phi_1,\dots,\phi_{N_k}) \mapsto (\phi_2,\dots,\phi_{N_k},\rho(\delta_k).\phi_1)
\end{align*}
and keeping the other data intact. Here $\lambda_j:=\lambda_{m_j^{(k)}}$ and $\phi_j:=\phi_{E_j^{(k)}}$ for $j=1,\dots,N_k$. 
\end{lem}

\begin{proof}
Follows from $\pi_+(\rho(\delta_k).\phi_1)=\rho(\delta_k).\pi_+(\phi_1)=\rho(\delta_k).\lambda_1$ and $\pi_-(\rho(\delta_k).\phi_1)=\rho(\delta_k).\pi_-(\phi_1)=\rho(\delta_k)^2.\lambda_1=\rho(\delta_k).(\rho(\delta_k).\lambda_1)$.  
\end{proof}
Hence the quotient stack $\P_{G,\Sigma}$ is independent of the choice of distinguished marked points. 
When no confusion can occur, we simply write $P_{G,\Sigma}=P_{G,\Sigma}^{(\{m_k\})}$. 




\bigskip
\paragraph{\textbf{Partially generic case}}
For any subset $\Xi \subset \mathbb{B}$, the moduli stack of $\Xi$-generic framed $G$-local systems with $\Xi$-pinnings (recall \cref{def:moduli_partial_pinnings}) is similarly defined as 
\begin{align}\label{eq:Betti_partial_pinnings}
    \P_{G,\Sigma;\Xi}:=[P_{G,\Sigma;\Xi}^{(\{m_k\})}/G],
\end{align}
where the algebraic variety $P_{G,\Sigma;\Xi}^{(\{m_k\})}$ is obtained from $P_{G,\Sigma}^{(\{m_k\})}$ by forgetting the $\P_G$-factors corresponding to $\mathbb{B}\setminus \Xi$. Here some of the distinguished marked points may be redundant to obtain the atlas. 
For $\Xi' \subset \Xi$, we have an obvious dominant morphism $\P_{G,\Sigma;\Xi} \to \P_{G,\Sigma;\Xi'}$. 
When $\Xi \neq \emptyset$, the stack $\P_{G,\Sigma;\Xi}$ is still representable. 

\bigskip
\paragraph{\textbf{Disconnected case}}
When the marked surface $\Sigma$ has $N$ connected components, we consider a rigidification of a framed $G$-local system (with pinnings) on each connected component. Then the atlas $P_{G,\Sigma}$ is defined to be the direct product of those for the connected components, on which $G^N$ acts. The moduli stack $\P_{G,\Sigma}$ is defined as the quotient stack for this $G^N$-action.

\subsection{Gluing morphisms}\label{subsubsec:gluing}\label{subsec:alg_str_decomposition}
An advantage of considering the moduli space $\P_{G,\Sigma}$, rather than $\X_{G,\Sigma}$, is its nice property under the gluing procedure of marked surfaces. Let us first give the \lq\lq topological'' definition of the gluing morphism. An explicit description as a morphism of stacks is given soon below.

Let $\Sigma$ be a (possibly disconnected) marked surface, and choose two distinct boundary intervals $E_L$ and $E_R$. Identifying the intervals $E_L$ and $E_R$, we get a new marked surface $\Sigma'$. Let $\mathrm{pr}:\Sigma \to \Sigma'$ be the corresponding projection, where $E':=\mathrm{pr}(E_L)=\mathrm{pr}(E_R)$ is a simple arc on $\Sigma'$. 

On the level of moduli spaces, given $(\L,\beta;p)$, the pinning $p_{E_Z}$ assigned to the boundary interval $E_Z$ determines a trivialization of $\L$ near $E_Z$ for $Z \in \{L,R\}$, since $\P_G$ is a principal $G$-space. Then there is a unique isomorphism beween the restrictions of $(\L,\beta)$ on $\Sigma$ to neighborhoods of $E_L$ and $E_R$ which identify the pinnings $p_{E_L}$ and $p_{E_R}^*$. In this way we get a framed $G$-local system with pinnings $q_{E_L,E_R}(\L,\beta;p)$ on $\Sigma'$. Note that the result is unchanged under the transformation $\alpha_{E_L,E_R^
{\mathrm{op}}}(h): (p_{E_L},p_{E_R}) \mapsto (p_{E_L}.h, p_{E_R}.w_0(h))$ for some $h \in H$. We get the \emph{gluing morphism} \cite[Lemma 2.12]{GS19}
\begin{align}\label{eq:gluing_morphism}
    q_{E_L,E_R}: \P_{G,\Sigma} \to \P_{G,\Sigma'},
\end{align}
which induces an open embedding $\overline{q}_{E_L,E_R}: \P_{G,\Sigma}/H \to \P_{G,\Sigma'}$, where $H$ acts on $\P_{G,\Sigma}$ via $\alpha_{E_L,E_R^
{\mathrm{op}}}$. 

The gluing operation is clearly associative. 
In particular, given an ideal triangulation $\Delta$ of $\Sigma$, we can decompose the moduli space $\P_{G,\Sigma}$ into a product of the configuration spaces $\Conf_3 \P_G$ as follows. 
Let $H^\Delta$ denote the product of copies of Cartan subgroups $H$, one for each interior edge of $\Delta$. It acts on the product space $\widetilde{\P_{G,\Sigma}^\Delta}:=\prod_{T \in t(\Delta)} \P_{G,T}$ from the right via $\alpha_{E_L,E_R^
{\mathrm{op}}}$ for each glued pair $(E_L,E_R)$ of edges. 

\begin{thm}[{\cite[Theorem 2.13]{GS19}}]\label{t:GS decomposition}
Let $\Delta$ be an ideal triangulation of the marked surface $\Sigma$. Then we have the gluing morphism
\begin{align}\label{eq:decomposition_moduli}
    q_\Delta: \widetilde{\P_{G,\Sigma}^\Delta} \to \P_{G,\Sigma},
\end{align}
which induces an open embedding $\overline{q}_\Delta: [\widetilde{\P_{G,\Sigma}^\Delta}/ H^{\Delta}] \to \P_{G,\Sigma}$. 
\end{thm}
The image of $q_\Delta$ is denoted by $\P_{G,\Sigma}^\Delta \subset \P_{G,\Sigma}$, which consists of framed $G$-local systems with pinnings such that the pair of flags associated with each interior edge of $\Delta$ is generic. 


\smallskip
\paragraph{\textbf{Presentation of the gluing morphism}}
Let us give an explicit presentation $\widetilde{q}_{E_L,E_R}:P_{G,\Sigma}^{(\{m_k\})} \to P_{G,\Sigma'}^{(\{m'_k\})}$ of the gluing morphism \eqref{eq:gluing_morphism} for some atlases for later use.  For simplicity, we assume that the resulting marked surface $\Sigma'$ is connected. Choose a basepoint $x' \in \Sigma'$ on the arc $E'$, and write $\mathrm{pr}^{-1}(x')=\{x_L,x_R\}$ with $x_L \in E_L$ and $x_R \in E_R$. 
We also use the identifications $\mathbb{M}(\Sigma')=\mathbb{M}(\Sigma) \setminus \{m_{E_R}^\pm\}$ and $\mathbb{B}(\Sigma')=\mathbb{B}(\Sigma)\setminus \{E_L,E_R\}$. 

Then we distinguish the two cases: (1) $\Sigma$ has two connected components $\Sigma_L$ and $\Sigma_R$ containing $E_L$ and $E_R$, respectively, or (2) $\Sigma$ is also connected. See \cref{fig:multiplicativity,fig:multiplicativity_connected_case}. 
For example, the gluing morphism in \cref{t:GS decomposition} is obtained by succesively applying the gluings of the first type. 

\begin{description}
\item[(1) The disconnected case]
In this case, we have the van Kampen isomorphism $\pi_1({\Sigma'}^\ast,x') \cong \pi_1(\Sigma_L^\ast,x_L) \ast \pi_1(\Sigma_R^\ast,x_R)$. 
For simplicity, we assume that the distinguished marked points on the boundary components containing $E_L$ and $E_R$ are identified under the gluing. 
The other cases are then obtained by composing the coordinate transformations given in \cref{lem:boundary_shift}. 

Given $(\rho_L,\lambda_L,\phi_L;\rho_R,\lambda_R,\phi_R) \in P_{G,\Sigma}=P_{G,\Sigma_L}\times P_{G,\Sigma_R}$, let us write $(\phi_L)_{E_L}=g_L.p_\std$ and $(\phi_R)_{E_R}=g_R.p_\std$. 
Define $(\rho',\lambda',\phi') \in P_{G,\Sigma'}$ by
\begin{align*}
    \rho'(\gamma)&:=\begin{cases}
        \rho_L(\gamma) & \mbox{if $\gamma \in \pi_1(\Sigma_L^\ast,x_L)$},\\
        \Ad_{g_L\vw g_R^{-1}}(\rho_R(\gamma))& \mbox{if $\gamma \in \pi_1(\Sigma_R^\ast,x_R)$}, 
    \end{cases} \\
    \lambda'_m&:=\begin{cases}
        (\lambda_L)_m & \mbox{if $m \in \mathbb{M}(\Sigma_L)$},\\
        g_L\vw g_R^{-1}.(\lambda_R)_m & \mbox{if $m \in \mathbb{M}(\Sigma_R)\setminus \{m_{E_R}^\pm\}$},
    \end{cases} \\
    \phi'_E&:=\begin{cases}
        (\phi_L)_E & \mbox{if $E \in \mathbb{B}(\Sigma_L) \setminus \{E_L\}$},\\
        g_L\vw g_R^{-1}.(\phi_R)_E & \mbox{if $E \in \mathbb{B}(\Sigma_R)\setminus \{E_R\}$}.
    \end{cases} 
\end{align*}
Then $\rho'$ is extended as a group homomorphism $\pi_1({\Sigma'}^\ast,x') \to G$. In terms of the rigidified framed $G$-local systems, we have chosen the rigidification data given on $\Sigma_L$ as that for $\Sigma'$. 

\item[(2) The connected case]
In this case, $\pi_1({\Sigma'}^\ast,x')$ is generated by the image of the inclusion $\pi_1(\Sigma^\ast,x_L)\to \pi_1({\Sigma'}^\ast,x')$ induced by $\Sigma\simeq\Sigma'\setminus E' \to \Sigma'$ and 
the based loop $\alpha:=\mathrm{pr}(\epsilon_{E_L}^{-1}\ast \epsilon_{E_R})$. 
When $E_L$ and $E_R$ belong to distinct boundary components, we assume that their distinguished marked points are identified under the gluing. 
The other cases are then obtained by composing the coordinate transformations given in \cref{lem:boundary_shift}. 

Given $(\rho,\lambda,\phi) \in P_{G,\Sigma}$, let us write $\phi_{E_L}=g_L.p_\std$ and $\phi_{E_R}=g_R.p_\std$. 
Define $(\rho',\lambda',\phi') \in P_{G,\Sigma'}$ by
\begin{align*}
    \rho'(\gamma)&:=\begin{cases}
        \rho(\gamma) & \mbox{if $\gamma \in \pi_1(\Sigma^\ast,x_L)$},\\
        g_L\vw g_R^{-1}& \mbox{if $\gamma=\alpha$}, 
    \end{cases} \\
    \lambda'&:=\lambda|_{\mathbb{M}(\Sigma')},\\
    \phi'&:=\phi|_{\mathbb{M}(\Sigma')}.
\end{align*}
Then $\rho$ is extended as a group homomorphism $\pi_1({\Sigma'}^\ast,x') \to G$. 
\end{description}

\begin{lem}
The morphism $\widetilde{q}_{E_L,E_R}:P_{G,\Sigma}^{(\{m_k\})} \to P_{G,\Sigma'}^{(\{m'_k\})}$ given above descends to a morphism $\P_{G,\Sigma} \to \P_{G,\Sigma'}$, which agrees with the topological definition of the gluing morphism \eqref{eq:gluing_morphism}. 
\end{lem}

\begin{proof}
The morphism $\widetilde{q}_{E_L,E_R}$ is clearly $G$-equivariant, and hence descends to a morphism $\P_{G,\Sigma} \to \P_{G,\Sigma'}$. In order to see that it agrees with the topological definition, observe the following. 

In the disconnected case, consider the action of the element $g_L\vw g_R^{-1}$ on the triple $(\rho_R,\lambda_R,\phi_R)$ by rescaling the rigidification. After such rescaling, the pinning assigned to the boundary interval $E_R$ gives $g_L\vw g_R^{-1}.(\phi_R)_{E_R}=g_L\vw.p_\std$, which is the opposite of the pinning $(\phi_L)_{E_L}$. Thus the gluing condition matches with the one explained in the beginning of this subsection.

In the connected case, note that the monodromy $\rho'(\alpha)$ is the unique element such that $\rho'(\alpha).(g_R.p_\std)^\ast=g_L.p_\std$, which is given by $\rho'(\alpha)=g_L\vw g_R^{-1}$. The remaining data is unchanged, since the system of curves $\{\delta_k,\epsilon_E\}$ on $\Sigma$ is naturally inherited by $\Sigma'$. 
\end{proof}

\subsection{Wilson lines and Wilson loops}\label{subsec:monodromy}
We are going to introduce the \emph{Wilson line} morphisms (\emph{Wilson lines} for short) on $\P_{G,\Sigma}$ for a marked surface with non-empty boundary. 

Let $E_\inn$, $E_\out \in \mathbb{B}$ be two boundary intervals, and $c$ a path from $E_\inn$ to $E_\out$ in $\Sigma^\ast$. 
First we give a topological definition. 
For a point $[\L,\beta;p] \in \P_{G,\Sigma}$, choose a local trivialization of $\L$ on a vicinity of $E_\inn$ so that the flat section $p_{E_\inn}$ of $\L_\P$ associated to $E_\inn$ corresponds to $p_\std$. This local trivialization can be extended along the path $c$ until it reaches $E_\out$. Then the flat section $p_{E_\out}$ determines a pinning under this trivialization, which is written as $g.p_\std^*$ for a unique element $g=g_c([\L,\beta;p]) \in G$. It depends only on the homotopy class $[c]$ of $c$ relative to $E_\inn$ and $E_\out$: we call such a homotopy class $[c]$ an \emph{arc class} on the marked surface $\Sigma$, and write $[c]: E_\inn \to E_\out$ in the sequel. 
Then we have a map
\begin{align*}
    g_{[c]}: \P_{G,\Sigma} \to G,
\end{align*}
which we call the \emph{Wilson line} along the arc class $[c]: E_\inn \to E_\out$. 

The Wilson lines can be defined as morphisms $\P_{G,\Sigma} \to G$, as follows. Fix a basepoint $x \in \Sigma^\ast$ and the collection $\{m_k\}$ of distinguished marked points, and consider the corresponding presentation $\P_{G,\Sigma}=[P_{G,\Sigma}^{(\{m_k\})}/G]$. 
Notice that any arc class $[c]:E_\inn \to E_\out$ can be decomposed as $[c]=[\epsilon_{E_\inn}]^{-1}\ast [\gamma_x] \ast [\epsilon_{E_\out}]$, where $\gamma_x$ is a based loop at $x$.

\begin{prop}\label{prop:presentation_Wilson_line}
Define a $G$-equivariant morphism $\widetilde{g}_{[c]}:P_{G,\Sigma}^{(\{m_k\})} \to G \times G$ of varieties by 
\begin{align*}
    \widetilde{g}_{[c]}(\rho,\lambda,\phi):=(g_{E_\inn}^{-1}\rho(\gamma_x) g_{E_\out}\vw,g_{E_\inn}),
\end{align*}
where we write $\phi_{E}=g_E.p_\std$ for a unique element $g_E\in G$ for a boundary interval $E$. Here $G$ acts on the first factor of $G \times G$ trivially, and by the left multiplication on the second. Then the induced morphism $g_{[c]}:\P_{G,\Sigma} \to G$ of varieties agrees with the topological definition given above.
\end{prop}

\begin{proof}
Observe that $[G \times G/G] =G$ by \cref{lem:quotient_example}. Therefore $\widetilde{g}_{[c]}$ induces a morphism $g_{[c]}:\P_{G,\Sigma} \to G$ of Artin stacks by \cref{lem:equivariant_morphism}. The action on the target amounts to forgetting the second factor.

Consider the rigidified framed $G$-local system with pinnings $(\L,\beta,p;s)$ corresponding to a given point of $P_{G,\Sigma}$. The rigidification $s$ determines a local trivialization of $\L$ near $x$, and the section $p_{E_\inn}$  (resp. $p_{E_\out}$) gives the element $g_{E_\inn}.p_\std$ (resp. $g_{E_\out}.p_\std$) via the parallel-transport along the path $\epsilon_{E_\inn}$ (resp. $\epsilon_{E_\out}$) under this local trivialization. Moreover, notice that the section $p_{E_\out}$ gives $\rho(\gamma_x)g_{E_\out}.p_\std$ via the parallel-transport along the path $\gamma_x \ast \epsilon_{E_\out}$. 
The local trivialization of $\L$ near $E_\inn$ so that $p_{E_\inn}$ corresponds to $p_\std$ is given by the rigidification $s.g_{E_\inn}^{-1}$. The latter trivialization can be continued along the path $\epsilon_{E_\inn}^{-1} \ast \gamma_x \ast \epsilon_{E_\out}$, for which the section $p_{E_\out}$ gives $g_{E_\inn}^{-1}\rho(\gamma_x)g_{E_\out}.p_\std=g_{E_\inn}^{-1}\rho(\gamma_x)g_{E_\out}\vw.p^\ast_\std$ as desired. 
\end{proof}

\begin{rem}\label{rem:partial_pinnings}
\begin{enumerate}
    \item As the proof indicates, the second component of the presentation morphism $\widetilde{g}_{[c]}$ is introduced in order to make it $G$-equivariant, rather than $G$-invariant (note that when $G$ acts on some variety $X$ trivially, then $[X/G] \neq X$ as a stack). As we shall see in \cref{rem:geometric_quotient}, $\P_{G,\Sigma}$ is a variety if $\Sigma$ has no punctures, hence $g_{[c]}: \P_{G,\Sigma} \to G$ can be directly defined. 
    \item The Wilson line $g_{[c]}$ along an arc class $[c]:E_\inn \to E_\out$ can be defined as a morphism
\begin{align*}
    g_{[c]}:\P_{G,\Sigma;\{E_\inn,E_\out\}} \to G,
\end{align*}
since it does not refer to the pinnings other than $p_{E_\inn}$ and $p_{E_\out}$. 
Here recall \eqref{eq:Betti_partial_pinnings}.
\end{enumerate}
\end{rem}

The Wilson lines $g_{[c]}$ have the following multiplicative property with respect to the gluing of marked surfaces. Let $\Sigma$ be a (possibly disconnected) marked surface, and consider two arc classes $[c_1]:E_1 \to E_2$ and $[c_2]:E'_2 \to E_3$. Let $\Sigma'$ be the marked surface obtained from $\Sigma$ by gluing the boundary intervals $E_2$ and $E'_2$. Then the concatenation of the arc classes $[c_1]$ and $[c_2]$ give an arc class $[c]: E_1 \to E_3$ on $\Sigma'$. See \cref{fig:multiplicativity,fig:multiplicativity_connected_case}. Recall the gluing morphism $q_{E_2,E'_2}: \P_{G,\Sigma} \to \P_{G,\Sigma'}$. 

\begin{prop}\label{prop:multiplicativity}
We have $q_{E_2,E'_2}^* g_{[c]} = g_{[c_1]}\cdot g_{[c_2]}$.
\end{prop}

\begin{proof}
Recall the presentation of the gluing morphism given in \cref{subsubsec:gluing}. We may assume that $\Sigma'$ is connected without loss of generality, and divide the argument into the two cases.

\begin{description}
\item[(1) Disconnected case]
In this case, we have
\begin{align*}
    \widetilde{q}_{E_2,E'_2}^* g_{[c]} &= g_{E_1}^{-1} (g_{E_2}\vw g_{E'_2}^{-1}\cdot g_{E_3})\vw \\
    &= g_{E_1}^{-1}g_{E_2}\vw \cdot g_{E'_2}^{-1}g_{E_3}\vw \\
    &=g_{[c_1]}\cdot g_{[c_2]}.
\end{align*}
\item[(2) Connected case]
In this case, consider the based loop $\alpha:=\mathrm{pr}(\epsilon_{E_2}^{-1}\ast \epsilon_{E'_2}) \in \pi_1(\Sigma',x')$. See \cref{fig:multiplicativity_connected_case}.  
Then we have $g_{[c]}=g_{E_1}^{-1}\rho(\alpha)g_{E_3}\vw$ and $\widetilde{q}_{E_2,E'_2}^*\rho(\alpha)=g_{E_2}\vw g_{E'_2}^{-1}$. Hence
\begin{align*}
    \widetilde{q}_{E_2,E'_2}^* g_{[c]} &= g_{E_1}^{-1} (g_{E_2}\vw g_{E'_2}^{-1}) g_{E_3}\vw \\
    &= g_{E_1}^{-1}g_{E_2}\vw \cdot g_{E'_2}^{-1}g_{E_3}\vw \\
    &=g_{[c_1]}\cdot g_{[c_2]}.
\end{align*}
\end{description}
\end{proof}


%
\begin{figure}[h]
\begin{tikzpicture}[scale=0.8]

\draw(0,0) ellipse (0.3 and 1);
\fill(0,1) circle(2pt);
\fill(0,-1) circle(2pt);
\node at (0,1.4) {$E_1$};
\draw(5,1) arc (90:-90:0.3 and 1);
\draw[dashed](5,1) arc (90:270:0.3 and 1);
\fill(5+0.3,0) circle(2pt);
\node at (5,1.4) {$E_2$};
\draw(0,1) ..controls (0.5,1) and (1.2,1.5).. (2.5,1.5) ..controls (3.8,1.5) and (4.5,1).. (5,1);
\draw(0,-1) ..controls (0.5,-1) and (1.2,-1.5).. (2.5,-1.5) ..controls (3.8,-1.5)  and (4.5,-1).. (5,-1);
\pic at (2.5,0) {handle};
\draw (2,0.7) circle(2pt);
\draw (3,0.7) circle(2pt);
\draw[->,red,thick](0.3*0.8661,-0.5) ..controls (0.3*0.8661+4,-1.5)  and (5+0.3*0.8661-1,0.5).. node[midway,below=0.2em]{$[c_1]$} (5+0.3*0.8661,0.5);
\begin{scope}[xshift=6cm]
\draw(0,0) ellipse (0.3 and 1);
\fill(0.3,0) circle(2pt);
\node at (0,1.4) {$E'_2$};
\draw(3,2.5) arc (90:-90:0.3 and 1);
\draw[dashed] (3,2.5) arc (90:270:0.3 and 1);
\fill(3,2.5) circle(2pt);
\fill(3,0.5) circle(2pt);
\node at (3,2.9) {$E_3$};
\draw(3,-0.5) arc (90:-90:0.3 and 1);
\draw[dashed] (3,-0.5) arc (90:270:0.3 and 1);
\fill(3+0.3,-1.5) circle(2pt);
\draw(0,1) ..controls (0.5,1) and (2.5,2.5).. (3,2.5);
\draw(0,-1) ..controls (0.5,-1) and (2.5,-2.5).. (3,-2.5);
\draw(3,0.5) ..controls (2,0.3) and (2,-0.3).. (3,-0.5);
\draw[->,red,thick](0.3*0.8661,0.5) ..controls (0.3*0.8661+1,0.5)  and (2+0.3,1.5).. node[midway,below=0.2em]{$[c_2]$} (3+0.3,1.5);
\end{scope}

\draw [ultra thick,-{Classical TikZ Rightarrow[length=4pt]},decorate,decoration={snake,amplitude=1.5pt,pre length=2pt,post length=3pt}] (10.5,0) -- node[midway,above]{Glue $E_2$ and $E'_2$} (12.5,0);
\begin{scope}[xshift=14cm]
\draw(0,0) ellipse (0.3 and 1);
\fill(0,1) circle(2pt);
\fill(0,-1) circle(2pt);
\node at (0,1.4) {$E_1$};
\draw[dashed](5,1) arc (90:-90:0.3 and 1);
\draw[dashed](5,1) arc (90:270:0.3 and 1);
\draw(5+0.3,0) circle(2pt);
\draw(0,1) ..controls (0.5,1) and (1.2,1.5).. (2.5,1.5) ..controls (3.8,1.5) and (4.5,1).. (5,1);
\draw(0,-1) ..controls (0.5,-1) and (1.2,-1.5).. (2.5,-1.5)  ..controls (3.8,-1.5)  and (4.5,-1).. (5,-1);
\pic at (2.5,0) {handle};
\draw (2,0.7) circle(2pt);
\draw (3,0.7) circle(2pt);
\draw[red,thick](0.3*0.8661,-0.5) ..controls (0.3*0.8661+4,-1.5)  and (5+0.3*0.8661-1,0.5).. node[midway,below=0.5em]{$[c]$} (5+0.3*0.8661,0.5);
{\begin{scope}[xshift=5cm]
\draw(3,2.5) arc (90:-90:0.3 and 1);
\draw[dashed] (3,2.5) arc (90:270:0.3 and 1);
\fill(3,2.5) circle(2pt);
\fill(3,0.5) circle(2pt);
\node at (3,2.9) {$E_3$};
\draw(3,-0.5) arc (90:-90:0.3 and 1);
\draw[dashed] (3,-0.5) arc (90:270:0.3 and 1);
\fill(3+0.3,-1.5) circle(2pt);
\draw(0,1) ..controls (0.5,1) and (2.5,2.5).. (3,2.5);
\draw(0,-1) ..controls (0.5,-1) and (2.5,-2.5).. (3,-2.5);
\draw(3,0.5) ..controls (2,0.3) and (2,-0.3).. (3,-0.5);
\draw[->,red,thick](0.3*0.8661,0.5) ..controls (0.3*0.8661+1,0.5)  and (2+0.3,1.5).. (3+0.3,1.5);
\end{scope}}
\end{scope}
\end{tikzpicture}
    \caption{Multiplicativity of Wilson lines, disconnected case.}
    \label{fig:multiplicativity}
\end{figure}

\begin{figure}
\begin{tikzpicture}[scale=0.8]
\draw(0,0) node[left]{$E_1$} arc(180:360:1 and 0.3);
\draw[dashed](0,0) arc(180:0:1 and 0.3);
\draw(0,-1) node[left]{$E_2$} arc(180:-180:1 and 0.3);
\draw(0,0) ..controls (0,1) and (-2,1).. (-2,-0.5) ..controls (-2,-2) and (0,-2).. (0,-1);
\draw(2,0) ..controls (2,4) and (-4,4).. (-4,-0.5) ..controls (-4,-5) and (2,-5).. (2,-1);
\fill(1,-0.3) circle(2pt);
\fill(1,-1.3) circle(2pt);

\draw (1,0)++(-0.7,-0.3*0.7) coordinate(x);
\draw (1,-1)++(-0.7,-0.3*0.7) coordinate(y);
\draw[red,thick,->] (x) ..controls (0.3,2) and (-3,2).. (-3,-0.5) node[left]{$[c]$} ..controls (-3,-3) and (0.3,-3).. (y);

\draw [ultra thick,-{Classical TikZ Rightarrow[length=4pt]},decorate,decoration={snake,amplitude=1.5pt,pre length=2pt,post length=3pt}] (3,-0.5) -- node[midway,above]{Glue $E_1$ and $E_2$} (5,-0.5);

\begin{scope}[xshift=10cm]
\draw[dashed](0,-0.5) arc(180:-180:1 and 0.3);
\draw(0,-0.5) ..controls (0,1) and (-2,1).. (-2,-0.5) ..controls (-2,-2) and (0,-2).. (0,-0.5);
\draw(2,-0.5) ..controls (2,4) and (-4,4).. (-4,-0.5) ..controls (-4,-5) and (2,-5).. (2,-0.5);
\draw(1,-0.8) circle(2pt);

\draw (1,-0.5)++(-0.7,0) coordinate(x);
\draw[red,thick,->] (x) ..controls (0.3,2) and (-3,2).. (-3,-0.5) node[left]{$|\gamma|$} ..controls (-3,-3) and (0.3,-3).. (x);
\end{scope}
\end{tikzpicture}
    \caption{Multiplicativity of Wilson lines, connected case.}
    \label{fig:multiplicativity_connected_case}
\end{figure}
As a variant of the above argument, we can describe the monodromy homomorphism in terms of the Wilson lines. Given a based loop $\gamma \in \pi_1(\Sigma^\ast,x)$, we have the evaluation morphism
\begin{align*}
    \mathrm{ev}_\gamma: P_{G,\Sigma} \to G, 
\end{align*}
which takes the monodromy $\rho(\gamma)$ along $\gamma$ for a given rigidified $G$-local system with pinnings.
Note that the set of conjugacy classes in $\pi_1(\Sigma^\ast,x)$ is identified with the set $\hat{\pi}(\Sigma^\ast):=[S^1,\Sigma^\ast]$ of free loops on $\Sigma^\ast$. 
Let $\pi_1(\Sigma^\ast,x) \to \hat{\pi}(\Sigma^\ast)$, $[\gamma] \mapsto |\gamma|$ denote the canonical projection. Then the morphism $\mathrm{ev}_\gamma$ descends to a morphism $\rho_{|\gamma|}:\P_{G,\Sigma} \to [G/\Ad G]$ that fits into the commutative diagram
\begin{equation*}
    \begin{tikzcd}
    P_{G,\Sigma} \ar[d] \ar[r,"\mathrm{ev}_\gamma"] & G \ar[d]\\
    \P_{G,\Sigma} \ar[r,"\rho_{|\gamma|}"'] & {[G/\Ad G]},
    \end{tikzcd}
\end{equation*}
which only depends on the free loop $|\gamma| \in \hat{\pi}(\Sigma^\ast)$. 
We call $\rho_{|\gamma|}$ the \emph{Wilson loop} along $|\gamma|$. 


\begin{prop}\label{prop:Wilson line-loop}
Let $\Sigma$ be a marked surface, $[c]:E_1 \to E_2$ an arc class. Let $\Sigma'$ be the marked surface obtained from $\Sigma$ by gluing the boundary intervals $E_1$ and $E_2$, and $\gamma \in \pi_1({\Sigma'}^\ast,x')$ be the based loop arising from $[c]$. Here we choose the basepoint $x' \in {\Sigma'}^\ast$ on the edge arising from $E_1$ and $E_2$.  
Then we have the following commutative diagram of morphisms of stacks:
\begin{equation}\label{eq:Wilson line-loop}
    \begin{tikzcd}
    \P_{G,\Sigma} \ar[r,"g_{[c]}"] \ar[d,"q_{E_1,E_2}"'] & G \ar[d] \\
    \P_{G,\Sigma'} \ar[r,"\rho_{|\gamma|}"'] & {[G/\Ad G]},
    \end{tikzcd}
\end{equation}
where the right vertical morphism is the canonical projection.
\end{prop}

\begin{proof}
From the presentation of the gluing morphism given in \cref{subsubsec:gluing}, we have
\begin{align*}
    \widetilde{q}_{E_1,E_2}^* ev_{\gamma}= g_{E_1}\vw g_{E_2}^{-1} = \Ad_{g_{E_2}}(\widetilde{g}^{(1)}_{[c]}),
\end{align*}
where $\widetilde{g}^{(1)}_{[c]}$ denotes the first component of the morphism $\widetilde{g}_{[c]}$ defined in \cref{prop:presentation_Wilson_line}. 
In other words, we have the commutative diagram
\begin{equation*}
    \begin{tikzcd}
    P_{G,\Sigma} \ar[r,"\widetilde{g}_{[c]}"] \ar[d,"\widetilde{q}_{E_1,E_2}"'] & G \times G \ar[d,"\Ad_{g_{E_2}}\circ \mathrm{pr}_1"] \ar[r,"\mathrm{pr}_1"] & G \ar[d]\\
    P_{G,\Sigma'} \ar[r,"\mathrm{ev}_{\gamma}"'] & G \ar[r] & {[G/\Ad G]}
    \end{tikzcd}
\end{equation*}
and thus we get the desired assertion.
\end{proof}


\begin{rem}[Twisted Wilson lines]\label{rem:twisted Wilson line}
Let $\Pi_1(\Sigma,\mathbb{B})$ be the groupoid whose objects are boundary intervals of $\Sigma$ and morphisms are arc classes with the composition rule given by concatenations. Then each point $[\L,\beta;p] \in \P_{G,\Sigma}$ defines a functor \begin{align*}
    g^{\mathrm{tw}}_\bullet ([\L,\beta;p]): \Pi_1(\Sigma,\mathbb{B}) \to G, \quad [c] \mapsto g_{[c]}^{\mathrm{tw}}([\L,\beta;p]),
\end{align*}
where $g_{[c]}^{\mathrm{tw}}([\L,\beta;p]):=g_{[c]}([\L,\beta;p])\vw$ denotes the \emph{twisted} Wilson line and the group $G$ is naturally regarded as a groupoid with one object. Note that an automorphism $[c]$ of a boundary interval $E$ in $\Pi_1(\Sigma,\mathbb{B})$ can be represented by a loop $\gamma$ based at $x \in E$, and the conjugacy class of the twisted Wilson line $g_{[c]}^{\mathrm{tw}}$ coincides with the Wilson loop $\rho_{|\gamma|}$. 
Although the Wilson lines themselves do not induce such a functor, we will see that they possess a nice positivity property as well as the multiplicativity for gluings explained above. 
\end{rem}

\subsection{Generation of $\cO(\P_{G,\Sigma})$ by matrix coefficients of (twisted) Wilson lines}\label{subsec:generate}
We are going to obtain an explicit presentation of the function algebra $\cO(\P_{G,\Sigma})$ by using the (twisted) Wilson lines when $\Sigma$ has no punctures. 
In the contrary case $\partial\Sigma=\emptyset$, the description of the function algebra $\cO(\P_{G,\Sigma})=\cO(\X_{G,\Sigma})$ as an $\cO(\mathrm{Loc}_{G,\Sigma})$-module has been already obtained in \cite[Section 12.5]{FG03}. 


Choose a generating set $\mathsf{S}=\{(\alpha_i,\beta_i)_{i=1}^g,(\gamma_m)_{m \in \mathbb{P}},(\delta_k)_{k=1}^b\}$ of $\pi_1(\Sigma^\ast,x)$, a collection $\{m_k\}$ of distinguished marked points, and paths $\epsilon_j^{(k)}=\epsilon_{E_j^{(k)}}$ as in \cref{subsec:alg_str_stack}. 
Then we get the atlas $P_{G,\Sigma}^{(\{m_k\})}$, which consists of triples $(\rho,\lambda,\phi)$ satisfying certain conditions described in \cref{lem:rigidified-P}. 

Assume that $\Sigma$ has no punctures, and choose one boundary interval, say, $E_0:=E_1^{(1)}$. Write $\phi_E=g_E.p_\std$, $g_E \in G$ for $E \in \mathbb{B}$ and set $g_0:=g_{E_0}$. 
Then we have a $G$-equivariant morphism
\begin{align*}
    \widetilde{\Phi}'_{E_0}: P^{(\{m_k\})}_{G,\Sigma} \to G^{2g+b} \times  G^{\mathbb{B}\setminus \{E_0\}} \times G
\end{align*}
which sends $(\rho,\lambda,\phi)$ to the tuple 
\begin{align*}
    \left((\rho_{E_0}(\alpha_i),
    \rho_{E_0}(\beta_i),
    \rho_{E_0}(\delta_k))_{\substack{i=1,\dots,g \\ k=1,\dots,b}},\ 
    (g_{E_0,E})_{E \neq E_0}, g_0\right).
\end{align*}
Here $\rho_{E_0}(\gamma):=g_0^{-1}\rho(\gamma)g_0$ is the monodromy along $\gamma$ for the local trivialization given by the pinning $\phi_{E_0}$ for $\gamma \in \mathsf{S}$, and $g_{E_0,E}:=g_0^{-1}g_E\vw$ is the Wilson line along the arc class $[\epsilon_{E_0,E}]=[\epsilon_{E_0}^{-1} \ast \epsilon_E]: E_0 \to E$. The group $G$ acts on the last factor of $G^{2g+b} \times  G^{\mathbb{B}\setminus \{E_0\}} \times G$ by left multiplication, and trivially on the other factors.
Then it descends to an embedding 
\begin{align}\label{eq:Betti_embedding}
    \Phi'_{E_0}:\P_{G,\Sigma}=[P_{G,\Sigma}^{(\{m_k\})}/G] \to G^{2g+b}\times G^{\mathbb{B}\setminus \{E_0\}}
\end{align}
of Artin stacks. Note from \cref{rem:twisted Wilson line} that $\rho_{E_0}(\gamma)$ for $\gamma \in \mathsf{S}$ can be regarded as the twisted Wilson line along the based loop $\gamma_x$ at $x \in E_0$. We can take their matrix coefficients, not only their traces.

For each $k=1,\dots,b$, consider the paths $\epsilon_{j,j-1}^{(k)}$ which are based-homotopic to boundary arcs which contain exactly one marked point $m_j^{(k)}$, for $j=1,\dots,N_k$. Here the indices are read modulo $N_k$.  

\begin{lem}\label{lem:Wilson_line_Borel}
The Wilson line $g_{j,j-1}^{(k)}$ along the arc class $[\epsilon_{j,j-1}^{(k)}]:E_j^{(k)} \to E_{j-1}^{(k)}$ takes values in $B^+$. 
\end{lem}

\begin{proof}
Let $E_1$ (resp. $E_2$) denote the boundary interval having $m_j^{(k)}$ as its initial (resp. terminal) point. Let $\phi_{i}=g_i.p_\std$ be the pinning assigned to $E_i$ for $i=1,2$. It can happen that $E_1=E_2$: in that case, we have $g_2=\rho(\delta_k)^{-1}g_1$. 
From the condition $\pi_+(\phi_1)=\lambda_{m_j^{(k)}}=\pi_-(\phi_2)$, we get $g_1.B^+=g_2.B^-=g_2\vw.B^+$. Hence $g_{j,j-1}^{(k)}=g_1^{-1}g_2\vw \in B^+$.
\end{proof}
Since $\epsilon_{E_0,E_j^{(k)}}=\epsilon_{E_0,E_{N_k}^{(k)}}\ast \epsilon_{N_k,N_k-1}^{(k)}\ast \cdots \ast \epsilon_{j+1,j}^{(k)}$, we have
\begin{align}\label{eq:pinning_reconstruction}
    g_{E_0,E_j^{(k)}} = (g^{(k)}\vw) (g_{N_k,N_k-1}^{(k)}\vw) \cdots (g_{j+2,j+1}^{(k)}\vw) g_{j+1,j}^{(k)}
\end{align}
for $k=1,\dots,b$ and $j=2,\dots,N_k-1$. Here $g^{(k)}:=g_{E_0,E_{N_k}^{(k)}}$ denotes the Wilson line along the arc class $[\epsilon_{E_0,E_{N_k}^{(k)}}]:E_0 \to E_{N_k}^{(k)}$ for $k=2,\dots,b$, and $g^{(1)}:=1$. See \cref{fig:Wilson_reparametrization}. 
\begin{figure}
\begin{tikzpicture}
\filldraw[fill=gray!30,draw=black,thick](0,0) circle(1cm); 
\foreach \i in {0,120,240}
{
\fill(\i:1) circle(2pt);
\draw[red,thick,->] (\i-20:1) ..controls (\i-20:2) and (\i+20:2).. (\i+20:1);
}
\node at (60:1.5) {\scalebox{0.85}{$E_1^{(1)}$}};
\draw(60:1.5)++(0,0.5) node{\rotatebox{90}{$=$}};
\draw(60:1.5)++(0,0.9) node{\scalebox{0.85}{$E_0$}};
\node at (-60:1.5) {\scalebox{0.85}{$E_2^{(1)}$}};
\node at (180:1.5) {\scalebox{0.85}{$E_3^{(1)}$}};
\node[right,red] at (0:1.6) {\scalebox{0.85}{$g_{2,1}^{(1)}$}};
\node[below,red] at (-120:1.6) {\scalebox{0.85}{$g_{3,2}^{(1)}$}};
\node[above,red] at (120:1.6) {\scalebox{0.85}{$g_{4,3}^{(1)}$}};

\draw[red,thick,->] (45:1) ..controls (45:2) and  (6,1.2).. (7,0);
\node[red] at (4,1.4) {$g^{(k)}$};
\begin{scope}[xshift=8cm]
\filldraw[fill=gray!30,draw=black,thick](0,0) circle(1cm); 
\foreach \i in {45,135,215,315}
{
\fill(\i:1) circle(2pt);
\draw[red,thick,->] (\i-20:1) ..controls (\i-20:2) and (\i+20:2).. (\i+20:1);
}
\node at (90:1.5) {\scalebox{0.85}{$E_1^{(k)}$}};
\node at (0:1.5) {\scalebox{0.85}{$E_2^{(k)}$}};
\node at (-90:1.5) {\scalebox{0.85}{$E_3^{(k)}$}};
\node at (-180:1.5) {\scalebox{0.85}{$E_4^{(k)}$}};
\node[right,red] at (45:1.65) {\scalebox{0.85}{$g_{2,1}^{(1)}$}};
\node[right,red] at (-45:1.65) {\scalebox{0.85}{$g_{3,2}^{(1)}$}};
\node[left,red] at (-135:1.65) {\scalebox{0.85}{$g_{4,3}^{(1)}$}};
\node[left,red] at (135:1.65) {\scalebox{0.85}{$g_{1,4}^{(1)}$}};
\end{scope}
\pic at (4,-0.5) {handle};
\end{tikzpicture}
    \caption{Some Wilson lines.}
    \label{fig:Wilson_reparametrization}
\end{figure}
Therefore the embedding \eqref{eq:Betti_embedding} gives rise to another embedding
\begin{align*}
    \Phi_{E_0}: \P_{G,\Sigma} \to G^{2g+b} \times G^{b-1} \times (B^+)^{\sum_{k=1}^b N_k},
\end{align*}
which is represented by a morphism $\widetilde{\Phi}_{E_0}$ that sends a $G$-orbit of $(\rho,\lambda,\phi)$ to the tuple 
\begin{align}\label{eq:Betti_modified_embedding}
    \left((\rho_{E_0}(\alpha_i),
    \rho_{E_0}(\beta_i),
    \rho_{E_0}(\delta_k))_{\substack{i=1,\dots,g \\ k=1,\dots,b}},\ 
    (g^{(k)})_{k=2,\dots,b},\ 
    (g_{j,j-1}^{(k)})_{\substack{k=1,\dots,b \\j=1,\dots,N_k}}, g_0
    \right).
\end{align}

\begin{thm}\label{thm:generation_by_Wilson_lines}
The image of the embedding $\Phi_{E_0}$ is the closed subvariety which consists of the tuples \eqref{eq:Betti_modified_embedding} satisfying the following conditions:
\begin{itemize}
    \item Monodromy relation: $\prod_{i=1}^g [\rho_{E_0}(\alpha_i),\rho_{E_0}(\beta_i)]\cdot \prod_{k=1}^b \rho_{E_0}(\delta_k) = 1$; \vspace{0.5em}
    \item Boundary relation: $(g_{N_k,N_k-1}^{(k)}\vw) \cdots (g_{2,1}^{(k)}\vw) (g_{1,N_k}^{(k)}\vw) = \rho_{E_0}(\delta_k)^{-1}$ for $k=1,\dots,b$.
\end{itemize}
\end{thm}

\begin{proof}
It is clear from the previous discussion and the multiplicative property of the twisted Wilson lines $g_{[c]}^\mathrm{tw}=g_{[c]}\vw$ that the image of $\Phi_{E_0}$ satisfies the conditions. Conversely, given a tuple \eqref{eq:Betti_modified_embedding} which satisfies the conditions, we can reconstruct the $G$-orbit of a triple $(\rho,\lambda,\phi) \in P_{G,\Sigma}^{(\{m_k\})}$, as follows. We first get the monodromy homomorphism $\rho_{E_0}$ normalized at the boundary interval $E_0$, and the pinning $\phi_{E_0}=p_\std$. The other pinnings are given by $\phi_E:=(g_{E_0,E}.p_\std)^\ast$ for $E \in \mathbb{B} \setminus \{E_0\}$, where $g_{E_0,E} \in G$ is determined by the formula \eqref{eq:pinning_reconstruction}. The collection $\lambda$ of the underlying flags is given by 
\begin{align*}
    \lambda_{m_j^{(k)}}
    &:= \pi_+(\phi_{E_j^{(k)}})  \\
    &= (g^{(k)}\vw) (g_{N_k,N_k-1}^{(k)}\vw) \cdots (g_{j+2,j+1}^{(k)}\vw) g_{j+1,j}^{(k)}.B^- \\
    &= (g^{(k)}\vw) (g_{N_k,N_k-1}^{(k)}\vw) \cdots (g_{j+2,j+1}^{(k)}\vw) (g_{j+1,j}^{(k)}\vw)g_{j,j-1}^{(k)}.B^+ \\
    &= \pi_-(\phi_{E_{j-1}^{(k)}}).
\end{align*}
Each consecutive pair of flags is generic, since
\begin{align*}
    [\lambda_{m_j^{(k)}}, \lambda_{m_{j-1}^{(k)}}] 
    =[B^-,\vw g_{j,j-1}^{(k)}.B^-] = [B^+,B^-]
\end{align*}
by $g_{j,j-1}^{(k)} \in B^+$. Thus we get $(\rho,\lambda,\phi) \in P_{G,\Sigma}^{(\{m_k\})}$ normalized as $\phi_{E_0}=p_\std$. 
\end{proof}

\begin{cor}
When $\Sigma$ has no punctures, we have
\begin{align*}
    \cO(\P_{G,\Sigma}) \cong (\cO(G)^{\otimes (2g+b)} \otimes \cO(G)^{\otimes (b-1)} \otimes \cO(B^+)^{\otimes \sum_{k=1}^b N_k})/\sqrt{\mathscr{I}_{G,\Sigma}},
\end{align*}
where $\sqrt{\mathscr{I}_{G,\Sigma}}$ is the radical of the ideal $\mathscr{I}_{G,\Sigma}$ generated by the two relations described in \cref{thm:generation_by_Wilson_lines}. In particular, the function algebra $\cO(\P_{G,\Sigma})$ is generated by the matrix coefficients of (twisted) Wilson lines. 
\end{cor}

\begin{rem}\label{rem:geometric_quotient}
When $\Sigma$ has no punctures, one can see that the variety $P_{G,\Sigma}^{(\{m_k\})}$ is affine via the embedding $\widetilde{\Phi}_{E_0}$, on which $G$ acts freely. Hence the moduli space $\P_{G,\Sigma}^{(\{m_k\})}$ is representable by the geometric quotient $P_{G,\Sigma}^{(\{m_k\})}/G$ by \cref{lem:geometric_quotient}. 
\end{rem}

\begin{rem}
When $\Sigma$ has punctures and non-empty boundary, we have
\begin{align*}
    \cO(\P_{G,\Sigma}) \cong (\cO(G)^{\otimes (2g+b)} \otimes  \cO(\widehat{G})^{\otimes p} \otimes \cO(G)^{\otimes (b-1)} \otimes \cO(B^+)^{\otimes \sum_{k=1}^b N_k})/\sqrt{\mathscr{I}_{G,\Sigma}},
\end{align*}
where $p$ is the number of punctures, $\widehat{G}:=\{(g,B) \in G \times \B_G \mid g \in B\}$ denotes the \emph{Grothendieck--Springer resolution}, to which the pair $(\rho_{E_0}(\gamma_{m}),\lambda_{m})$ for $m \in \mathbb{P}$ belongs. The ideal $\mathscr{I}_{G,\Sigma}$ is generated by the monodromy relation 
\begin{align*}
    \prod_{i=1}^g [\rho_{E_0}(\alpha_i),\rho_{E_0}(\beta_i)]\cdot \prod_{j=1}^p \rho_{E_0}(\gamma_{m_j})\cdot \prod_{k=1}^b \rho_{E_0}(\delta_k) = 1
\end{align*}
and the same boundary relation, where we fixed an appropriate enumeration $\mathbb{P}=\{m_1,\dots,m_p\}$. In particular, $\cO(\widehat{G})$ contains some functions not coming from the matrix coefficients of the twisted Wilson line $\rho_{E_0}(\gamma_m)$: see \cite[Section 12.5]{FG03} and \cite[Section 4.2]{Shen20} for a detail. 
\end{rem}

\begin{ex}
When $\Sigma=T$ is a triangle, we have
\begin{align*}
    \cO(\P_{G,T}) \cong \cO(B^+)^{\otimes 3}/\sqrt{\mathscr{I}_{G,T}},
\end{align*}
where $\mathscr{I}_{G,T}=\langle\ g_{3,2} \vw g_{2,1} \vw g_{1,3} \vw =1 \ \rangle$. 
As we have seen in \cref{c:LR}, the images of the Wilson lines $g_{j,j-1}$ are in fact contained in the double Bruhat cell $B^+_\ast$. This can be seen as $g_{3,2}=(\vw g_{1,3}^{-1} \vw) \vw (\vw g_{2,1}^{-1}\vw) \in B^- \vw B^-$ and the cyclic symmetry. 
\end{ex}

\begin{figure}
    \begin{tikzpicture}[scale=0.9]
    \draw(0,0) -- (3,0) -- (3,3) -- (0,3) --cycle;
    \foreach \i in {1,2,3,4}
    \draw(90*\i:1.8)++(1.5,1.5) node{$E_\i$};
    \draw[red,thick,->] (3,1) to[out=180,in=90] (2,0);
    \draw[red,thick,->] (1,0) to[out=90,in=0] (0,1);
    \draw[red,thick,->] (0,2) to[out=0,in=-90] (1,3);
    \draw[red,thick,->] (2,3) to[out=-90,in=180] (3,2);
    \draw[red,thick,->] (1.5,3) --node[midway,left]{$g_{1,3}$} (1.5,0);
    \node at (1.5,-1) {$g_{1,3} \in G^{w_0,w_0}$};
    
    \begin{scope}[xshift=5cm]
    \draw(3,3) -- (0,3) -- (0,0) -- (3,0);
    \draw[dashed] (3,0) -- (3,3);
    \foreach \i in {1,2,3,4}
    \draw(90*\i:1.8)++(1.5,1.5) node{$E_\i$};
    \draw[red,thick,->] (1,0) to[out=90,in=0] (0,1);
    \draw[red,thick,->] (0,2) to[out=0,in=-90] (1,3);
    \draw[red,thick,->] (1.5,3) --node[midway,left]{$g_{1,3}$} (1.5,0);
    \node at (1.5,-1) {$g_{1,3} \in B^+\vw B^+$};
    \end{scope}
    
    \begin{scope}[xshift=10cm]
    \draw(0,0) -- (3,0);
    \draw(0,3) -- (3,3);
    \draw[dashed] (0,0) -- (0,3);
    \draw[dashed] (3,0) -- (3,3);
    \foreach \i in {1,2,3,4}
    \draw(90*\i:1.8)++(1.5,1.5) node{$E_\i$};
    \draw[red,thick,->] (1.5,3) --node[midway,left]{$g_{1,3}$} (1.5,0);
    \node at (1.5,-1) {$g_{1,3} \in G$};
    \end{scope}
    \end{tikzpicture}
    \caption{Some Wilson lines on the moduli space $\P_{G,Q;\Xi}$. The boundary intervals not belonging to $\Xi$ are shown by dashed lines.}
    \label{fig:square_Wilson_lines}
\end{figure}

\begin{ex}
When $\Sigma=Q$ is a quadrilateral, we have 
\begin{align*}
    \cO(\P_{G,Q}) \cong \cO(B^+)^{\otimes 4}/\sqrt{\mathscr{I}_{G,Q}},
\end{align*}
where $\mathscr{I}_{G,Q}=\langle\ g_{4,3}\vw g_{3,2} \vw g_{2,1} \vw g_{1,4} \vw =1 \ \rangle$. 
Let us consider the Wilson line $g_{1,3}$ on $\P_{G,Q}$ shown in the left of \cref{fig:square_Wilson_lines}. 
Letting $g_{3,1}:=\vw g_{1,3}^{-1} \vw = (g_{1,3}^\ast)^\mathsf{T}$, we get the relations
\begin{align*}
    &g_{1,3}\vw g_{3,2}\vw g_{2,1}\vw =1, \\
    &g_{3,1}\vw g_{1,4}\vw g_{4,3}\vw =1.
\end{align*}
Then similarly to the previous example, we get $g_{1,3} \in B^-\vw B^-$ and $g_{3,1} \in B^-\vw B^-$, the latter being equivalent to $g_{1,3} \in B^+\vw B^+$. Thus we get $g_{1,3} \in G^{w_0,w_0}=B^+\vw B^+ \cap B^-\vw B^-$. 
\end{ex}

\begin{ex}[partially generic cases]
The restriction $g_{1,3} \in G^{w_0,w_0}$ in the previous example can be viewed as a consequence of the genericity condition for the consecutive flags. Let us consider the moduli space $\P_{G,Q:\Xi}$ with $\Xi=\{E_1,E_2,E_3\}$ and $\Xi=\{E_1,E_3\}$, which are schematically shown in the middle and in the right in \cref{fig:square_Wilson_lines}, respectively. In these cases we have less Wilson lines and less restrictions for the values of $g_{1,3}$: it can take an arbitrary value in $B^+\vw B^+$ and in $G$, respectively. 

In particular, we have $\P_{G,Q;\{E_1,E_3\}} \cong G$. The configuration of flags is parametrized as $[B^+,B^-,g_{1,3}.B^+,g_{1,3}.B^-]$. Our discussion shows that the image of the dominant morphism $\P_{G,Q} \to \P_{G,Q;\{E_1,E_3\}}\cong G$ is exactly the double Bruhat cell $G^{w_0,w_0}$.
\end{ex}

\subsection{Decomposition formula for Wilson lines}\label{subsec:standard_config}
We are going to give a certain explicit representative of an element of $\Conf_3 \P_G$. 
We also introduce certain functions on these spaces called the \emph{basic Wilson lines}, which will be the local building blocks for the general Wilson line morphisms. The standard configuration makes it apparent that the values of the basic Wilson lines are upper or lower triangular.

\subsubsection{The moduli space for a triangle} 
Let $T$ be a triangle, \emph{i.e.}, a disk with three special points. 
The choice $m$ of a distinguished marked point determines an atlas $P_{G,T}^{(m)}$ of the moduli space $\P_{G,T}$. Note that the representable stack $[P_{G,T}^{(m)}/G]$ is nothing but the configuration space $\Conf_3 \P_G$. 
In other words, the moduli space $\P_{G,T}$ can be identified with the configuration space $\Conf_3 \P_G$ in three ways, depending on the choice of a distinguished marked point. Let us denote the isomorphism by
\begin{align}\label{eq:moduli_triangle}
    f_{m}:\P_{G,T} \xrightarrow{\sim} \Conf_3 \P_G.
\end{align}
For later use, it is useful to indicate the distinguished marked point by the symbol $\ast$ on the corresponding corner in figures, which we call the \emph{dot}. See \cref{fig:polygon} for an example. 

In topological terms, the isomorphisms $f_{m}$ are described as follows. Let us denote the three marked points of $T$ by $m,m',m''$ in this counter-clockwise order. Let $E:=[m,m'], E':=[m',m''], E'':=[m'',m]$ be the three boundary intervals. 
Given $[\L,\beta;p] \in \P_{G,T}$, the local system $\L$ is trivial. We have three sections $\beta_m,\beta_{m'},\beta_{m''}$ of $\L_\B$ defined near each marked point, and three sections $p_E,p_{E'},p_{E''}$ of $\L_\P$ defined on each boundary interval. Then 
\begin{align}\label{eq:moduli_triangle_explicit}
    f_{m}:\ &\P_{G,T} \xrightarrow{\sim} \Conf_3 \P_G, \\
    \nonumber
    &[\L,\beta;p] \mapsto [\beta_m(x),\beta_{m'}(x),\beta_{m''}(x);p_{E}(x),p_{E'}(x),p_{E''}(x)].
\end{align}
Here we extend the domain of each section until a common point $x \in T$ via the parallel transport defined by $\L$. The following is a special case of \cref{lem:boundary_shift}.

\begin{lem}\label{l:cyclic shift}
The coordinate transformation $f_{m'}\circ f_{m}^{-1}$ is given by the cyclic shift
\begin{align*}
    \mathcal{S}_3: \Conf_3 \P_G \xrightarrow{\sim} \Conf_3 \P_G, \quad [p_{E},p_{E'},p_{E''}] \mapsto [p_{E'},p_{E''},p_{E}],
\end{align*}
which is an isomorphism.
\end{lem}
An explicit computation of the cyclic shift $\mathcal{S}_3$ in terms of the standard configuration is given in \cref{sec:positivity}. 

\subsubsection{The standard configuration and basic Wilson lines}
Now let us more look into the configuration space $\Conf_3 \P_G$, which models the moduli space $\P_{G,T}$ as we have seen just above. 
Let $U^\pm_*:= U^\pm \cap B^\mp \vw B^\mp$ denote the \emph{open unipotent cell}, which is an open subvariety of $U^\pm$. 
Let $\phi':U^+_* \xrightarrow{\sim} U^-_*$ be the unique isomorphism such that $\phi'(u_+).B^+=u_+.B^-$. See \cref{sec:twist map} for details. 

\begin{lem}[cf. {\cite[(8.7)]{FG03}}]\label{l:standardconfig}
We have an isomorphism
\begin{align*}
\wC_3 &: H \times H \times U_*^+ \xrightarrow{\sim} \Conf_3 \P_G, \\ &(h_1,h_2,u_+) \mapsto [B^+, B^-, u_+.B^-; p_\std, \phi'(u_+)h_1\vw.p_\std, u_+h_2\vw.p_\std]
\end{align*}
of varieties. 
\end{lem}
We call the representative in the right-hand side or the parametrization $\wC_3$ itself the \emph{standard configuration}.

\begin{proof}
Since $\wC_3$ is clearly a morphism of varieties, it suffices to prove that it is bijective. 
Let $(B_1,B_2,B_3; p_{12},p_{23},p_{31})$ be an arbitrary configuration. Using the genericity condition for the pairs $(B_1,B_2)$, $(B_2,B_3)$ and $(B_3,B_1)$, we can write $[B_1,B_2,B_3] = [B^+,B^-,u_+.B^-]$ for some $u_+ \in U_*^+$. Using an element of $B^+ \cap B^- =H$, we can further translate the configuration so that $p_{12}=p_\std$. Note that a representative of $(B_1,B_2,B_3; p_{12},p_{23},p_{31})$ satisfying these conditions is unique. 

Since $p_{31}$ is now a pinning over the pair $(u_+.B^-,B^+)$, there exists $h_2 \in H$ such that $p_{31} = u_+h_2\vw.p_\std$. Let us write $p_{23} = g.p_\std$ for some $g \in G$. Since $p_{23}$ is a pinning over the pair $(B^-,u_+.B^-)$, we have $g.B^+ = B^-$ and hence $g = b_-\vw$ for some $b_- \in B^-$. We also have $g.B^- = u_+.B^- = \phi'(u_+)\vw. B^-$, where the latter is the very definition of the map $\phi'$. 
It implies that $b_- =\phi'(u_+)h_1$ for some $h_1 \in H$. 
Thus we get the desired parametrization.
\end{proof}
Thus we get an induced isomorphism 
\begin{align*}
    \cO(\Conf_3 \P_G) \xrightarrow{\sim} \cO(H)^{\otimes 2} \otimes \cO(U^+_*)
\end{align*}
of the rings of regular functions. 
Note that we can represent a configuration $\wC \in \Conf_3 \P_G$ in the following two ways:
\begin{itemize}
    \item $\wC = [p_\std,\overline{p}_{23},\overline{p}_{31}]$, where the first component is normalized,
    \item $\wC = [\overline{p}'_{12},\overline{p}'_{23},p_\std]$, where the last component is normalized.
\end{itemize}
Such representatives are unique since the set of pinnings is a principal $G$-space.  

\begin{dfn}\label{d:basic-Conf3}
Define the elements $b_L=b_L(\wC)$, $b_R=b_R(\wC) \in G$ (\lq\lq left", \lq\lq right") by the condition 
\begin{align*}
    \overline{p}_{31} = (b_L.p_\std)^*, \quad \overline{p}'_{12} = (b_R.p_\std)^*.
\end{align*}
The resulting maps $b_L,b_R:\Conf_3 \P_G \to G$ are called the \emph{basic Wilson lines}. 
\end{dfn}
Note that we have $(b_R\vw)^{-1} = b_L\vw$, since 
\begin{align*}
    \wC 
    = [\overline{p}'_{12},\overline{p}'_{23},p_\std] 
    = [b_R\vw.p_\std,\overline{p}'_{23},p_\std] 
    = [p_\std, (b_R\vw)^{-1}.\overline{p}'_{23}, (b_R\vw)^{-1}.p_\std].
\end{align*}
We remark here that these functions already appeared in \cite[Section 6.2]{GS14}.
The following is a direct consequence of \cref{l:standardconfig}:

\begin{cor}\label{c:LR}
We have $b_L(\wC) \in B^+_*$ and $b_R(\wC) \in B^-_*$ for any configuration $\wC \in \Conf_3 \P_G$. The resulting maps
\begin{align}
    &b_L: \Conf_3 \P_G \to B^+_*, \quad \wC \mapsto b_L(\wC),\label{eq:phimap1} \\
    &b_R: \Conf_3 \P_G \to B^-_*, \quad \wC \mapsto b_R(\wC)\label{eq:phimap2} 
\end{align}
are morphisms of varieties, which are explicitly given by
\begin{align*}
    b_L(\wC_3(h_1,h_2,u_+)) &= u_+h_2 \in B^+_*, \\
    b_R(\wC_3(h_1,h_2,u_+)) &= \vw^{-1} (u_+h_2)^{-1}\vw=((u_+h_2)^{\ast})^{\mathsf{T}}  \in B^-_*.
\end{align*}
\end{cor}

\begin{rem}\label{rem:LR_Wilson_line}
The definition of the basic Wilson lines $b_L$ and $b_R$ can be rephrased as follows. Let us write a configuration as $\wC=[g_1.p_\std,g_2.p_\std,g_3.p_\std] \in \Conf_3 \P_G$. Then we have 
\begin{align*}
    b_L(\wC)=g_1^{-1}g_3\vw \quad\mbox{and}\quad b_R(\wC)=g_3^{-1}g_1\vw,
\end{align*}
and their regularity is also clear from this presentation. 
\end{rem}

\paragraph{\textbf{The $H^3$-action.}}
Recall the right $H^3$-action on $\Conf_3 \P_G$ given by $[p_{12},p_{23},p_{31}].(k_1,k_2,k_3)=[p_{12}.k_1,p_{23}.k_2,p_{31}.k_3]$ for $(k_1,k_2,k_3) \in H^3$. 
It is expressed in the standard configuration by 
\begin{align}\label{eq:H3-action_config}
    \wC_3(h_1,h_2,u_+).(k_1,k_2,k_3)=\wC_3(k_1^{-1}h_1w_0(k_2),k_1^{-1}h_2w_0(k_3),\Ad_{k_1}^{-1}(u_+))
\end{align}
for $(h_1,h_2,u_+) \in H \times H \times U_*^+$. 
By this action, the functions $b_L$ and $b_R$ are rescaled as
\begin{align}
    b_L(\wC.(k_1,k_2,k_3)) &= k_1^{-1}b_L(\wC)w_0(k_3), \label{eq:H-invarianceL}\\
    b_R(\wC.(k_1,k_2,k_3)) &= k_3^{-1}b_L(\wC)w_0(k_1).\label{eq:H-invarianceR}
\end{align}

The following is the geometric rephrasing:

\begin{lem}\label{l:H3-action}
Let us use the notation in \eqref{eq:moduli_triangle_explicit}. 
Then via the isomorphism 
\begin{align*}
    \wC_{3,m}:=f_{m}^{-1}\circ \wC_3: H \times H \times U_*^+ \xrightarrow{\sim} \P_{G,T},
\end{align*}
the action $\P_{G,T} \times H^3 \to \P_{G,T}$ defined by $(\alpha_{E},\alpha_{E'},\alpha_{E''})$ is given by
\begin{align*}
    \wC_{3,m}(h_1,h_2,u_+).(k_1,k_2,k_3) = \wC_{3,m}(k_1^{-1}h_1w_0(k_2), k_1^{-1}h_2w_0(k_3), \mathrm{Ad}_{k_1}^{-1}(u_+))
\end{align*}
for $(h_1,h_2,u_+) \in H \times H \times U_*^+$ and $(k_1,k_2,k_3) \in H^3$.
\end{lem}

Let us briefly summarize what we have seen. Given an ideal triangulation $\Delta$ of a marked surface $\Sigma$, we get the gluing morphism $q_\Delta: \widetilde{\P_{G,\Sigma}^\Delta}=\prod_{T \in t(\Delta)} \P_{G,T} \to \P_{G,\Sigma}$ (recall \cref{t:GS decomposition}). 
Through this morphism, the triangle pieces $\P_{G,T}$ can be regarded as local components of $\P_{G,\Sigma}$. 
Moreover, if we choose a marked point $m_T$ of each triangle $T$, then we get isomorphisms $f_{m_T}:\P_{G,T} \xrightarrow{\sim} \Conf_3 \P_G$ by \eqref{eq:moduli_triangle}, which enables us to study the moduli spaces in terms of the unipotent cells via the standard configuration (\cref{l:standardconfig}). In this paper, the choice of $m_T$ is indicated by the symbol $\ast$ in the figures (e.g.~\cref{fig:polygon} below). We call the data $\Delta_\ast:=(\Delta,(m_T)_{T \in t(\Delta)})$ a \emph{dotted triangulation}.

\subsubsection{Decomposition formula in the polygon case}\label{subsec:pairing}
Let $\Pi$ be an oriented $k$-gon, which can be regarded as a marked surface (\emph{i.e.}, a disk with $k$ special points on the boundary). 
Note that for two boundary intervals $E_\inn,E_\out$ of $\Pi$, there is a unique arc class of the form $[c]: E_\inn \to E_\out$ in this case.

Take an ideal triangulation $\Delta$ of $\Pi$. Choose a representative $c$ so that the intersection with $\Delta$ is minimal. Let $T_1,\dots,T_M$ be the triangles of $\Delta$ which $c$ traverses in this order. Note that for each $\nu=1,\dots,M$, the intersection $c \cap T_\nu$ is either one of the two patterns shown in \cref{f:intersection}. The \emph{turning pattern} of $c$ with respect to $\Delta$ is encoded in the sequence $\tau_\Delta([c]) = (\tau_\nu)_{\nu=1}^M \in \{L,R\}^{N}$, where $\tau_\nu = L$ (resp. $\tau_\nu=R$) if $c \cap T_\nu$ is the left (resp. right) pattern in \cref{f:intersection}. For our purpose, it is enough to consider the case $T_1 \cup \dots \cup T_M = \Pi$. An example for $k=6$ is shown in \cref{fig:polygon}. 

\begin{figure}
\[
\begin{tikzpicture}
\draw (0,0) coordinate (B1) node[above]{$B_1^{(\nu)}$} node[below]{$\ast$};
\draw (240: 2) coordinate (B2) node[left]{$B_2^{(\nu)}$};
\draw (300: 2) coordinate (B3) node[right]{$B_3^{(\nu)}$};
\draw (B1) -- (B2) -- (B3) --cycle;
\draw[->,thick,color=red] (240:1) arc[start angle=240, end angle=300, radius=1cm] node[midway,below]{$c$};
\draw(0,-2.5) node{$\tau_\nu=L$};
\begin{scope}[xshift=5cm]
\draw (0,0) coordinate (B1) node[above]{$B_1^{(\nu)}$} node[below]{$\ast$};
\draw (240: 2) coordinate (B2) node[left]{$B_2^{(\nu)}$};
\draw (300: 2) coordinate (B3) node[right]{$B_3^{(\nu)}$};
\draw (B1) -- (B2) -- (B3) --cycle;
\draw[->,thick,color=red] (300:1) arc[start angle=300, end angle=240, radius=1cm] node[midway,below]{$c$};
\draw(0,-2.5) node{$\tau_\nu=R$};
\end{scope}
\end{tikzpicture}
\]
\caption{Two intersection patterns of $c \cap T_\nu$}
\label{f:intersection}
\end{figure}

\begin{figure}
\begin{tikzpicture}
\foreach \i in {0,1,2,3,4,5}
\draw(60*\i:2) coordinate(A\i) -- (60*\i+60:2);
\draw(A0) -- (A2);
\draw(A2) -- (A5);
\draw(A5) -- (A3);
\draw(210:2) node{$E_\inn$};
\draw(30:2) node[right=-0.5em]{$E_\out$};
\draw(2,0) node[above left]{$\ast$};
\draw(120:2) node[below right]{$\ast$};
\draw(300:2) node[above left]{$\ast$};
\draw(180:2) node[below right]{$\ast$};
\path(210:1.732)++(0:0.5) coordinate(c1);
\path(30:1.732)++(180:0.5) coordinate(c2);
\draw[->,thick,red] (210:1.732) ..controls (c1) and (c2) .. node[midway,above]{$c$} (30:1.732);
\end{tikzpicture}
    \caption{An example of arc class for $k=6$. The turning pattern is $\tau_\Delta([c])=(L,R,L,R)$.}
    \label{fig:polygon}
\end{figure}

\begin{dfn}\label{d:dot associated with side pair}
Let $\Delta$ be an ideal triangulation of $\Pi$. 
Let $[c]:E_\inn \to E_\out$ be an arc class such that its minimal representative $c$ traverses every triangle of $\Delta$. We put a dot $m_\nu$ on the triangle $T_\nu$ so that the arc $c \cup T_\nu$ separates the corner with the dot $m_\nu$ from the other two corners (see \cref{f:intersection}). The resulting dotted triangulation is written as $\Delta_*=\Delta_*([c])$, and called the \emph{dotted triangulation associated with the arc class $[c]$}.
\end{dfn}
Let us consider the restriction of the Wilson line $g_{[c]}$ to the substack $\P_{G,\Pi}^\Delta$. Let $q_\Delta: \prod_{\nu=1}^M \P_{G,T_\nu} \to \P_{G,\Pi}^\Delta$ be the gluing morphism, and
define
\begin{align*}
    g_\nu:= 
    \begin{cases} 
        b_L\circ f_{m_\nu}: \P_{G,T_\nu} \to B^+_* & \text{if $\tau_\nu = L$}, \\
        b_R\circ f_{m_\nu}: \P_{G,T_\nu} \to B^-_* & \text{if $\tau_\nu = R$}
    \end{cases}
\end{align*}
for $\nu=1,\dots,M$. Here recall the isomorphisms \eqref{eq:moduli_triangle}. 
Then we have the following:

\begin{prop}[Decomposition formula: polygon case]\label{p:mon-decomp}
For an arc class $[c]: E_\inn \to E_\out$ on $\Pi$ as above, we have
\begin{align*}
    q_{\Delta}^*g_{[c]} = \mu_M \circ \prod_{\nu=1}^M g_\nu,
\end{align*}
where $\mu_M$ denotes the multiplication of $M$ elements in $G$.
\end{prop}

\begin{proof}
We proceed by induction on $M \geq 1$. The case $M=1$ can be verified by comparing the presentation of the Wilson line (\cref{prop:presentation_Wilson_line}) with \cref{rem:LR_Wilson_line}. For $M >1$, consider the polygon $\Pi':=T_1\cup \dots \cup T_{M-1}$ and the gluing morphism 
\begin{align*}
    q_M:\P_{G,\Pi' \sqcup T_M} \to \P_{G,\Pi}
\end{align*}
where a boundary interval $E$ of $\Pi'$ and that $E'$ of $T_M$ are glued to give the common edge of $T_{M-1}$ and $T_M$ in $\Pi$. 
Then \cref{prop:multiplicativity} implies $q_M^*g_{[c]}= g_{[c']}\cdot g_{[c_M]}$, where $[c']:E_\inn \to E$ and $[c_M]:E' \to E_\out$. The assertion follows from the induction hypothesis. 

\end{proof}

\subsubsection{Decomposition formula in general: covering argument}\label{subsec:regularity_of_Wilson_lines_and_loops}

Assume $\partial \Sigma \neq \emptyset$ and choose an arc class $[c]: E_\inn \to E_\out$. 
Let $\varpi: \widetilde{\Sigma} \to \Sigma^\ast$ be the universal cover, and take a representative $c$ and its lift $\widetilde{c}$ to $\widetilde{\Sigma}$. 
Fix an ideal triangulation $\Delta$ of $\Sigma$, which is also lifted to a tesselation $\widetilde{\Delta}$ of the universal cover. Applying an isotopy if necessary, we may assume that the intersections of $c$ with $\Delta$ and $\widetilde{c}$ with $\widetilde{\Delta}$ are minimal. 
Let $\Pi_{c;\Delta} \subset \widetilde{\Sigma}$ be the smallest polygon (a union of triangles in $\widetilde{\Delta}$) containing $\widetilde{c}$. The two endpoints of $\widetilde{c}$ lies on lifts of the edges $E_\inn$, $E_\out$, which are denoted by $\widetilde{E}_\inn$ and $\widetilde{E}_\out$, respectively. 
By definition the polygon $\Pi_{c;\Delta}$ is equipped with an ideal triangulation induced from $\widetilde{\Delta}$, which we denote by $\Delta_c$. 
Let us write the associated turning pattern as $\tau_\Delta([c]):=\tau_{\Delta_c}(\widetilde{c})=(\tau_1,\dots,\tau_M) \in \{L,R\}^M$, which we call the \emph{turning pattern} of the arc class $[c]$ with respect to $\Delta$.  
Let $\widetilde{T}_1,\dots,\widetilde{T}_M$ be the sequence of triangles of $\Delta_c$ which are traversed by $\widetilde{c}$ in this order.  

Let $\pi_c:=\varpi|_{\Pi_{c;\Delta}}: \Pi_{c;\Delta} \to \Sigma^\ast$, which is a covering map over its image. It induces a map $\pi_c^*: \P_{G,\Sigma} \to \P_{G,\Pi_{c;\Delta};\{\widetilde{E}_\inn,\widetilde{E}_\out\}}$ via pull-back. Here recall \cref{def:moduli_partial_pinnings}. 
From the definitions and \cref{rem:partial_pinnings}, we have:

\begin{lem}\label{l:line-pairing}
The following diagram commutes:
\begin{equation*}
    \begin{tikzcd}
        \P_{G,\Sigma} \ar[r,"\pi_c^*"] \ar[rd,"g_{[c]}"'] & \P_{G,\Pi_{c;\Delta};\{\widetilde{E}_\inn,\widetilde{E}_\out\}} \ar[d, "g_{[\widetilde{c}]}"]\\
         & G.
    \end{tikzcd}
\end{equation*}
\end{lem}
Combined with \cref{p:mon-decomp}, we would obtain a decomposition formula for Wilson lines. We are going to write it down explicitly.

Fix a dotted triangulation $\Delta_*$ of $\Sigma$. Let $m_T$ denote the dot assigned to a triangle $T \in t(\Delta)$. 
Note that $\Delta_*$ determines a dotted triangulation $\Delta_*^{\mathrm{lift}}$ of the polygon $\Pi_{c;\Delta}$ over $\Delta_c$ by lifting the dots, which may not agree with the \lq\lq canonical'' dotted triangulation $\Delta_*^{\mathrm{can}}:=\Delta_*([\widetilde{c}])$ associated with the arc class $[\widetilde{c}]$ in the polygon. The only difference is the position of dots: denote the dot on the triangle $\widetilde{T}_\nu$ for the triangulation $\Delta_*^{\mathrm{can}}$ (resp. $\Delta_*^{\mathrm{lift}}$) by $m_\nu$ (resp. $n_\nu$) for $\nu=1,\dots,M$. Then the disagreement of the two dots $m_\nu$ and $n_\nu$ results in a relation $f_{m_\nu} = \mathcal{S}_3^{t_\nu}\circ f_{n_\nu}$ for some $t_\nu \in \{0,\pm 1\}$. 

For $\nu=1,\dots,M$, let  $T_\nu:=\pi_c(\widetilde{T}_\nu) \in t(\Delta)$ denote the projected image of the $\nu$-th triangle, which do not need to be distinct. Finally, set $\overline{f}_{n_\nu} := f_{m_{T_\nu}} \circ \mathrm{pr}_{T_\nu}:\prod_{T \in t(\Delta)} \P_{G,T} \to \Conf_3 \P_G$. 

\begin{thm}[Decomposition formula for Wilson lines]\label{t:Wilson_line_regular}
For an arc class $[c]:E_\inn \to E_\out$ and a dotted triangulation $\Delta_*$,
let the notations as above. Then we have 
\begin{align*}
    q_\Delta^*g_{[c]} = \mu_M\circ  \prod_{\nu=1}^M (b_{\tau_\nu}\circ \mathcal{S}_3^{t_\nu} \circ \overline{f}_{n_\nu}).
\end{align*}
\end{thm}
See \cref{e:Wilson_line_annulus} and \cref{fig:Wilson_line_example} for an example. 


\begin{proof}

Since the pull-back via the covering map $\pi_c$ commutes with the gluing morphisms, we have the commutative diagram

\begin{equation}\label{eq:pullback-gluing}
    \begin{tikzcd}
    \prod_{T \in t(\Delta)} \P_{G,T} \ar[r,"\widetilde{\pi}_c^*"] \ar[d,"q_\Delta"'] & \prod_{\nu=1}^M \P_{G,\widetilde{T}_\nu} \ar[d,"q_{\Delta_c}"] \\
    \P_{G,\Sigma}^\Delta \ar[r,"\pi_c^*"'] & \P_{G,\Pi_{c;\Delta};\{\widetilde{E}_\inn,\widetilde{E}_\out\}}^{\Delta_c},
    \end{tikzcd}
\end{equation}
where $\widetilde{\pi}_c^*:=\prod_{\nu=1}^M (\pi_c|_{\widetilde{T}_\nu})^*$, and the right vertical map is the composite of the gluing morphism and the projection forgetting the pinnings except for those assigned to $\widetilde{E}_\inn$ and $\widetilde{E}_\out$. 
From the definition of $t_\nu$ and $\overline{f}_{n_\nu}$, the following diagram commutes:

\begin{equation}\label{eq:dots_adjustment}
    \begin{tikzcd}
    \prod_{\nu=1}^M \Conf_3 \P_G \quad \ar[r,"\prod_\nu \mathcal{S}_3^{t_\nu}"] & \quad \prod_{\nu=1}^M \Conf_3 \P_G \\
    \prod_{T \in t(\Delta)} \P_{G,T} \ar[u,"( \overline{f}_{n_\nu})_{\nu}"] \ar[r,"\widetilde{\pi}_c^*"'] \quad & \quad \prod_{\nu=1}^M \P_{G,\widetilde{T}_\nu}. \ar[u,"\prod_\nu f_{m_\nu}"'] \ar[lu,"\prod_\nu f_{n_\nu}"]
    \end{tikzcd}
\end{equation}

Combining together, we get 
\begin{alignat}{2}\label{eq:Wilson_line_product_formula}
    q_\Delta^*g_{[c]} 
    &= g_{[\widetilde{c}]} \circ \pi_c^* \circ q_\Delta &\qquad &\mbox{(by \cref{l:line-pairing})} \\
    &= g_{[\widetilde{c}]} \circ q_{\Delta_c} \circ \widetilde{\pi}_c^* &\qquad &\mbox{(by \eqref{eq:pullback-gluing})} \\
    &= \mu_M\circ \left(\prod_{\nu=1}^M g_\nu\right) \circ \widetilde{\pi}_c^* &\qquad &\mbox{(by \cref{p:mon-decomp})} \\
    &=\mu_M\circ \left(\prod_{\nu=1}^M b_{\tau_\nu}\circ f_{m_\nu}\right) \circ \widetilde{\pi}_c^* &\qquad &\label{eq:Wilson_line_product_formula_2} \\
    &=\mu_M\circ  \prod_{\nu=1}^M (b_{\tau_\nu}\circ \mathcal{S}_3^{t_\nu} \circ \overline{f}_{n_\nu}) &\qquad & \mbox{(by \eqref{eq:dots_adjustment})},
\end{alignat}
as desired. 

\end{proof}

\begin{ex}[A Wilson line on a marked annulus]\label{e:Wilson_line_annulus}
Let $\Sigma$ be a marked annulus with one marked point on each boundary component, equipped with the dotted triangulation $\Delta$ shown in \cref{fig:Wilson_line_example}. Consider the arc class $[c]: E_0 \to E_4$ shown there. The turning pattern is $\tau_\Delta([c])=(L,L,R,R)$. 
Under the projection $\pi_c: \Pi_{c;\Delta} \to \Sigma^\ast$, we have $T=\pi_c(\widetilde{T}_1)=\pi_c(\widetilde{T}_3)$, $T'=\pi_c(\widetilde{T}_2)=\pi_c(\widetilde{T}_4)$. By comparing the two pictures in the right, we have $t_1=0$, $t_2=0$, $t_3=-1$, $t_4=1$. Thus we have
\begin{align}\label{eq:annulus_example}
    q_\Delta^* g_{[c]} = \mu_4\circ ((b_L \circ f_{m_T})\times(b_L \circ f_{m_{T'}})\times(b_R \circ \mathcal{S}_3^{-1} \circ f_{m_T})\times(b_R \circ \mathcal{S}_3 \circ f_{m_{T'}})).
\end{align}
Of course, we could have chosen another triangulation $\Delta'$ obtained from $\Delta$ by the flip along the edge $E_2$. In this case we have $\tau_{\Delta'}([c])=(L,R)$ and we need no cyclic shifts to express $g_{[c]}$. 
\end{ex}

In the situation of \cref{prop:Wilson line-loop}, an ideal triangulation $\Delta$ of $\Sigma$ descends to an ideal triangulation $\Delta'$ of $\Sigma'$. Then $q_\Delta^\ast g_{[c]}:\prod_{T \in t(\Delta)}\P_{G,T} \to G$ induces $q_{\Delta'}^\ast \rho_{|\gamma|}:\prod_{T \in t(\Delta')}\P_{G,T} \to [G/\Ad G]$, where $t(\Delta')=t(\Delta)$. 
Hence the decomposition formula given in \cref{t:Wilson_line_regular} also gives a formula for $q_{\Delta'}^\ast \rho_{|\gamma|}$. 

\begin{ex}[A Wilson loop on a once-punctured torus]
A once-punctured torus $\Sigma'$ is obtained by gluing the boundary intervals $E_0$ and $E_4$ of the marked annulus $\Sigma$ considered above. The arc class $[c]:E_0 \to E_4$ descends to a free loop $|\gamma| \in \hat{\pi}({\Sigma'}^\ast)$. 
Then $q_{\Delta'}^\ast \rho_{|\gamma|}$ can be computed from \eqref{eq:annulus_example}.
\end{ex}

\begin{figure}
\[
\begin{tikzpicture}
\fill[gray!30] (0,0) circle [radius=0.5];
\draw (0,0) circle [radius=0.5];
\draw (0,0) circle [radius=2];
\fill(0.5,0) circle(2pt);
\fill(2,0) circle(2pt);
\draw[->>,shorten >=2pt] (0.5,0) -- (2,0);
\draw(0.5,0) ..controls (0.5,-1) and (-1,-1) .. (-1,0) ..controls (-1,1) and (1,2) .. (2,0);
\draw[->,thick,red] (-0.5,0) ..controls (-1.5,-1.5) and (1,-2) .. (1,0) node[midway,below right]{$c$} .. controls (1,1) and (0.5,1.5) .. (0,2);
\draw(0,-0.6) node{$\ast$};
\draw(0.6,-0.2) node{$\ast$};
\draw(0,0.7) node{$T$};
\draw(-1.5,0) node{$T'$};
\draw(0,-2.5) node{$(\Sigma,\Delta_*)$};
{\color{blue}
\draw(-0.5,0) node[right]{\scalebox{0.8}{$E_0$}};
\draw(-1,1) node{\scalebox{0.8}{$E_1$}};
\draw(1.5,0) node[below]{\scalebox{0.8}{$E_2$}};
\draw(0,2) node[above]{\scalebox{0.8}{$E_4$}};
}
\begin{scope}[xshift=4cm, yshift=2cm]
\draw(0,1) -- (8,1) -- (8,-1) -- (0,-1) --cycle;
\draw(4,1) -- (0,-1);
\draw(8,1) -- (4,-1);
\foreach \x in {0,4,8}
{
\fill(\x,1) circle(2pt);
\fill(\x,-1) circle(2pt);
\draw[->>,shorten >=2pt] (\x,1) -- (\x,-1);
}
\draw[->,thick,red] (2,1) node[above]{\scalebox{0.9}{$\widetilde{E}_\inn$}} ..controls (2,0) and (6,0) .. (6,-1) node[below]{\scalebox{0.9}{$\widetilde{E}_\out$}} node[midway,above right]{$\widetilde{c}$};
\draw(1,0.5) node{$\widetilde{T}_1$};
\draw(3,-0.5) node{$\widetilde{T}_2$};
\draw(5,0.5) node{$\widetilde{T}_3$};
\draw(7,-0.5) node{$\widetilde{T}_4$};
\draw(3,0.8) node{$\ast$};
\draw(3.6,0.5) node{$\ast$};
\draw(7,0.8) node{$\ast$};
\draw(7.6,0.5) node{$\ast$};
\draw(4,-1.5) node{$(\Pi_{c;\Delta},\Delta_*^{\mathrm{lift}})$};
\end{scope}

\begin{scope}[xshift=4cm, yshift=-2cm]
\draw(0,1) -- (8,1) -- (8,-1) -- (0,-1) --cycle;
\draw(4,1) -- (0,-1);
\draw(8,1) -- (4,-1);
\foreach \x in {0,4,8}
{
\fill(\x,1) circle(2pt);
\fill(\x,-1) circle(2pt);
\draw[->>,shorten >=2pt] (\x,1) -- (\x,-1);
}
\draw[->,thick,red] (2,1) node[above]{\scalebox{0.9}{$\widetilde{E}_\inn$}} ..controls (2,0) and (6,0) .. (6,-1) node[below]{\scalebox{0.9}{$\widetilde{E}_\out$}} node[midway,above right]{$\widetilde{c}$};
\draw(1,0.5) node{$\widetilde{T}_1$};
\draw(3,-0.5) node{$\widetilde{T}_2$};
\draw(5,0.5) node{$\widetilde{T}_3$};
\draw(7,-0.5) node{$\widetilde{T}_4$};
\draw(3,0.8) node{$\ast$};
\draw(3.6,0.5) node{$\ast$};
\draw(4.4,-0.5) node{$\ast$};
\draw(5,-0.8) node{$\ast$};
\draw(4,-1.5) node{$(\Pi_{c;\Delta},\Delta_*^{\mathrm{can}})$};
\end{scope}
\end{tikzpicture}
\]
\caption{A marked annulus $\Sigma$ with a dotted triangulation $\Delta_*$ (left), a polygon $\Pi_{c;\Delta}$ with the dotted triangulations $\Delta_*^{\mathrm{lift}}$ (right top) and $\Delta_*^{\mathrm{can}}$ (right bottom). Glued edges are marked by double-head arrows.}
\label{fig:Wilson_line_example}
\end{figure}

\section{Factorization coordinates and their relations}\label{sec:coordinate expressions}
As a preparation for the subsequent sections, we recall several parametrizations and coordinates of factorizing nature: Lusztig parametrizations on unipotent cells, coweight parametrizations of double Bruhat cells, and Goncharov--Shen coordinates on the configuration space $\Conf_3 \P_G$. 
A necessary background on the cluster algebra is reviewed in \cref{sec:quivers}. 

\begin{conv}\label{calculus_notation}
For a torus $T=\bG_m^N$ equipped with a coordinate system $\mathbf{X}=(X_k)_{k=1}^N$ and a map $f: T \to V$ to a variety $V$, we occasionally write $f= f(\mathbf{X})$ as in the usual calculus. 
\end{conv}

\subsection{Lusztig parametrizations on the unipotent cells and the Goncharov--Shen potentials}\label{subsec:Lusztig}
For $w \in W(G)$, let $U^\pm_w := U^\pm \cap B^\mp w B^\mp$ denote the \emph{unipotent cell}. 

\begin{prop}\label{p:Lusztig-param}
Given a reduced word $\bs = (s_1,\dots, s_l)$ of $w$, the maps
\begin{align*}
    &x^\bs: \bG_m^l \to U^+_w, \quad (t_1,\dots,t_l) \mapsto x_{s_1}(t_1) \dots x_{s_l}(t_l), \\
    &y^\bs: \bG_m^l \to U^-_w, \quad (t_1,\dots,t_l) \mapsto y_{s_1}(t_1) \dots y_{s_l}(t_l)
\end{align*}
are open embeddings.
\end{prop}
We call these parametrizations \emph{Lusztig parametrizations}. 
These maps induce injective $\mathbb{C}$-algebra homomorphisms 
\begin{align}
    &(x^\bs)^{\ast}:  \cO(U^+_w)\to \cO(\bG_m^l )=\C[t_1^{\pm 1},\dots,t_l^{\pm 1}],\label{eq:Lusztig-param}\\
    &(y^\bs)^{\ast}:  \cO(U^-_w)\to \cO(\bG_m^l )=\C[t_1^{\pm 1},\dots,t_l^{\pm 1}].
\end{align}
When $w=w_0$ is the longest element, we have $U^\pm_{w_0}=U^\pm_*$.

\paragraph{\textbf{Goncharov--Shen potentials}}
Let us consider the configuration space 
\begin{align*}
    \Conf^{\ast} (\A_G,\B_G,\B_G):= G \backslash \{(\widehat{B}_1,B_2,B_3)  \mid \text{$(B_1,B_2)$ and $(B_1,B_3)$ are generic} \}.
\end{align*}
It has a parametrization 
\begin{align*}
    \beta_3: U^+ \xrightarrow{\sim} \Conf (\A_G,\B_G,\B_G), \quad u_+ \mapsto \left[[U^+], B^-, u_+. B^-\right].
\end{align*}
The map $\beta_3^{-1}$ pulls-back the additive characters $\chi_s: U^+ \to \bA^1$,  $\chi_s(u_+):=\Delta_{\varpi_s, r_s\varpi_s}(u_+)$ for $s \in S$ to give a function
\[
W_s:=\chi_s\circ \beta_3^{-1}: \Conf^* (\A_G,\B_G,\B_G)\to \bA^1,
\]
 which we call the \emph{Goncharov--Shen potential}. 
Note that $u_+ \in U^+_w$ if and only if $w(B_2,B_3)=w^\ast$, when we write $\beta_3(u_+)=(A_1,B_2,B_3)$. Here $W(G)\to W(G), w\mapsto w^{\ast}$ is an involution given by
\[
w^{\ast}=w_0ww_0.
\]
The following relation will be used later:

\begin{lem}\label{l:char-lus}
For a reduced word $\bs=(s_1,\dots,s_l)$ of $w$, let $u_+=x_{s_1}(t_1) \dots x_{s_l}(t_l) \in U^+_w$ be the corresponding Lusztig parametrization. Then we have
\begin{align*}
    \chi_s(u_+) = \sum_{k: s_k=s} t_k.
\end{align*}
\end{lem}

\subsection{Coweight parametrizations on double Bruhat  cells}\label{subsec:DBcells}
The coweight parametrizations on double Bruhat cells are introduced in \cite{FG06} and further investigated in \cite{Williams}. 

Let $G$ be an adjoint group. 
For each $u,v \in W(G)$, the \emph{double Bruhat cell} is defined to be $G^{u,v} := B^+ u B^+ \cap B^- v B^-$. It is a subvariety of $G$. In this paper, we only treat with the special cases $u=e$ or $v=e$. See \cite{FG06,Williams} for the general construction\footnote{Indeed, the general case is obtained by a suitable amalgamation from the cases $u=e$ and $v=e$.}. 

Let us write $B^+_v:=G^{e,v}$ and $B^-_u:=G^{u,e}$. First consider $B^+_v$. 
Let $\bs=(s_1, \dots, s_l)$ be a reduced word for $v$. Then the \emph{evaluation map} $ev^+_\bs: \bG_m^{n+l} \to B^+_v$ is defined by
\begin{align*}
    ev^+_\bs(\mathbf{x}):= \left(\prod_{s=1}^n H^s(x_s) \right) \cdot \dprod_{k=1,\dots,l} (\mathbb{E}^{s_k}H^{s_k}(x_{n+k})),
\end{align*}
where $\mathbf{x}=(x_k)_{k=1}^{n+l}$ and the symbol $\overrightarrow{\prod}_{k=1,\dots,l}$ means that we multiply the elements successively from the left to the right, namely $\overrightarrow{\prod}_{k=1,\dots,l}g_k:=g_1\dots g_l$. 
Similarly in the case $v=e$, we take a reduced word $\bs$ for $u$ and define $ev^-_\bs: \bG_m^{n+l} \to B^-_u$ by replacing each $\mathbb{E}$ with $\mathbb{F}$. We call the variables $\mathbf{x}=(x_k)_k$ the \emph{coweight parameters}.

The following indexing for the coweight parameters $\mathbf{x}$ will turn out to be useful: for a reduced word $\bs=(s_1,\dots,s_l)$ of an element of $W(G)$, let $k(s,i)$ denote the $i$-th number $k$ such that $s_k=s$. Let $n^s(\bs)$ be the number of $s$ which appear in the word $\bs$. If we relabel the variables as 
\begin{align}
 x_i^s:=x_{k(s, i)}   
\end{align}\label{eq:relabeling}
for $s \in S,~ i=0,\dots, n^s(\bs)$, then they always appear in the form $H^s(x_i^s)$ in the expression of $ev^\pm_{\bs}(\mathbf{x})$. 
Let 
\begin{align}\label{eq:relabeling_set}
    I(\bs):=\{(s,i) \mid s \in S,~ i=0,\dots, n^s(\bs)\}.
\end{align}
Then for a reduced word $\bs$ of $u \in W(G)$, the evaluation maps give open embeddings $ev_\bs^\pm: \bG_m^{I(\bs)} \to B^\pm_u$ where the variable assigned to the component $(s,i) \in I(\bs)$ is substituted to the $k(s,i)$-th position.

\begin{ex}[Type $A_3$]\label{e:coweight}
Let $G=PGL_4$. 
The evaluation map associated with the reduced word $\bs=(1,2,3,1,2,1)$ is given by
\begin{align*}
&ev^+_\bs(\mathbf{x}) \\
	 &= H^1(x_0^1)H^2(x_0^2)H^3(x_0^3) \mathbb{E}^1 H^1(x_1^1) \mathbb{E}^2 H^2(x_1^2) \mathbb{E}^3 H^3(x_1^3) \mathbb{E}^1 H^1(x_2^1) \mathbb{E}^2 H^2(x_2^2) \mathbb{E}^1 H^1(x_3^1).
\end{align*}
The evaluation maps are compatible with group multiplication. For example, let us consider $\bs:=(1,2,3)$ and $\bs':=(1)$. Then we have 
\begin{align*}
    &ev^+_\bs(x_0^1,x_0^2,x_0^3,x_1^1,x_1^2,x_1^3)\cdot ev^+_{\bs'}(y_0^1,y_0^2,y_0^3,y_1^1) \\
    &=H^1(x_0^1)H^2(x_0^2)H^3(x_0^3) \mathbb{E}^1 H^1({\color{red} x_1^1}) \mathbb{E}^2 H^2({\color{blue} x_1^2}) \mathbb{E}^3 H^3({\color{mygreen} x_1^3})\cdot H^1({\color{red} y_0^1})H^2({\color{blue} y_0^2})H^3({\color{mygreen} y_0^3}) \mathbb{E}^1 H^1(y_1^1) \\
    &=H^1(x_0^1)H^2(x_0^2)H^3(x_0^3) \mathbb{E}^1 H^1({\color{red} x_1^1y_0^1}) \mathbb{E}^2 H^2({\color{blue} x_1^2y_0^2}) \mathbb{E}^3 H^3({\color{mygreen} x_1^3y_0^3})\mathbb{E}^1 H^1(y_1^1) \\
    &=ev^+_{\bs''}(x_0^1,x_0^2,x_0^3,x_1^1y_0^1,x_1^2y_0^2,x_1^3y_0^3,y_1^1).
\end{align*}
with $\bs'':=(1,2,3,1)$. 
Here in the third line, we used the fact that $H^s(x)$ and $H^t(y)$ always commute with each other, and that $\mathbb{E}^s$ and $H^t(x)$ commutes with each other when $s \neq t$. If we denote the variable assigned to the component $(i,s) \in I(\bs'')$ by $z_i^s$, then 
\begin{align*}
    z_0^1&=x_0^1, & z_0^2&=x_0^2, & z_0^3&=x_0^3, & \\
    z_1^1&=x_1^1y_0^1, & z_1^2&=x_1^2y_0^2, & z_1^3&=x_1^3y_0^3, & z_2^1=y_1^1.
\end{align*}
\end{ex}

For a reduced word $\bs$ of $w \in W(G)$ and $\epsilon \in \{+,-\}$, each variable $x_i^s$ of the coweight parametrization $ev_{\bs}^\epsilon$ is assigned to the vertex $v_i^s$ of the weighted quiver $\bJ^\epsilon(\bs)$. See \cref{sec:quivers}. The group multiplication corresponds to an appropriate amalgamation of quivers. For example, the multiplication considered in \cref{e:coweight} corresponds to the quiver amalgamation shown in \cref{f:mult-amal}. The pair $\sfS^\epsilon(\bs):=(\bJ^\epsilon(\bs), (x_i^s)_{(s,i) \in I(\bs)})$ forms an \emph{$X$-seed} in the ambient field $\cF=\Rat(B^\epsilon_w)$. 

\begin{thm}[Fock-Goncharov \cite{FG06}, Williams \cite{Williams}]
For an element $w \in W(G)$ and $\epsilon \in \{+,-\}$, the seeds $\sfS^\epsilon(\bs)$ associated with reduced words $\bs$ of $w$ are mutation-equivalent to each other. Hence the collection$(\sfS^\epsilon(\bs))_{\bs}$ is a cluster Poisson atlas (\cref{d:cluster atlas}) on the double Bruhat cell $B^\epsilon_w$. 
\end{thm}


\begin{figure}
    \begin{tikzpicture}
\begin{scope}[>=latex]
\foreach \i in {0,1}
\draw (2*\i+1,0) circle(2pt) node[below]{$x_{\i}^1$};
\foreach \i in {0,1}
\draw (2*\i+2, 1) circle(2pt) node[above=0.2em]{$x_{\i}^2$};
\foreach \i in {0,1}
\draw (2*\i+3, 2) circle(2pt) node[above]{$x_{\i}^3$};
\qarrow{3,0}{1,0}
\qarrow{4,1}{2,1}
\qarrow{5,2}{3,2}
\qdarrow{1,0}{2,1}
\qdarrow{2,1}{3,2}
\qdarrow{3,0}{4,1}
\qdarrow{4,1}{5,2}
\qarrow{2,1}{3,0}
\qarrow{3,2}{4,1}
\foreach \i in {0,1}
\draw (2*\i+5,0) circle(2pt) node[below]{$y_{\i}^1$};
\foreach \i in {0}
\draw (2*\i+6, 1) circle(2pt) node[above=0.2em]{$y_{\i}^2$};
\foreach \i in {0}
\draw (2*\i+7, 2) circle(2pt) node[above]{$y_{\i}^3$};
\qarrow{7,0}{5,0}
\qdarrow{6,1}{7,0}
\qdarrow{5,0}{6,1}

\draw[ultra thick,-{Classical TikZ Rightarrow[length=4pt]}] (8,1) to (9,1);

{\begin{scope}[xshift=9cm]
\foreach \i in {0,1,2}
\draw (2*\i+1,0) circle(2pt);
\foreach \i in {0,1}
\draw (2*\i+2, 1) circle(2pt);
\foreach \i in {0,1}
\draw (2*\i+3, 2) circle(2pt);
\draw (1,0) node[below]{$x_0^1$};
\draw (3,0) node[below]{$x_1^1y_0^1$};
\draw (5,0) node[below]{$x_2^1$};
\draw (2,1) node[left=0.2em]{$x_0^2$};
\draw (4,1) node[right=0.2em]{$x_1^2y_0^2$};
\draw (3,2) node[above]{$x_0^3$};
\draw (5,2) node[above]{$x_1^3y_0^3$};
\foreach \i in {1,2}
\qarrow{2*\i+1,0}{2*\i-1,0};
\foreach \i in {1}
\qarrow{2*\i+2,1}{2*\i,1};
\qarrow{5,2}{3,2}
\qarrow{2,1}{3,0}
\qdarrow{4,1}{5,0}
\qarrow{3,2}{4,1}
\qdarrow{1,0}{2,1}
\qdarrow{2,1}{3,2}
\qarrow{3,0}{4,1}
\qdarrow{4,1}{5,2}
\end{scope}}
\end{scope}
    \end{tikzpicture}
    \caption{Amalgamation of the quivers $\bJ^+(1,2,3)$ and $\bJ^+(1)$ for type $A_3$ produces the quiver $\bJ^+(1,2,3,1)$.}
    \label{f:mult-amal}
\end{figure}

The following lemma directly follows from the definition of the Dynkin involution and \cref{l:Dynkininv}, which will be useful in the sequel.

\begin{lem}\label{l:generator_involution}
We have the following relations:
\begin{align*}
    \vw^{-1} H^s(x)^{-1} \vw &= H^{s^\ast}(x), \\
    \vw^{-1} (\mathbb{E}^s)^{-1} \vw &= \mathbb{F}^{s^\ast},\\
    \vw^{-1} (\mathbb{F}^s)^{-1} \vw &= \mathbb{E}^{s^\ast}.
\end{align*}
\end{lem}

Since the map $g \mapsto \vw^{-1} g^{-1} \vw$ is an anti-homomorphism, we get the following:

\begin{cor}\label{l:ev_involution}
For a reduced word $\bs=(s_1,\dots,s_N)$ of $w_0 \in W(G)$, let  $\bs^\ast_{\op}:=(s_N^\ast,\dots,s_1^\ast)$. 
Then we have
\begin{align*}
    \vw^{-1} ev^+_{\bs}(\mathbf{x})^{-1} \vw = ev^-_{\bs^\ast_{\op}}\circ \iota^*(\mathbf{x}),
\end{align*}
where $\iota^*:\bG_m^{I(\bs)} \to \bG_m^{I((\bs)_\op^\ast)}$ is an isomorphism induced by the bijection
\begin{align}\label{eq:iota-involution}
    \iota: I(\bs_\op^\ast) \to I(\bs),\quad (s^\ast,i) \mapsto (s,n^s(\bs)-i). 
\end{align}
\end{cor}

\subsection{Goncharov--Shen coordinates on $\Conf_3 \P_G$} \label{subsec:GScoord}
We recall the Goncharov--Shen's cluster Poisson coordinate system on $\Conf_3 \P_G$ associated with a reduced word $\bs = (s_1,\dots,s_N)$ of $w_0 \in W(G)$. See \cite{GS19} for a detail. 
Let $[B_1,B_2,B_3; p_{12},p_{23},p_{31}] \in \Conf_3 \P_G$. 
Using \cref{c:flagchain}, we take the decomposition of the generic pair $(B_2,B_3)$ with respect to $\bs$:
\[
B_2 = B_2^0 \xrightarrow{s_1^{\ast}} B_2^1 \xrightarrow{s_2^{\ast}} \dots \xrightarrow{s_N^{\ast}} B_2^N = B_3,
\]
where $w(B_2^{k-1}, B_2^k) = s_k^{\ast}$ for $k = 1,\dots, N$. Suppose that the triple $(B_1,B_2,B_3)$ is ``sufficiently generic'' so that each pair $(B_1,B_2^k)$ is generic for $k=0,\dots,N$. Let $\widehat{B}_1$, $\widehat{B}'_1$ be two lifts of $B_1$ determined by the pinnings $p_{12}$, $p_{31}^\ast$, respectively. Now we define:
\begin{align*}
X_{s \choose i}:=
\begin{cases}
W_{s} \left(\widehat{B}_1,B_2, B_2^{k(s,1)} \right) & (i=0), \\
W_{s} \left(\widehat{B}_1, B_2^{k(s,i)}, B_2^{k(s,i+1)} \right)/W_{s} \left(\widehat{B}_1, B_2^{k(s,i-1)}, B_2^{k(s,i)} \right) & (i=1,\dots, n^{s}(\bs)-1), \\
W_{s} \left( \widehat{B}'_1,B_2^{k(s,n^{s}(\bs)-1)}, B_2^{k(s,n^{s}(\bs))} \right)^{-1} & (i=n^{s}(\bs)).
\end{cases}
\end{align*}
Here as before, $k(s,i)$ denotes the $i$-th number $k$ such that $s_k=s$ in $\bs$. Using the embedding $\bG_m \hookrightarrow \bA^1$, we regard $X_{s \choose i}$ as a $\bG_m$-valued rational function on $\Conf_3 \P_G$. 


Let $\widetilde{G}$ be the simply-connected group which covers $G$ and take a lift
\begin{align}\label{eq:liftdecomp}
\widetilde{B}_2 = \widetilde{B}_2^0 \xrightarrow{s_1^{\ast}} \widetilde{B}_2^1 \xrightarrow{s_2^{\ast}} \dots \xrightarrow{s_N^{\ast}} \widetilde{B}_2^N = \widetilde{B}_3
\end{align}
of the above chain to $\A_{\widetilde{G}}:=\widetilde{G}/U^+$ so that the pair $(\widetilde{B}_2,\widetilde{B}_3)$ is determined by the pinning $p_{23}$ and the conditions  $h(\widetilde{B}_2^{j},\widetilde{B}_2^{j-1}) = 1$ for $j=1,\dots,N$ hold. Here the $h$-invariant and the $w$-distance of the elements of $\A_{\widetilde{G}}^{\times 2}$ are defined in the same way as those for $\A_{G}^{\times 2}$. Such a lift exists thanks to \cite[Lemma-Definition 5.3]{GS19}. Then we have the \emph{primary coordinates}
\[
P_{\bs,k} := \frac{\Lambda_{s_k} (\widetilde{B}_1, \widetilde{B}_2^k)} {\Lambda_{s_k} (\widetilde{B}_1, \widetilde{B}_2^{k-1})}, 
\]
where $\widetilde{B}_1$ is an arbitrary lift of $\widehat{B}_1\in \A_G$ to $\A_{\widetilde{G}}$ and 
$\Lambda_{s}:  \A_{\widetilde{G}}\times \A_{\widetilde{G}}\to \bA^1$ is the unique $\widetilde{G}$-invariant rational function such that 
\[
\Lambda_{s}(h.[U^+], \vw.[U^+])=h^{\varpi_s} \in \bG_m.
\]
Note that $P_{\bs,k}$ does not depend on the choice of the lifts $\widetilde{B}_1, \widetilde{B}_2$ and it gives a well-defined $\bG_m$-valued regular function on $\Conf_3 \P_G$. See also \cref{l:flaglift}. 

Let $\beta_k^\bs := r_{s_N} \dots r_{s_{k+1}}(\alpha_{s_k}^\vee)$ be a sequence of coroots associated with $\bs$. For each $s \in S$, there exists a unique $k = k(s)$ such that $\beta_k^\bs = \alpha_{s}^\vee$. Then we set 
\[
X_{s \choose -\infty} :=P_{\bs,k(s)}. 
\]
\begin{dfn}\label{d:GS coord}
The rational
functions $X_{s \choose i}$ ($s \in S, ~i=-\infty, 0,1,\dots,n^{s}(\bs)$) are called the \emph{Goncharov--Shen coordinates} (\emph{GS coordinates} for short) on $\Conf_3 \P_G$, associated with the reduced word $\bs$. When we want to emphasize the dependence on the reduced word $\bs$, we write $X_{s \choose i}=:X_{s \choose i}^{\bs}$.
\end{dfn}
Conversely, we can construct an embedding 
\begin{align}
\psi_{\bs}: \bG_m^{I_\infty(\bs)} \to \Conf_3 \P_G\label{eq:GS coord}
\end{align}
from given set of GS coordinates, where $I_\infty(\bs):=\{(s,i) \mid s \in S, ~i=-\infty, 0,1,\dots,n^{s}(\bs)\}$. 
If $G=PGL_{n+1}$ and the reduced word $\bs=\bs_\std(n)$ is the one defined inductively by 
\begin{align}\label{eq:std_word}
    \bs_\std(n) = (1,2,\dots,n)\bs_\std(n-1), \quad
    \bs_\std(1) = (1),
\end{align}
then the GS coordinates are nothing but the Fock--Goncharov coordinates introduced in \cite[Section 9]{FG03}.


\begin{lem}[{\cite[Lemma 9.2]{GS19}}]\label{l:H3-action on GS coord}
Let $(k_1, k_2, k_3)\in H^3$ and denote the action of $(k_1, k_2, k_3)$ on $\Conf_3\P_G$ described in \cref{l:H3-action} by $\alpha_{k_1, k_2, k_3}:  \Conf_3\P_G\to \Conf_3\P_G$. Then for $s\in S$, we have
\begin{align*}
\alpha_{k_1, k_2, k_3}^{\ast}X_{s\choose i}=\begin{cases}
k_1^{-\alpha_s}X_{s\choose 0}&\text{if } i=0,\\
k_3^{-\alpha_{s^{\ast}}}X_{s\choose n^s(\bs)}&\text{if } i=n^s(\bs),\\
k_2^{-\alpha_s}X_{s\choose -\infty}&\text{if } i=-\infty,\\
X_{s\choose i}&\text{otherwise}.
\end{cases} 
\end{align*}
\end{lem}
\begin{proof}
The first three equalities are given in \cite[Lemma 9.2]{GS19}. The last one straightforwardly follows from the definition of the $H^3$-action and $W_s$. 
\end{proof}
For $s \in S$, the GS coordinate $X_{s \choose i}$ for $i=0,1,\dots,n^{s}(\bs)$ (resp. $i=-\infty$) is assigned to the vertex $v_i^s$ (resp. $y_s$) of the weighted quiver $\widetilde{\bJ}^+(\bs)$. See \cref{sec:quivers}. The indices in the subset $I_\uf(\bs):=\{(s,i) \in I_\infty(\bs) \mid s \in S,~0 < i < n^s(\bs)\}$ are declared to be unfrozen. Then the pair $\widetilde{\sfS}(\bs):=(\widetilde{\bJ}^+(\bs),(X_{s \choose i})_{(s,i) \in I_\infty(\bs)})$ forms a seed in the ambient field $\cF=\Rat(\Conf_3 \P_G)$. 

\begin{thm}[Goncharov--Shen {\cite[Theorem 7.2]{GS19}}]
The seeds $\widetilde{\sfS}(\bs)$ associated with any two reduced words $\bs$ of the longest element $w_0$ are mutation-equivalent to each other. Hence the collection $(\widetilde{\sfS}(\bs))_{\bs}$ is a cluster Poisson atlas (\cref{d:cluster atlas}) on $\Conf_3 \P_G$.
\end{thm}


\begin{rem}\label{rem:GS coordinates_partial_pinnings}
\begin{enumerate}
    \item Note that the frozen coordinates $X_{s \choose 0}$, $X_{s \choose n^s(\bs)}$ and $X_{s \choose -\infty}$ depend only on one of the three pinnings. Moreover, observe from \cref{l:H3-action on GS coord} that these frozen coordinates are uniquely distinguished by their degrees with respect to the $H^3$-action among the GS coordinates associated with $\bs$. 
    \item On the other hand, the unfrozen coordinates $X_{s\choose i}$ for $(s,i) \in I_\uf(\bs)$ only depend on the underlying flags $(B_1,B_2,B_3)$. Hence we have the following birational charts for the configuration spaces with some of the pinnings dropped:
\begin{align*}
    &\bG_m^{I_\infty(\bs)\setminus \{(s,-\infty)\mid s \in S\}} \to [G\backslash \{(B_1,B_2,B_3;p_{12},p_{31})\}], \\ 
    &\bG_m^{I_\infty(\bs)\setminus \{(s,0),(s,-\infty)\mid s \in S\}} \to [G\backslash \{(B_1,B_2,B_3;p_{31})\}], 
\end{align*}
and so on. Here a pair of flags over which no pinning is assigned is not required to be generic. 
\end{enumerate}
\end{rem}

\begin{rem}\label{r:GScoord}
The above definition of coordinates is the same as the original one given in \cite{GS19}. 
It can be verified as follows. 
For each $s \in S$, take the unique flag $B_{2,s}^k$ such that $w(B_2^k,B_{2,s}^k) = w_0r_{s}$ and $w(B_{2,s}^k, B_1) = r_{s}$. 
Then actually we have $B_{2,s}^{k-1} = B_{2,s}^k$ whenever $s_k \neq s$ \cite[Lemma 7.9]{GS19}. See also \cref{r:additional flags}. We collect all the distinct flags among $B_{2,s}^k$ and relabel them as $B_{s \choose i}$, $(s,i) \in I(\bs) \subset I_\infty(\bs)$. 
Then the triple $\left(\widehat{B}_1,B_{s \choose i}, B_{s \choose i+1}\right)$ determines a configuration of $SL_2$-flags, namely an element of $\Conf^*(\A_{SL_2},\B_{SL_2},\B_{SL_2})$. See \cite[(291)]{GS19}. 
Then by \cite[Proposition 7.10]{GS19}, we have
\[
W_s \left(\widehat{B}_1, B_2^{k(s,i)}, B_2^{k(s,i+1)} \right)=W\left(\widehat{B}_1,B_{s \choose i}, B_{s \choose i+1}\right),
\]
where the right-hand side is the potential of a configuration of $SL_2$-flags.
\end{rem}

\subsection{Coordinate expressions of basic Wilson lines}\label{subsec:coordinate comparison}
\begin{NB}
\begin{prop}\label{p:coweight-lus}
\begin{enumerate}
    \item For $b^+ \in G^{e, v}$, let $b^+=hu^+$ $(h \in H, u^+ \in U^+_*)$ be its triangular decomposition. For any reduced word $\bs=(s_1,\dots,s_l)$ of $v$, the Lusztig parameters $(t_k)_{k=1,\dots,l}$ of $u^+$ and the coweight parameters $(x_i^s)_{s \in S,i=0,\dots,n^s(\bs)}$ of $b^+$ are related by
    \[
    t_{k(s,i)}^{-1} = x_{i}^s \dots x_{n^s(\bs)}^s
    \]
    for $s \in S$ and $i=1,\dots,n^s(\bs)$.
    \item For $b^- \in B^-_u$, let $b^-=u^-h$ $(h \in H, u^- \in U^-_*)$ be its triangular decomposition. For any reduced word $\bs=(s_1,\dots,s_l)$ of $u$, the Lusztig parameters $(t_k)_{k=1,\dots,l}$ of $u^-$ and the coweight parameters $(x_i^s)_{s \in S,i=0,\dots,n^s(\bs)}$ of $b^+$ are related by
    \[
    t_{k(s,i)}^{-1} = x_0^s \dots x_{i-1}^s
    \]
    for $s \in S$ and $i=1,\dots,n^s(\bs)$.
\end{enumerate}
\end{prop}

\begin{proof}
For the first part, let us write 
\begin{align*}
    b^+ = ev^+_\bs(\mathbf{x}) = \left(\prod_{s=1}^n H^s(x_s) \right) \cdot \overrightarrow{\prod_{k=1,\dots,l}} (\mathbb{E}^{s_k}H^{s_k}(x_{n+k})).
\end{align*}
Then using the relation $H^s(a)^{-1}x_u(b)H^s(a) = x_u(a^{-\delta_{su}}b)$ (and recalling $x_u(1) = \mathbb{E}^u$), we can move all the $H$'s to the left. The result takes the form
\begin{align*}
    b^+ = \left(\prod_{s=1}^n H^s((x_0^s \dots x_{n^s(\bs)}^s)^{-1}) \right) \cdot x_{s_1}(t_1) \dots x_{s_l}(t_l),
\end{align*}
where we have $t_{k(s,i)} = (x_1^s \dots x_{i}^s)^{-1}$. The second part is similarly proved by using the relation $H^s(a)y_u(b)H^s(a)^{-1} = y_u(a^{-\delta_{su}}b)$ and moving all the $H$'s to the right.
\end{proof}
\end{NB}
In this section, we give an expression of the basic Wilson lines $b_L,b_R$ (\cref{d:basic-Conf3}) in terms of the GS coordinates, relating the coweight parametrizations on the double Bruhat cells $B^+_*:=G^{e,w_0}$, $B^-_*:=G^{w_0,e}$ with the GS coordinates. 
The index set $I(\bs)$ introduced in \eqref{eq:relabeling_set} is naturally regarded as a subset of $I_\infty(\bs)$.

\begin{thm}[cf. {\cite[Lemma 7.29]{GS19}}]\label{t:coweight-GS}
For each reduced word $\bs$ of $w_0 \in W(G)$ we have 
\begin{align*}
    \psi_{\bs}^*b_L = ev^+_{\bs}, \quad \psi_{\bs}^*b_R = ev^-_{\bs_\op^\ast}\circ \iota^*,
\end{align*}
where $\iota^*:\bG_m^{I(\bs)} \to \bG_m^{I(\bs_\op^\ast)}$ is the isomorphism induced by \eqref{eq:iota-involution}.  
\end{thm}
Below we give a proof of this theorem based on the standard configuration (\cref{l:standardconfig}). Let us write
\begin{align}
    \wC_3^{-1}\circ \psi_\bs=(\sfh_1(\bX),\sfh_2(\bX),\sfu_+(\bX)): \bG_m^{I_\infty(\bs)} \to H \times H \times U^+_*.\label{eq:def_of_hu}
\end{align}
Then from \cref{c:LR}, we have 
\begin{align*}
    \psi_\bs^*b_L=\sfu_+(\bX)h_2(\bX), \quad \psi_\bs^*b_R = \vw^{-1}(\sfu_+(\bX)h_2(\bX))^{-1}\vw.
\end{align*}
We are going to compute the functions $\sfu_+(\bX)$ and $\sfh_2(\bX)$. 
 
In the following, we use the short-hand notations $x_{[i\ j]}^\bs(\mathbf{t}):=x_{s_i}(t_i) \dots x_{s_j}(t_j)$ and $y_{[i\ j]}^{\bs}(\mathbf{t}):= y_{s_i}(t_i) \dots y_{s_j}(t_j)$ for a reduced word $\bs=(s_1,\dots,s_N)$ of $w_0 \in W(G)$ and $1 \leq i < j \leq N$. 

\begin{lem}\label{l:lus-GS}
For a configuration $\wC=\wC_3(h_1,h_2,u_+) \in \Conf_3 \P_G$ and its representative as in \cref{l:standardconfig}, write $u_+ = x_{s_1}(t_1)\dots x_{s_N}(t_N) = x_{[1\ N]}^{\bs}(\mathbf{t})$ using the Lusztig coordinates associated with $\bs$. Let $(X_{s \choose i}) \in \bG_m^{I_\infty(\bs)}$ be the GS coordinates of $\wC$ associated with $\bs$. 
Then we have the following:
\begin{enumerate}
\item
$B_2^k = x_{[1\ k]}^{\bs}(\mathbf{t}) B^-$.
\item
For each $s \in S$ and $i = 1,\dots,n^s(\bs)$, we have
\begin{align}\label{eq:Lus-GS}
    t_{k(s,i)} = X_{s \choose 0} \dots X_{s \choose i-1}.
\end{align}
Here $k(s,i)$ is the $i$-th number $k$ with $s_k = s$ in $\bs$ from the left. 
\end{enumerate} 
\end{lem}
Substituting \eqref{eq:Lus-GS} into $x_{[1,N]}^\bs(\mathbf{t})$, we get an expression of the function $\sfu_+(\bX)$.

\begin{proof}
To check that the right-hand side of the first statement indeed gives $B_2^k$, let us compute
\begin{align*}
    w(x_{[1\ k-1]}^{\bs}(\mathbf{t}) B^-,x_{[1\ k]}^{\bs}(\mathbf{t}) B^-) 
    = w(B^-, x_{s_k}(t_k) B^-) = r_{s_k^\ast},
\end{align*}
where we used the relation $x_s(t) = y_s(t^{-1})\alpha_s^\vee(t) \overline{r}_s^
{-1}y_s(t^{-1})$. 
Then the uniqueness statement of \cref{c:flagchain} and an induction on $k$ implies $B_2^k = x_{[1\ k]}^{\bs}(\mathbf{t}) B^-$.

To prove the second statement, we compute
\begin{align*}
[\widehat{B}_1,B_2^{k(s,i)}, B_2^{k(s,i+1)}]
&= \left[[U^+], x_{[1\ k(s,i)]}^{\bs}(\mathbf{t}) B^-, x_{[1\ k(s,i+1)]}^{\bs}(\mathbf{t}) B^-\right] \\
&= \left[[U^+], B^-, x_{[k(s,i)+1\ k(s,i+1)]}^{\bs}(\mathbf{t}) B^- \right]. 
\end{align*}
Thus we have $W_{s}(\widehat{B}_1,B_2^{k(s,i)}, B_2^{k(s,i+1)}) = \chi_{s}(x_{[k(s,i)+1\ k(s,i+1)]}^\bs(\mathbf{t})) = t_{k(s,i+1)}$ by \cref{l:char-lus}. A similar computation shows that $X_{s \choose 0} = W_s(\widehat{B}_1,B_2, B_2^{k(s,1)}) = t_{k(s,1)}$ and we get $t_{k(s,i)} = X_{s \choose 0} \dots X_{s \choose i-1}$ by induction on $i$.
\end{proof}

\begin{lem}\label{l:Cartan-GS}
We have $\sfh_2(\bX) = \prod_{s \in S} H^{s} (\mathbb{X}_s)$, where $\mathbb{X}_s := \prod_{i=0}^{n^s(\bs)} X_{s \choose i}$.
\end{lem}

\begin{proof}
Again fix a configuration $\wC=\wC_3(h_1,h_2,u_+)$. 
Recall that the pinning $p_{12} = p_\std$ corresponds to the lift $\widehat{B}_1 = [U^+]$ of $B_1$. On the other hand, the pinning $p_{31}^\ast = u_+ h_2.p_\std$ corresponds to the lift $\widehat{B}'_1 = h_2. [U^+]$. Then we can compute:
\begin{align*}
    \left[\widehat{B}'_1,B_2^{k(s,n^s(\bs)-1)}, B_2^{k(s,n^s(\bs))} \right] 
    &=\left[h_2. [U^+],B^-, x_{[k(s,n^s(\bs)-1)+1\ k(s,n^s(\bs))]}^{\bs}(\mathbf{t}). B^-\right] \\
    &=\left[[U^+],B^-, \mathrm{Ad}_{h_2}^{-1}(x_{[k(s,n^s(\bs)-1)+1\ k(s,n^s(\bs))]}^{\bs}(\mathbf{t})). B^-\right]. 
\end{align*}
With a notice that $\mathrm{Ad}_{h_2}^{-1}(x_{s}(t_{k(s,n^s(\bs))})) =  x_s(h_2^{-\alpha_{s}}t_{k(s,n^s(\bs))})$, 
we get 
\begin{align*}
    X_{s \choose n^s(\bs)} = W_s \left( \widehat{B}'_1,B_2^{k(s,n^s(\bs)-1)}, B_2^{k(s,n^s(\bs))} \right)^{-1} =   h_2^{\alpha_{s}}t_{k(s,n^s(\bs))}^{-1}.
\end{align*}
Hence from \cref{l:lus-GS} we get $h_2^{\alpha_{s}} = t_{k(s,n^s(\bs))} X_{s \choose n^s(\bs)} = \mathbb{X}_s$, which implies $h_2 = \prod_{s \in S} H^{s} (\mathbb{X}_s)$.
\end{proof}

\begin{proof}[Proof of \cref{t:coweight-GS}]
From the definition of the one-parameter subgroup $x_u$, we have the relation $H^s(a)x_u(b)H^s(a)^{-1} = x_u(a^{\delta_{su}}b)$. In particular $x_s(t) = H^s(t)\mathbb{E}^s H^s(t)^{-1}$. Combining with \cref{l:lus-GS,l:Cartan-GS}, we get
\begin{align*}
    \psi_\bs^*b_L
    =\sfu_+(\bX) \sfh_2(\bX) 
    = \dprod_{k=1,\dots,N}(H^{s_k}(t_k) \mathbb{E}^{s_k} H^{s_k}(t_k)^{-1}) \cdot \prod_{s \in S} H^{s} (\mathbb{X}_s), 
\end{align*}
where each $t_k=t_k(\bX)$ is a monomial given by \eqref{eq:Lus-GS}. 
Since $H^s$ commutes with $\mathbb{E}^u$ for $s \neq u$, we can obtain the following expression by the relabeling as in \eqref{eq:relabeling}: 
\begin{align*}
    \psi_\bs^*b_L = \prod_{s \in S} H^{s} (t_{k(s,1)})\dprod_{\substack{s\in S\\ i=1,\dots,n^s(\bs)}}( \mathbb{E}^{s} H^{s}(t_{k(s,i)})^{-1}H^{s}(t_{k(s,i+1)}))\prod_{s \in S} H^{s} (\mathbb{X}_s)
\end{align*}
where $t_{k(s, n^s(\bs)+1)}:=1$ for $s\in S$. Note that it already has the form of the coweight parametrization $ev^+_{\bs}(\mathbf{x})$, where the parameter $\mathbf{x}=(x_i^s)_{s \in S,~i=0,\dots,n^s(\bs)}$ is computed as follows:
\begin{align*}
    x_i^s = 
    \begin{cases}
    t_{k(s,1)}=X_{s \choose 0} & (i=0), \\
    t_{k(s,i)}^{-1}t_{k(s,i+1)} = X_{s \choose i} & (i=1,\dots,n^s(\bs)-1), \\
    t_{k(s,n^s(\bs))}^{-1}\mathbb{X}_s = X_{s \choose n^s(\bs)} & (i=n^s(\bs)),
    \end{cases}
\end{align*}
where we used \cref{l:lus-GS} for the second steps. 
Thus we have $\psi_\bs^*b_L(\bX) = ev_\bs^+(\bX)$, as desired. The second statement follows from \cref{l:ev_involution}.

\end{proof}

\begin{rem}\label{r:additional flags}
Similarly to the proof of \cref{l:lus-GS} we can compute the flags defined in \cref{r:GScoord}, as follows:
\begin{align*}
B_{2,s}^k &= x_{[1\ k]}^{\bs}(\mathbf{t}) \vw \overline{r}_{s^\ast}. B^-, \\
B_{s \choose k} &= x_{[1\ k]}^{\bs}(\mathbf{t}) \vw \overline{r}_{s_k^\ast}. B^-.
\end{align*}
\end{rem}

\subsection{Primary coordinates in the standard configuration}
Let $\bs=(s_1,\dots, s_{N})$ be a reduced word of $w_0\in W(G)$. The following computation of the primary coordinates $P_{\bs,k}$ in terms of the standard configuration will be used in \cref{sec:positivity}.

\begin{lem}\label{l:flaglift}
For a configuration $\wC=\wC_3(h_1,h_2,u_+) \in \Conf_3 \P_G$ and its representative as in \cref{l:standardconfig}, write $u_+=x_{s_1}(t_1)\dots x_{s_N}(t_N) = x_{[1\ N]}^{\bs}(\mathbf{t})$ using the Lusztig coordinates associated with $\bs$. Then we have the following: 
\begin{itemize}
    \item[(1)] $\widetilde{B}_2=\widetilde{h}_1\vw.[U^+]$ and $\widetilde{B}_2^k=x_{[1\ k]}^{\bs}(\mathbf{t})\widetilde{h}_1^{(k)}\vw.[U^+]$ 
    for $k=0,1,\dots, N$. 
    Here $\widetilde{h} \in \widetilde{G}$ is a lift of $h_1$, and
    \[
    \widetilde{h}_1^{(k)}:=r_{s_k}\cdots r_{s_1}(\widetilde{h}_1)\prod_{j=1}^k r_{s_k}\cdots r_{s_{j+1}}(\alpha_{s_j}^{\vee}(t_j)).
    \]
    \item[(2)] $P_{\bs,k}(\wC_3(h_1, h_2, u_+))=t_k((\widetilde{h}_1^{(k)})^{-\alpha_{s_k}})$. 
\end{itemize}
\end{lem}
Observe that the decorated flags given in (1) are indeed projected to those given in \cref{l:lus-GS} (1). Also note that the right-hand side of (2) does not depend on the choice of the lift $\widetilde{h}_1$. 
\begin{proof}
Since the second component in the representative of $\wC$ is $\phi'(u_+)h_1\vw.p_\std$, the decorated flag $\widetilde{B}_2$ must be a lift of 
\[
\phi'(u_+)h_1\vw.[U^+]=h_1\vw \Ad_{(h_1\vw)^{-1}}(\phi'(u_+)).[U^+]=h_1\vw .[U^+].
\]
Such an element is written as $\widetilde{h}_1\vw.[U^+]$ for some lift of $h_1$ to $\widetilde{G}$, which proves the first statement of (1). Set $\widetilde{B}^{(k)}:=x_{[1\ k]}^{\bs}(\mathbf{t})\widetilde{h}_1^{(k)}\vw.[U^+]$. In order to show the second statement of (1), it suffices to see that 
\begin{itemize}
    \item[(i)] $\widetilde{B}^{(0)}=\widetilde{B}_2$,
    \item[(ii)] $w(\widetilde{B}^{(k)}, \widetilde{B}^{(k-1)})=r_{s_k^{\ast}}$ and  $h(\widetilde{B}^{(k)}, \widetilde{B}^{(k-1)})=1$ for $k=1,\dots, N$.
\end{itemize}
The statement (i) is immediate from the definition. In $\Conf_2 \A_{\widetilde{G}}$, we have 
\begin{align*}
    &\left[\widetilde{B}^{(k)}, \widetilde{B}^{(k-1)}\right]\\
    &=\left[x_{s_k}(t_k)\widetilde{h}_1^{(k)}\vw.[U^+], \widetilde{h}_1^{(k-1)}\vw.[U^+]\right]\\
    &=\left[y_{s_k}(t_k^{-1})\alpha_{s_k}^{\vee}(t_k)\overline{r}_{s_k}^{-1}y_{s_k}(t_k^{-1})\widetilde{h}_1^{(k)}\vw.[U^+], \widetilde{h}_1^{(k-1)}\vw.[U^+]\right]\\
    &=\left[\alpha_{s_k}^{\vee}(t_k)\overline{r}_{s_k}^{-1}\widetilde{h}_1^{(k)}\vw\Ad_{(\widetilde{h}_1^{(k)}\vw)^{-1}}(y_{s_k}(t_k^{-1})).[U^+], \widetilde{h}_1^{(k-1)}\vw\Ad_{(\widetilde{h}_1^{(k-1)}\vw)^{-1}}(y_{s_k}(-t_k^{-1})).[U^+]\right]\\
    &=\left[\alpha_{s_k}^{\vee}(t_k)r_{s_k}(\widetilde{h}_1^{(k)})\overline{r}_{s_k}^{-1}\vw.[U^+], \widetilde{h}_1^{(k-1)}\vw.[U^+]\right].
\end{align*}
Moreover we have 
\begin{align}
\alpha_{s_k}^{\vee}(t_k)r_{s_k}(\widetilde{h}_1^{(k)})&=\alpha_{s_k}^{\vee}(t_k)r_{s_{k-1}}\cdots r_{s_1}(\widetilde{h}_1)\alpha_{s_k}^{\vee}(t_k^{-1})\prod_{j=1}^{k-1} r_{s_{k-1}}\cdots r_{s_{j+1}}(\alpha_{s_j}^{\vee}(t_j))\notag \\
&=r_{s_{k-1}}\cdots r_{s_1}(\widetilde{h}_1)\prod_{j=1}^{k-1} r_{s_{k-1}}\cdots r_{s_{j+1}}(\alpha_{s_j}^{\vee}(t_j))=\widetilde{h}_1^{(k-1)}. \label{eq:hk}
\end{align}
Thus we get
\[
\left[\widetilde{B}^{(k)}, \widetilde{B}^{(k-1)}\right]=\left[\overline{r}_{s_k}^{-1}\vw.[U^+], \vw.[U^+]\right]=\left[[U^+], \overline{r}_{s_k^{\ast}}.[U^+]\right],
\]
which shows (ii). 

For the computation of $P_{\bs,k}(\wC_3(h_1, h_2, u_+))$, we may take a lift of the first flag $\widetilde{B}_1$ as $[U^+]$. Then, 
\begin{align*}
P_{\bs,k}(\wC_3(h_1, h_2, u_+))&=\frac{\Lambda_{s_k} ([U^+], 
x_{[1\ k]}^{\bs}(\mathbf{t})\widetilde{h}_1^{(k)}\vw.[U^+])}{\Lambda_{s_k} ([U^+], x_{[1\ k-1]}^{\bs}(\mathbf{t})\widetilde{h}_1^{(k-1)}\vw.[U^+])}\\
&=\frac{\Lambda_{s_k} ((\widetilde{h}_1^{(k)})^{-1}.[U^+], 
\vw.[U^+])}{\Lambda_{s_k} ((\widetilde{h}_1^{(k-1)})^{-1}.[U^+], \vw.[U^+])}\\
&=\frac{\Lambda_{s_k} ((\widetilde{h}_1^{(k)})^{-1}.[U^+], 
\vw.[U^+])}{\Lambda_{s_k} ((\alpha_{s_k}^{\vee}(t_k)r_{s_k}(\widetilde{h}_1^{(k)}))^{-1}.[U^+], \vw.[U^+])} \quad (\text{by }\eqref{eq:hk})\\
&=t_k\frac{(\widetilde{h}_1^{(k)})^{-\varpi_{s_k}}}{r_{s_k}(\widetilde{h}_1^{(k)})^{-\varpi_{s_k}}}=t_k(\widetilde{h}_1^{(k)})^{-\varpi_{s_k}+r_{s_k}(\varpi_{s_k})}=t_k(\widetilde{h}_1^{(k)})^{-\alpha_{s_k}}
\end{align*}
as desired.
\end{proof}
\begin{cor}\label{c:x-infty}
For $s\in S$, we have
\[
(\wC_3^{\ast}X_{s \choose -\infty})(h_1,h_2, u_+)=t_{k(s)}h_1^{\alpha_{s^{\ast}}}\prod_{j=1}^{k(s)} t_j^{\langle r_{s_{1}}\cdots r_{s_{j}}(\alpha_{s_j}^{\vee}), \alpha_{s^{\ast}}\rangle}
\]
for $(h_1, h_2, u_+)\in H\times H\times U^+_*$ with $u_+=x_{[1\ N]}^{\bs}(\mathbf{t})$. Here recall that $k(s)$ is determined by $r_{s_N} \dots r_{s_{k(s)+1}}(\alpha_{s_{k(s)}}^\vee) = \alpha_{s}^\vee$. 
\end{cor}
\begin{proof}
By the definition of $X_{s \choose -\infty}$ and \cref{l:flaglift} (2), the desired statement follows from the following calculation: 
\begin{align*}
r_{s_{k(s)}}\cdots r_{s_1}(h_1)^{-\alpha_{s_{k(s)}}}&=h_1^{r_{s_1}\cdots r_{s_{k(s)-1}}(\alpha_{s_{k(s)}})}=h_1^{-w_0r_{s_N} \dots r_{s_{k(s)+1}}(\alpha_{s_{k(s)}})}
=h_1^{-w_0\alpha_{s}}=h_1^{\alpha_{s^{\ast}}}.
\end{align*}
\end{proof}
The results in \cref{l:lus-GS,l:Cartan-GS} and \cref{c:x-infty} gives explicit forms of $\mathsf{h}_1,\mathsf{h}_2:  \mathbb{G}_m^{I_\infty(\bs)}\to H$ and $\mathsf{u}_+:  \mathbb{G}_m^{I_\infty(\bs)}\to U^+_{\ast}$ as follows: 
\begin{lem}\label{l:GScoord_standard}
	 Write $t_{k(s,i)} := X_{s \choose 0} \dots X_{s \choose i-1}$ for $(i, s)\in I(\bs)$ (recall the notation in \eqref{eq:Lus-GS}). Then we have the following: 
\begin{itemize}
	\item[(1)] $\mathsf{h}_1(\mathbf{X})=\prod_{s \in S} H^{s} (\widetilde{X}_s)$, where 
	\begin{align*}
		\widetilde{X}_s=X_{s^{\ast} \choose -\infty}t_{k(s^{\ast})}^{-1}\prod_{k=1}^{k(s^{\ast})} t_k^{\langle r_{s_{1}}\cdots r_{s_{k-1}}(\alpha_{s_k}^{\vee}), \alpha_{s}\rangle}
	\end{align*}
	where $k(s^{\ast})$ is determined by $r_{s_1} \dots r_{s_{k(s^{\ast})-1}}(\alpha_{s_{k(s^{\ast})}}^\vee) = \alpha_{s}^\vee$.
	\item[(2)] $\mathsf{h}_2(\mathbf{X})=\prod_{s \in S} H^{s} (\mathbb{X}_s)$, where $\mathbb{X}_s := \prod_{i=0}^{n^s(\bs)} X_{s \choose i}$. 
	\item[(3)] $\mathsf{u}_+(\mathbf{X})=x_{s_1}(t_1)\dots x_{s_N}(t_N)$.
\end{itemize}
\end{lem}

\subsection{Goncharov--Shen coordinates on $\P_{G,\Sigma}$ via amalgamation}\label{ss:GS_coord}

\smallskip
\paragraph{\textbf{Triangle case.}} 
Let $\Sigma=T$ be a triangle. Recall that we have the isomorphism \eqref{eq:moduli_triangle} determined by choosing a distinguished marked point $m$, or a dot. Let us choose a reduced word $\bs$ of $w_0$. 
Then we define the GS coordinates on $\P_{G,T}$ to be $X_{s \choose i}^{(T,m,\bs)}:=f_{m}^*X_{s \choose i}^{\bs}$, where $X_{s \choose i}^{\bs}$ denote the GS coordinates on $\Conf_3 \P_G$ associated with the reduced word $\bs$. 
The accompanying quiver $Q_{T,m,\bs}$ is defined to be the quiver $\widetilde{\bJ}^+(\bs)$ placed on $T$ so that for $s \in S$,
\begin{itemize}
    \item the vertices $v_0^s$ lie on the edge $E$,
    \item the vertices $v_{n^s(\bs_T)}^s$ lie on the edge $E^R$, and 
    \item the vertices $y_s$ lie on the edge $E^L$.
\end{itemize}
Here we use the notation in \eqref{eq:moduli_triangle_explicit} for the edges. 
See \cref{fig:triangle_quiver} for an example. Here the isotopy class of the embedding of the quiver $\widetilde{\bJ}^+(\bs)$ into the triangle $T$ relative to the boundary intervals is included in the defining data of $Q_{T,m,\bs}$. 
Then the pair $\sfS_{(T,m,\bs)}:=\left(Q_{T,m,\bs},(X_{s \choose i}^{(T,m,\bs)})_{(s,i) \in I_\infty(\bs)}\right)$ is a seed in $\Rat(\P_{G,T})$. 
Recall the cyclic shift given in \cref{l:cyclic shift}.

\begin{thm}[Goncharov--Shen {\cite[Theorem 5.11]{GS19}}]\label{t:cluster_structure_triangle}
The cyclic shift $\mathcal{S}_3$ is realized as a sequence of cluster transformations. 
In particular, the seeds $\sfS_{(T,m,\bs)}$ for any choice of a dot $m$ and a reduced word $\bs$ of $w_0$ are mutation-equivalent to each other. Hence the collection $\{\sfS_{(T,m,\bs)}\}_{m,\bs}$ is a cluster Poisson atlas (\cref{d:cluster atlas}) on $\P_{G,T}$.
\end{thm}

\begin{figure}[ht]
\scalebox{0.9}{
\begin{tikzpicture}
\begin{scope}[>=latex]
{\begin{scope}[scale=1.25]
\draw(-4,0) node[left]{\scalebox{0.9}{$m'$}} -- (4,0) node[right]{\scalebox{0.9}{$m''$}} -- (0,4) node[above]{\scalebox{0.9}{$m$}} -- (-4,0) --cycle;
\draw(0,3.75) node{$\ast$};
\foreach \i in {0,1,2,3}
	\dnode{0.67*\i-1,3} coordinate(V3\i) node[above =0.2em]{$v_{\i}^3$};
\foreach \i in {0,1,2,3}
{
    \draw(1.34*\i-2,2) circle(2pt) coordinate(V2\i) node[above right]{$v_{\i}^2$};      
    \draw(2*\i-3,1) circle(2pt) coordinate(V1\i) node[above right]{$v_{\i}^1$};
}
{\color{myblue}
    \foreach \j in {1,2}
        \draw(2*\j-4,0) circle(2pt) coordinate(Y\j) node[below]{$y_\j$};
    \dnode{2,0} coordinate(Y3) node[below=0.2em]{$y_3$};
}
\foreach \i in {0,1,2}
    {
    {\color{red}
    \pgfmathsetmacro{\j}{\i+1}
    \qsarrow{V3\j}{V3\i};
     \qarrow{V2\j}{V2\i};
     \qsarrow{V1\j}{V1\i};
    }
    }
\qshdarrow{V20}{V30}
\qstarrow{V30}{V21}
\qsharrow{V21}{V31}
\qstarrow{V31}{V22}
\qsharrow{V22}{V32}
\qstarrow{V32}{V23}
\qshdarrow{V23}{V33}

\qdarrow{V10}{V20}
\qarrow{V20}{V11}
\qarrow{V11}{V21}
\qarrow{V21}{V12}
\qarrow{V12}{V22}
\qarrow{V22}{V13}
\qdarrow{V13}{V23}
{\color{myblue}
\qarrow{V10}{Y1}
\qarrow{Y1}{V11}
\qarrow{V11}{Y2}
\qarrow{Y2}{V12}
\qsarrow{V32}{Y3}
\qsarrow{Y3}{V33}
\qdarrow{Y2}{Y1}
\qstdarrow{Y3}{Y2}
}
\end{scope}}
\end{scope}
\end{tikzpicture}}
    \caption{Placement of the weighted quiver $\widetilde{\bJ}^+((123)^3)$ with $\lie=C_3$ on a dotted triangle $T$.}
    \label{fig:triangle_quiver}
\end{figure}
By \cref{rem:partial_pinnings} (1), the frozen coordinates are distinguished by their degrees with respect to the triple of $H$-actions $\alpha_E$ associated to each edge $E$ of $T$. For each edge $E$, let $X_{s_E}^{(T,m,\bs)}$ denote the unique one among the cooridnates $X_{s \choose i}^{(T,m,\bs)}$ which has the degree $-\alpha_s$ for the action $\alpha_E$. 

\paragraph{\textbf{General case.}} 
Let us consider a general marked surface $\Sigma$. 
A \emph{decorated triangulation} of $\Sigma$ consists of the following data $\bD = (\Delta_*, \bs_\Delta)$:
\begin{itemize}
\item
An \underline{oriented} dotted triangulation $\Delta_*$, which is a dotted triangulation $\Delta_*$ of $\Sigma$ equipped with an orientation for each edge \footnote{We need the orientation only to fix a bijection between the GS coordinates on an edge and the index set $S$.}. Let $m_T$ denote the marked point corresponding to the dot on a triangle $T$.
\item
A choice $\bs_\Delta = (\bs_T)_T$ of reduced words $\bs_T$ of the longest element $w_0 \in W(G)$, one for each triangle $T$ of $\Delta_*$.
\end{itemize}
We call $\Delta_*$ the \emph{underlying dotted triangulation} of $\bD$. 

Let $q_\Delta: \widetilde{\P_{G,\Sigma}^\Delta}=\prod_{T \in t(\Delta)} \P_{G,T} \to \P_{G,\Sigma}^\Delta$ be the gluing morphism with respect to the underlying triangulation $\Delta$. For $T \in t(\Delta)$, let $\mathrm{pr}_T:\widetilde{\P_{G,\Sigma}^\Delta} \to \P_{G,T}$ be the projection. 
Then the \emph{GS coordinate system} on $\P_{G,\Sigma}^\Delta \subset \P_{G,\Sigma}$ associated with $\bD$ is the collection
\begin{align*}
    \mathbf{X}_{\bD}:=\{ X_{s \choose i}^{(T;\bD)}\}_{\subalign{&(s,i) \in I_\uf(\bs)\\&T \in t(\Delta)}} \cup \{X^{(E;\bD)}_s\}_{{\subalign{&s \in S \\&E \in e(\Delta)}}}
\end{align*}
of rational functions, which are characterized as follows:
\begin{itemize}
    \item For $T \in t(\Delta)$ and $(s,i) \in I_\uf(\bs)$, 
    \begin{align*}
        q_\Delta^\ast X_{s \choose i}^{(T;\bD)} = X_{s \choose i}^{(T,m_T,\bs_T)} \circ \mathrm{pr}_T.
    \end{align*}
    \item For an interior edge $E \in e(\Delta)$ and $s \in S$, 
    \begin{align*}
        q_\Delta^\ast X_{s}^{(E;\bD)} = \left( X_{s_{E^L}}^{(T^L,m_{T^L},\bs_{T^L})}\circ \mathrm{pr}_{T^L}\right) \cdot \left( X_{s^\ast_{E^R}}^{(T^R,m_{T^R},\bs_{T^R})}\circ \mathrm{pr}_{T^R}\right).
    \end{align*}
    Here $T^L$ (resp, $T^R$) is the triangle that lie on the left (resp. right) of $E$, and $E^L \subset \partial T^L$ (resp. $E^R \subset \partial T^R$) is the edge that projects to $E$. 
    \item For a boundary edge $E \in e(\Delta)$, 
    \begin{align*}
        q_\Delta^\ast X_{s}^{(E;\bD)} =  X_{s_{E}}^{(T,m_{T},\bs_{T})}\circ \mathrm{pr}_{T}.
    \end{align*}
    Here $T$ is the unique triangle that contain $E$.
\end{itemize}
One can verify that the right-hand sides of these characterizing equations are indeed $H^\Delta$-invariant. 


\begin{figure}[ht]
\scalebox{0.9}{
\begin{tikzpicture}
\begin{scope}[>=latex]
\path(0,0) coordinate(x1);
\path(-6,4) coordinate(x2);
\path(0,8) coordinate(x3);
\path(6,4) coordinate(x4);
\draw(x1) --(x2) -- (x3) -- (x4) -- (x1) --(x3);
\path (x3)++(-0.1,-0.3) node{$\ast$}; 
\path (x4)++(-0.2,0) node{$\ast$}; 

\quiverAthree{x1}{x2}{x3}
    \qdarrow{V10}{V20}
    \qdarrow{V20}{V30}
    {\color{myblue}
    \qdarrow{Y3}{Y2}
    \qdarrow{Y2}{Y1}
    }
\quiverAthree{x1}{x3}{x4}
    \qdarrow{V10}{V20}
    \qdarrow{V20}{V30}
    \qdarrow{V22}{V13}
    \qdarrow{V31}{V22}

\end{scope}
\end{tikzpicture}}
    \caption{The quiver on $T^L {}_{E^L}\cup_{E^R} T^R$ with $\lie=A_3$ and $\bs_{T^L}=\bs_{T^L}=(1,2,3,1,2,1)$. We glue the vertices as $(v^1_3)'=y_1''$, $(v^2_2)'=y_2''$ and $(v^3_1)'=y_3''$.
    }
    \label{fig:quadrilateral_quiverA3}
\end{figure}

\begin{figure}[ht]
\[
\scalebox{0.9}{
\begin{tikzpicture}
\begin{scope}[>=latex]
\path(0,0) coordinate(x1);
\path(-6,4) coordinate(x2);
\path(0,8) coordinate(x3);
\path(6,4) coordinate(x4);
\draw(x1) --(x2) -- (x3) -- (x4) -- (x1) --(x3);
\path (x3)++(-0.1,-0.3) node{$\ast$}; 
\path (x4)++(-0.2,0) node{$\ast$}; 

\quiverCthree{x1}{x2}{x3}
    \qdarrow{V10}{V20}
    \qshdarrow{V20}{V30}
    {\color{myblue}
    \qstdarrow{Y3}{Y2}
    \qdarrow{Y2}{Y1}
    }
\quiverCthreeY{x1}{x3}{x4}
    \qdarrow{V10}{V20}
    \qshdarrow{V20}{V30}
    \qstdarrow{V23}{V33}
    \qdarrow{V13}{V23}
\end{scope}
\end{tikzpicture}}
\]
    \caption{The quiver on $T^L {}_{E^L}\cup_{E^R} T^R$ with $\lie=C_3$ and $\bs_{T^L}=\bs_{T^R}=(1,2,3)^3$. We glue the vertices as $(v_3^s)'=y_s''$ for $s=1,2,3$.}
    \label{fig:quadrilateral_quiverC3}
\end{figure}


Correspondingly, for an interior edge $E \in e(\Delta)$ shared by two triangles $T^L$ and $T^R$ as above, the two seeds $\sfS_{(T^L,m_{T^L},\bs_{T^L})}$ and $\sfS_{(T^R,m_{T^R},\bs_{T^R})}$ are amalgamated according to the gluing data
\begin{align}
    &F:=\{ s_{E^L} \mid s \in S\}, \quad F':=\{ s_{E^R} \mid s \in S\},\nonumber \\
    &\phi: F \to F', \quad s_{E^L} \mapsto s^\ast_{E^R} \mbox{ for $s \in S$}. \label{eq:GS_gluing rule}
\end{align}
See \cref{d:amalgamation} for a detail on the amalgamation procedure. See \cref{fig:quadrilateral_quiverA3,fig:quadrilateral_quiverC3} for examples. 


Applying this procedure for each interior edge, we get a weighted quiver $Q_{\bD}$ drawn on the surface $\Sigma$. 
In the light of \cref{t:GS decomposition}, the collection $\mathbf{X}_{\bD}$ of functions provide an open embedding $\psi_{\bD}: \bG_m^{I_{\bD}} \to \P_{G,\Sigma}$, whose image is contained in $\P_{G,\Sigma}^\Delta$ and the index set is given by 
\begin{align*}
    I(\bD):=\{(s,i;T) \mid T \in t(\Delta),~(s,i) \in I_\uf(\bs)\} \sqcup \{(s;E) \mid E \in e(\Delta),~s \in S\}.
\end{align*}
Thus the pair $\sfS(\bD):=(\mathbf{X}_{\bD},Q_{\bD})$ forms a seed in the ambient field $\Rat(\P_{G,\Sigma})$. 

\begin{thm}[\cite{FG03,Le16,GS19}]\label{t:cluster_structure}
The seeds $\sfS(\bD)$ associated with the decorated triangulations $\bD$ of $\Sigma$ are mutation-equivalent to each other. Hence the collection $(\sfS(\bD))_{\bD}$ is a cluster Poisson atlas on $\P_{G,\Sigma}$. 
\end{thm}

\paragraph{\textbf{Comparison with the cluster Poisson algebra}}
Let $\sfS_{G,\Sigma}$ denote the cluster Poisson structure on $\P_{G,\Sigma}$ which includes the cluster Poisson atlas $(\sfS(\bD))_{\bD}$. Then \cref{t:cluster_structure} tells us that our moduli space $\P_{G,\Sigma}$ is birationally isomorphic to the cluster Poisson variety $\X_{\sfS_{G,\Sigma}}$, and hence their fields of rational functions are isomorphic. Slightly abusing the notation, let us denote the cluster Poisson algebra by $\cO_{\mathrm{cl}}(\P_{G,\Sigma}):=\cO(\X_{\sfS_{G,\Sigma}})$. 
Shen proved the following stronger result:

\begin{thm}[Shen \cite{Shen20}]\label{thm:Shen}
We have an isomorphism $\cO_{\mathrm{cl}}(\P_{G,\Sigma}) \cong \cO(\P_{G,\Sigma})$ of $\C$-algebras. 
\end{thm}


In particular, we have:

\begin{cor}\label{cor:Laurent}
The matrix coefficients of Wilson lines and the traces of Wilson loops are universally Laurent polynomials:
\begin{align*}
    c_{f,v}^V(g_{[c]}), \mathrm{tr}_V(\rho_{|\gamma|}) \in \cO_{\mathrm{cl}}(\P_{G,\Sigma})
\end{align*}
for any representation $V$, $f \in V^\ast$, $v \in V$, arc class $[c]$ and a free loop $|\gamma|$. 
\end{cor}
Our aim in the sequel is to obtain an explicit formula for these Laurent polynomials, and prove the positivity of coefficients when the coordinate system is associated with a decorated triangulation.




\bigskip 

\paragraph{\textbf{Partially generic case}}
For a subset $\Xi \subset \mathbb{B}$, consider the moduli stack $\P_{G,\Sigma;\Xi}$ of $\Xi$-generic framed $G$-local systems with $\Xi$-pinnings (recall \cref{def:moduli_partial_pinnings} and \eqref{eq:Betti_partial_pinnings}). For a decorated triangulation $\bD$, set 
\begin{align}\label{eq:index set_partial_pinnings}
    I_\Xi(\bD):=I(\bD) \setminus \{(s;E) \mid E \in \mathbb{B} \setminus \Xi,\ s \in S\}.
\end{align}
Then by \cref{rem:GS coordinates_partial_pinnings}, we have an open embedding $\bG_m^{I_\Xi(\bD)} \to \P_{G,\Sigma;\Xi}$ such that the projection $\P_{G,\Sigma} \to \P_{G,\Sigma;\Xi}$ forgetting the pinnings except for those assigned to the boundary intervals in $\Xi$ is expressed as a coordinate projection. In other words, the moduli space $\P_{G,\Sigma;\Xi}$ also has a canonical cluster Poisson atlas so that the projections $\P_{G,\Sigma} \to \P_{G,\Sigma;\Xi}$ is a cluster projection.


\section{Laurent positivity of Wilson lines and Wilson loops}\label{sec:positivity}
\subsection{The statements}
In this section, we show that Wilson lines and Wilson loops have remarkable positivity properties. Let us first clarify the positivity properties which we will deal with, and state the main theorems of this section.  

Let $\Sigma$ be a marked surface (See \cref{subsec:moduli} for our assumption on the marked surface). 
We say that $f\in \cO(\P_{G, \Sigma})$ is \emph{\GSuniv\ positive Laurent} if it is expressed as a Laurent polynomial with non-negative integral coefficients in the GS coordinate system associated with any decorated triangulation $\boldsymbol{\Delta}$. A morphism $F:  \P_{G,\Sigma}\to G$ is called a \emph{\GSuniv\ positive Laurent morphism} if for any finite-dimensional representation $V$ of $G$, there exists a basis $\mathbb{B}$ of $V$ such that $c_{f, v}^V\circ F\in \cO(\P_{G, \Sigma})$ is \GSuniv\ positive Laurent for all $v\in \mathbb{B}$ and $f\in \mathbb{F}$, where $\mathbb{F}$ is the basis of $V^{\ast}$ dual to $\mathbb{B}$ (see \eqref{eq:mat_coeff}). 
 
 	\begin{rem}\label{r:GSpos}
 		In \cite{FG03}, Fock and Goncharov introduced the notion of \emph{special good positive Laurent polynomials on $\X_{PGL_{n+1},\Sigma}$}. Our \GSuniv\ positive Laurent property is a straightforward generalization of their notion. Indeed, if we set $G=PGL_{n+1}$ and $\bs_{\mathrm{std}}(n)=(1,2,\dots, n, \dots,1,2,3,1, 2, 1)$ as in \eqref{eq:std_word} for all triangles $T$ of the decorated triangulation $\bD=(\Delta_*, (\bs_{\mathrm{std}}(n))_T)$ of $\Sigma$, the GS coordinate system on $\P_{PGL_{n+1},\Sigma}$ associated with $\bD$ is the \emph{special atlas} on $\X_{PGL_{n+1},\Sigma}$ in \cite[Definiton 9.1]{FG03} (modulo the the difference between $\P_{PGL_{n+1},\Sigma}$ and $\X_{PGL_{n+1},\Sigma}$). 
 		
 		We should remark that the definition of \GSuniv\ positive Laurent morphism does not change if we replace ``any  finite-dimensional representation $V$'' in its definition with ``any \emph{simple} finite-dimensional representation $V(\lambda)$, $\lambda\in X^{\ast}(H)_+$'' because of the complete reducibility of finite-dimensional representations. 
 	\end{rem}
%
%
%
The following theorem is the main result in this section. 

\begin{thm}\label{t:Wilson_line_positivity}
	Let $G$ be a semisimple algebraic group of adjoint type, and assume that our marked surface $\Sigma$ has non-empty boundary. 
	Then for any arc class  $[c]:E_\inn \to E_\out$, 
	the Wilson line $g_{[c]}:  \P_{G,\Sigma}\to G$ is a \GSuniv\ positive Laurent morphism.  
\end{thm}
Combining with \cref{prop:Wilson line-loop}, we immediately get the following: 

\begin{cor}\label{t:monodromy_positivity}
	Let $G$ be a semisimple algebraic group of adjoint type, and $|\gamma| \in \widehat{\pi}(\Sigma^\ast)$ a free loop. Then, for any finite dimensional representation $V$ of $G$, the trace of the Wilson loop $\tr_V(\rho_{|\gamma|}):=\tr_V\circ\rho_{|\gamma|}\in \cO(\P_{G, \Sigma})$ is \GSuniv\ positive Laurent. 
\end{cor}
\cref{t:monodromy_positivity} is a generalization of \cite[Theorem 9.3, Corollary 9.2]{FG03}. The rest of this section is devoted to the proof of \cref{t:Wilson_line_positivity}.

\subsection{A basis of $\cO(\P_{G, T})$ with positivity}
Our computation is performed locally on each triangle $T$ of an arbitrarily fixed decorated triangulation $\bD$ of $\Sigma$. An important fact is the existence of a basis of $\cO(\mathcal{P}_{G, T})$ with an appropriate positivity. In this subsection, we explain a construction of such a nice basis. 
Fix a triangle $T$. 
Recall the standard configuration $\wC_3 : H \times H \times U_*^+ \xrightarrow{\sim} \Conf_3 \P_G$ in \cref{l:standardconfig}, and the map $f_{m}: \mathcal{P}_{G, T}\xrightarrow{\sim} \Conf_3 \P_G$ given in \eqref{eq:moduli_triangle} associated with each marked point $m$ of $T$. Then we have an isomorphism 
\[
\wC_{3, m}:=f_{m}^{-1}\circ \wC_3:  H \times H \times U_*^+ \xrightarrow{\sim} \mathcal{P}_{G, T},
\] 
which induces an isomorphism of the coordinate rings 
\[
\wC_{3, m}^{\ast} :  \cO(\mathcal{P}_{G, T}) \xrightarrow{\sim} \cO(H \times H \times U_*^+)=\cO(H)\otimes \cO(H)\otimes \cO(U_*^+).   
\]
The coordinate ring $\cO(H)$ is identified with the group algebra $\C[X^{\ast}(H)]=\sum_{\mu\in X^{\ast}(H)}\mathbb{C}e^{\mu}$. To distinguish the first component of $\cO(H)\otimes \cO(H)\otimes \cO(U_*^+)$ from its second component, we write the element $e^{\mu}$, $\mu\in X^{\ast}(H)$ in the first (resp.~second) component as $e_1^{\mu}$ (resp.~$e_2^{\mu}$).  Recall that the coordinate algebra $\cO(U_*^+)$ has a $X^*(H)$-grading $\cO(U_*^+)=\bigoplus_{\beta\in X^*(H)}\cO(U_*^+)_{\beta}$ such that 
\[
F\circ \Ad_{h}|_{U^+_*}=h^{\beta}F
\]
for $h\in H$ and $F\in \cO(U_*^+)_{\beta}$. For $\xi\in X^{\ast}(H)$, set 
\begin{align}
    \Delta_{w_0, \xi}^+:=(\Delta_{\lambda_1, w_0\lambda_1}|_{U_*^+})^{-1}\Delta_{\lambda_2, w_0\lambda_2}|_{U_*^+}\in \cO(U_*^+)\label{eq:Delta}
\end{align}
with $\lambda_1, \lambda_2\in X^{\ast}(H)_+$ such that $-\lambda_1+\lambda_2=\xi$. Note that $\Delta_{w_0, \xi}^+$ is a well-defined element. 

The description of the cyclic shift $\mathcal{S}_3$ on $\Conf_3\P_G$ (\cref{l:cyclic shift}) in terms of the standard configuration is important in the sequel. In the description, we use 
the \emph{Berenstein--Fomin--Zelevinsky twist automorphism} $\eta_{w_0}$ \cite{BFZ96,BZ97} (we call it the \emph{twist automorphism} for short) defined by 
\[
\eta_{w_0}: U_*^+ \to U_*^+, u_+ \mapsto [\overline{w_0}u_+^{\mathsf{T}}]_+.
\]
This is a regular automorphism of $U_*^+$. The properties of $\eta_{w_0}$ are collected in \cref{sec:twist map}. 
\begin{lem}\label{l:t=1_cyclic}
	We have $\mathcal{S}_3 (\wC_3(h_1,h_2,u_+)) =\wC_3(h'_1,h'_2,u'_+)$, where
	\begin{align*}
		h'_1 &= h_1^{\ast}h_2w_0([u_+\overline{w_0}]_0), \\
		h'_2 &= h_1^{\ast}, \\
		u'_+ &= \Ad_{h_1^{\ast}}(\eta_{w_0}(u_+)^{\ast}).
	\end{align*}
Hence the isomorphism $\wC_3^{-1}\circ \mathcal{S}_3\circ\wC_3:  H \times H \times U_*^+\xrightarrow{\sim}H \times H \times U_*^+$ is expressed as
\begin{align}
(\wC_3^{-1}\circ \mathcal{S}_3\circ\wC_3)^{\ast}(e_1^{\mu}\otimes e_2^{\nu}\otimes F)=e_1^{\mu^{\ast}+\nu^{\ast}+\beta^{\ast}}\otimes e_2^{\mu}\otimes \Delta_{w_0, -\mu^{\ast}-\beta}^+\cdot (\eta_{w_0}^{\ast})^4(F)\label{eq:cyclic_formula}
\end{align}
for $\mu,\nu\in X^{\ast}(H)$ and $F\in \cO(U_*^+)_{\beta}$. 
\end{lem}
\begin{proof}
	We have
	\begin{align*}
		\mathcal{S}_3(\wC_3(h_1,h_2,u_+)) 
		&= [\phi'(u_+)h_1\vw.p_\std, u_+h_2\vw.p_\std, p_\std] \\
		&= [p_\std, (\phi'(u_+)h_1\vw)^{-1}u_+h_2\vw.p_\std, (\phi'(u_+)h_1\vw)^{-1}.p_\std].
	\end{align*}
	The third component is rewritten as 
	\begin{align*}
		\vw h_1^{-1}\phi'(u_+)^{-1}.p_\std
		&= \Ad_{\vw h_1^{-1}}(\phi'(u_+)^{-1}) w_0(h_1^{-1})\vw. p_\std\\
		&= \Ad_{h_1^{\ast}}((\phi'(u_+)^{\mathsf{T}})^{\ast})h_1^{\ast}\vw. p_\std\\
		&=\Ad_{h_1^{\ast}}(\eta_{w_0}(u_+)^{\ast})h_1^{\ast}\vw. p_\std. \quad (\text{by  \cref{l:phi-tw}})
	\end{align*}
	Hence we have $h'_2 =h_1^{\ast}$ and $u'_+ =\Ad_{h_1^{\ast}}(\eta_{w_0}(u_+)^{\ast})$. The second component is rewritten as 
	\begin{align*}
		\vw h_1^{-1}\phi'(u_+)^{-1}u_+h_2\vw. p_\std
		&= \vw h_1^{-1}[u_+\overline{w_0}]_0\vw^{-1} \phi(u_+) h_2\vw. p_\std \quad (\text{by  \cref{p:phi-decomp}})\\
		&= h_1^{\ast}w_0([u_+\overline{w_0}]_0)\phi(u_+) h_2\vw. p_\std\\
		&= \Ad_{h_1^{\ast}w_0([u_+\overline{w_0}]_0)}(\phi(u_+)) h_1^{\ast}w_0([u_+\overline{w_0}]_0)h_2\vw. p_\std.
	\end{align*}
	Hence, by reading the Cartan part off, we see that $h'_1 = h_1^{\ast}h_2w_0([u_+\overline{w_0}]_0)$. 
	
	From the computation of $h_1'. h_2'$ and $u_+'$ above, it follows that 
	\begin{align*}
\left((\wC_3^{-1}\circ \mathcal{S}_3\circ\wC_3)^{\ast}(e_1^{\mu}\otimes e_2^{\nu}\otimes F)\right)(h_1, h_2, u_+)&=(e_1^{\mu}\otimes e_2^{\nu}\otimes F)(h_1', h_2', u_+')\\
&=h_1^{\mu^{\ast}+\nu^{\ast}+\beta^{\ast}}h_2^{\mu}(w_0([u_+\overline{w_0}]_0))^{\mu}(F\circ \ast)(\eta_{w_0}(u_+)).
\end{align*}
for $\mu,\nu\in X^{\ast}(H)$ and $F\in \cO(U_*^+)_{\beta}$. Moreover, for $\mu_1, \mu_2\in X^{\ast}(H)_+$ with $-\mu_1+\mu_2=-\mu^{\ast}$, we have 
\begin{align*}
(w_0([u_+\overline{w_0}]_0))^{\mu}&=[u_+\overline{w_0}]_0^{-\mu^{\ast}}\\
&=(\Delta_{\mu_1, \mu_1}([u_+\overline{w_0}]_0))^{-1}\Delta_{\mu_2, \mu_2}([u_+\overline{w_0}]_0)\\
&=(\Delta_{\mu_1, \mu_1}(u_+\overline{w_0}))^{-1}\Delta_{\mu_2, \mu_2}(u_+\overline{w_0})\\
&=(\Delta_{\mu_1, w_0\mu_1}(u_+))^{-1}\Delta_{\mu_2, w_0\mu_2}(u_+)\\
&=\Delta_{w_0, -\mu^{\ast}}^+(u_+). 
\end{align*}
By \cref{l:tw-power}, we have 
\begin{align*}
(F\circ \ast)(\eta_{w_0}(u_+))&=(\Delta_{w_0, \beta}^+\cdot (\eta_{w_0}^{\ast})^3(F))(\eta_{w_0}(u_+))\\
&=(\eta_{w_0}^{\ast}(\Delta_{w_0, \beta}^+))(u_+)((\eta_{w_0}^{\ast})^4(F))(u_+)\\
&=(\Delta_{w_0, -\beta}^+\cdot (\eta_{w_0}^{\ast})^4(F))(u_+).\\
\end{align*}
Therefore, we obtain 
\[
(\wC_3^{-1}\circ \mathcal{S}_3\circ\wC_3)^{\ast}(e_1^{\mu}\otimes e_2^{\nu}\otimes F)=e_1^{\mu^{\ast}+\nu^{\ast}+\beta^{\ast}}\otimes e_2^{\mu}\otimes \Delta_{w_0, -\mu^{\ast}-\beta}^+\cdot (\eta_{w_0}^{\ast})^4(F)
\]
as desired. 
\end{proof}
We now construct a ``canonical'' basis of $\cO(\mathcal{P}_{G, T})$ from that of $\cO(U_*^+)$ with some nice properties (G), (T), and (M). 
\begin{lem}\label{l:triangle_basis}
	Let $\mathbb{F}$ be a basis of $\cO(U_*^+)$ such that 
	\begin{itemize}
		\item[(G)] the elements of $\mathbb{F}$ are homogeneous with respect to the $X^*(H)$-grading $\cO(U_*^+)=\bigoplus_{\beta\in X^*(H)}\cO(U_*^+)_{\beta}$, 
		\item[(T)] $\mathbb{F}$ is preserved by the twist automorphism $\eta_{w_0}^{\ast}:  \cO(U_*^+)\to \cO(U_*^+)$ as a set, and
		\item[(M)] $\Delta_{w_0, \xi}^+\cdot F\in \mathbb{F}$ for any $\xi \in X^{\ast}(H)$ and $F\in \mathbb{F}$. 
	\end{itemize}
	 Then the basis $\widetilde{\mathbb{F}}_T$ of $\cO(\mathcal{P}_{G, T})$ given by 
	\begin{align}
	\widetilde{\mathbb{F}}_T:=\{(\wC_{3, m}^{\ast})^{-1}(e_1^{\mu}\otimes e_2^{\nu}\otimes F)\mid \mu,\nu\in X^{\ast}(H), F\in \mathbb{F}\}\label{eq:triangle_basis}
	\end{align}
	does not depend on the marked point $m$ of $T$. 
\end{lem}
\begin{proof}
	Using the isomorphisms 
	\[
	\wC_{3}^{\ast}:  \cO(\Conf_3 \P_G)\xrightarrow{\sim} \cO(H)\otimes \cO(H)\otimes \cO(U_*^+),\ f_{m}^{\ast}:  \cO(\Conf_3 \P_G)\xrightarrow{\sim} \cO(\mathcal{P}_{G, T}),
	\]
 we set 
	\begin{align*}
	&\widetilde{\mathbb{F}}_3:=\{(\wC_{3}^{\ast})^{-1}(e_1^{\mu}\otimes e_2^{\nu}\otimes F)\mid \mu,\nu\in X^{\ast}(H), F\in \mathbb{F}\},\\
	&\widetilde{\mathbb{F}}_{m}:= f_{m}^{\ast}(\widetilde{\mathbb{F}}_3).
	\end{align*}
	By the rotational symmetry, it suffices to show that $\widetilde{\mathbb{F}}_{m}=\widetilde{\mathbb{F}}_{m'}$, where $m'$ is the marked point placed right after $m$ along the boundary orientation. By \cref{l:cyclic shift}, we have 
	\begin{align*}
	\widetilde{\mathbb{F}}_{m'}&=f_{m'}^{\ast}(\widetilde{\mathbb{F}}_3)\\
	&=(f_{m}^{\ast}\circ (f_{m'}\circ f_{m}^{-1})^{\ast})(\widetilde{\mathbb{F}}_3)\\
	&=(f_{m}^{\ast}\circ \mathcal{S}_3^{\ast})(\widetilde{\mathbb{F}}_3). 
	\end{align*}
Therefore, it remains to show that $\mathcal{S}_3^{\ast}(\widetilde{\mathbb{F}}_3)=\widetilde{\mathbb{F}}_3$. For $\mu,\nu\in X^{\ast}(H)$ and $F\in \mathbb{F}$ with $F\in \cO(U_*^+)_{\beta}$, we have 
\begin{align*}
	&\mathcal{S}_3^{\ast}((\wC_{3}^{\ast})^{-1}(e_1^{\mu}\otimes e_2^{\nu}\otimes F))\\
	&=(\wC_{3}^{\ast})^{-1}((\wC_{3}^{-1}\circ \mathcal{S}_3\circ \wC_{3})^{\ast}(e_1^{\mu}\otimes e_2^{\nu}\otimes F))\\
	&=(\wC_{3}^{\ast})^{-1}(e_1^{\mu^{\ast}+\nu^{\ast}+\beta^{\ast}}\otimes e_2^{\mu}\otimes \Delta_{w_0, -\mu^{\ast}-\beta}^+\cdot (\eta_{w_0}^{\ast})^4(F))
\end{align*}
by \cref{l:t=1_cyclic}. The assumptions (T) and (M) imply $\Delta_{w_0, -\mu^{\ast}-\beta}^+\cdot (\eta_{w_0}^{\ast})^4(F)\in \mathbb{F}$. Therefore, 
$\mathcal{S}_3^{\ast}((\wC_{3}^{\ast})^{-1}(e_1^{\mu}\otimes e_2^{\nu}\otimes F))\in \widetilde{\mathbb{F}}_3$, which proves $\mathcal{S}_3^{\ast}(\widetilde{\mathbb{F}}_3)=\widetilde{\mathbb{F}}_3$. 
\end{proof}
\begin{rem}
	There are several examples of bases of $\cO(U_*^+)$ which satisfy the properties (G), (T), and (M): 
	\begin{itemize}
		\item the dual semicanonical basis, in the case when $\mathfrak{g}$ is of symmetric type \cite[Theorem 15.10]{GLS:Kac-Moody}, \cite[Theorem 6]{GLS:Chamber}. 
		\item the dual canonical basis (specialized at $q=1$) \cite[Definition 4.6, Theorem 6.1]{Kimura-Oya}.
		\item the simple object basis arising from the monoidal categorification via quiver Hecke algebras \cite[Section 5.1, Theorem 5.13]{KKOP:loc} (see also Remark \ref{r:positivechoice}). 
	\end{itemize}
In this paper, we mainly use the last one because it has a convenient positivity. 
\end{rem}
We use the following strong fact in order to construct a basis of $\cO(\mathcal{P}_{G, T})$ with an appropriate positivity. 
\begin{thm}\label{t:positivechoice}
	There exist 
	\begin{itemize}
	    \item a basis $\pF$ of $\cO(U_*^+)$, and
	    \item two bases $\mathbb{B}(\lambda):=\{G_{\lambda}(\sfb)\mid \sfb\in \mathscr{B}(\lambda)\}$ and $\mathbb{B}^{\rm up}(\lambda):=\{G_{\lambda}^{\rm up}(\sfb)\mid \sfb\in \mathscr{B}(\lambda)\}$ of $V(\lambda)$ for each $\lambda\in X^{\ast}(H)_+$ (here $\mathscr{B}(\lambda)$ is just an index set)
	\end{itemize}
	satisfying the following properties:
	\begin{itemize}
		\item[\Grep] $\mathbb{B}(\lambda)$ and $\mathbb{B}^{\rm up}(\lambda)$ consist of weight vectors of $V(\lambda)$, and we have  
		\[
		( G_{\lambda}(\sfb), G_{\lambda}^{\rm up}(\sfb'))_{\lambda}=\delta_{\sfb, \sfb'}
		\]
		for $\sfb, \sfb'\in \mathscr{B}(\lambda)$. 
		\item[(G)] the elements of $\pF$ are homogeneous with respect to the $X^*(H)$-grading $\cO(U_*^+)=\bigoplus_{\beta\in X^*(H)}\cO(U_*^+)_{\beta}$. 
		\item[(T)] $\pF$ is preserved by the twist automorphism $\eta_{w_0}^{\ast}:  \cO(U_*^+)\to \cO(U_*^+)$ as a set. 
		\item[(M)] $\Delta_{w_0, \xi}^+\cdot F\in \pF$ for any $\xi \in X^{\ast}(H)$ and $F\in \pF$. 
		\item[(P1)] Recall the notation \eqref{eq:vee}. For $\sfb, \sfb'\in \mathscr{B}(\lambda)$, we have 
		\[
		c^{\lambda}_{G_{\lambda}(\sfb)^{\vee}, G_{\lambda}^{\rm up}(\sfb')}|_{U^+_{\ast}}\in \sum_{F\in \pF}\mathbb{Z}_{\geq 0}F,\quad 
		(c^{\lambda}_{G_{\lambda}(\sfb)^{\vee}, G_{\lambda}^{\rm up}(\sfb')}\circ \mathsf{T})|_{U^+_{\ast}}\in \sum_{F\in \pF}\mathbb{Z}_{\geq 0}F. 
		\]
		\item[(P2)] Recall the notation \eqref{eq:Lusztig-param}. For any reduced word $\bs=(s_1,\dots, s_N)$ of $w_0$ and $F\in \pF$, we have 
		\[
		(x^\bs)^{\ast}(F)\in \mathbb{Z}_{\geq 0}[t_1^{\pm 1}, \dots, t_{N}^{\pm 1}].
		\]
	\end{itemize}
\end{thm}
In the following, we write the weight of $G_{\lambda}(\sfb)$ (and $G_{\lambda}^{\rm up}(\sfb')$) as $\weight \sfb$. 
\begin{rem}\label{r:condition_remarks}
	In (P1), either $c^{\lambda}_{G_{\lambda}(\sfb)^{\vee}, G_{\lambda}^{\rm up}(\sfb')}$ or $(c^{\lambda}_{G_{\lambda}(\sfb)^{\vee}, G_{\lambda}^{\rm up}(\sfb')}\circ \sfT)|_{U^+_{\ast}}$ is equal to $0$ if the weights of $G_{\lambda}(\sfb)$ and $G_{\lambda}^{\rm up}(\sfb')$ are distinct. Since $c^{\lambda}_{G_{\lambda}(\sfb)^{\vee}, G_{\lambda}^{\rm up}(\sfb')}\circ \mathsf{T}=c^{\lambda}_{G_{\lambda}^{\rm up}(\sfb')^{\vee}, G_{\lambda}(\sfb)}$, we can interchange the roles of $\mathbb{B}(\lambda)$ and $\mathbb{B}^{\rm up}(\lambda)$. 
\end{rem}
\begin{rem}\label{r:positivechoice}
	\cref{t:positivechoice} is highly non-trivial, but it is now known that an example of such bases can be obtained from the theory of \emph{categorification of $\cO(U_*^+)$ via quiver Hecke algebras}, developed in  \cite{KL:I,Rou08,KL:II,Rou12,KK:hw,KKKO:monoidal,KKOP:strata,KKOP:loc}. 
 
Let us give a brief explanation. We refer to the corresponding references for all missing details. The coordinate ring $\cO(U^+)$ can be regarded as the graded dual $\cU(\fu^+)^{\ast}_{\rm gr}$ of the universal enveloping algebra of $\fu^+$, which is the Lie algebra of $U^+$. Here the algebra structure on $\cU(\fu^+)^{\ast}_{\rm gr}$ is induced from the usual coalgebra structure on $\cU(\fu^+)$ (see \cite[Section 5.1]{GLS:Kac-Moody}). In \cite{KL:I,Rou08,KL:II}, it is shown that the quantum analogue of $\cU(\fu^+)$ (resp.~$\cU(\fu^+)^{\ast}_{\rm gr}$) is realized as the complexified split Grothendieck ring of the category of finitely generated graded projective modules (resp.~the complexified Grothendieck ring of the category of finite-dimensional graded modules) over appropriate quiver Hecke algebras. Hence $\cU(\fu^+)$ (resp.~$\cU(\fu^+)^{\ast}_{\rm gr}$) possesses the basis $\bP$ (resp.~$\bL$) consisting of the classes of the indecomposable projective objects (resp.~simple objects). Here specializing the quantum parameter $q$ at $1$ corresponds to forgetting the grading of the objects. The structure constants with respect to $\bP$ and $\bL$ are manifestly non-negative by definition, and the basis $\bL$ is dual to $\bP$. In particular, for any reduced word $\bs=(s_1,\dots, s_N)$ of $w_0$ and $L\in \bL$,
\begin{align}
(x^\bs)^{\ast}(L)\in \mathbb{Z}_{\geq 0}[t_1^{\pm 1}, \dots, t_{N}^{\pm 1}]\label{eq:Feigin_pos}
\end{align}
(cf.~\cite[Section 6]{GLS:Kac-Moody}). 

The coordinate ring $\cO(U_*^+)$ of $U_*^+$ can be obtained as the localization of $\cO(U^+)$ with respect to $\{\Delta_{\lambda, w_0\lambda}|_{U^+}\mid \lambda\in X^{\ast}(H)_+\}$. Hence we can define the subset $\widetilde{\bL} \subset \cO(U_*^+)$ as 
\[
\widetilde{\bL}=\{\Delta_{w_0, \xi}^+\cdot L \mid L\in \bL,~ \xi\in X^{\ast}(H)\}.
\]
By \cite[Section 5.1]{KKOP:loc}, $\widetilde{\bL}$ is a basis of $\cO(U_*^+)$ arising from the classes of simple objects of a certain monoidal abelian category $\widetilde{\mathscr{C}}_{w_0}$, and $\widetilde{\bL}$ provides an example of $\pF$ in \cref{t:positivechoice}. Indeed, (G) and (T) can be checked easily, and  (T) holds because the twist automorphism $\eta_{w_0}^{\ast}$ is categorified in \cite{KKOP:loc} as the left dual functor (see \cite[Theorem 5.13]{KKOP:loc}). The property (P2) can be shown from \eqref{eq:Feigin_pos} and the fact that $(x^\bs)^{\ast}(\Delta_{w_0, \xi}^+)$ is a Laurent monomial with coefficient $1$ in $t_1,\dots, t_N$ for all $\xi\in X^{\ast}(H)$ (see \cite[Lemma 6.4]{BZ97}). 

For $\lambda\in X^{\ast}(H)_+$, the highest weight module $V(\lambda)$ is realized as the complexified split Grothendieck ring of the category of finitely generated graded projective modules and the complexified Grothendieck ring of the category of finite-dimensional graded modules over appropriate \emph{cyclotomic} quiver Hecke algebras \cite{KK:hw}. Here these two realizations correspond to two choices of $\Z$-forms of $V(\lambda)$. Then we can obtain a basis $\bP(\lambda)$ (resp.~$\bL(\lambda)$) of $V(\lambda)$ consisting of the classes of the indecomposable projective objects (resp.~simple objects). Then if we take 
\begin{align*}
    \mathbb{B}(\lambda)=\bP(\lambda), \quad \mathbb{B}^{\rm up}(\lambda)=\bL(\lambda), \quad \pF=\widetilde{\bL},
\end{align*}
respectively,
then they satisfies \Grep\ and (P1)
in \cref{t:positivechoice}. Indeed, for (P1), it suffices to show that the action of the elements of $\bP$ can be expressed as the matrices with non-negative integer entries with respect to the bases $\bP(\lambda)$ and $\bL(\lambda)$. This property holds since the action of the elements of $\bP$ is also categorified as certain functors. See \cite{KK:hw} for more details\footnote{The further details of Remark \ref{r:positivechoice} can be found in Appendix B of the version 2 of our preprint arXiv:2011.14260.}. 
\end{rem}
In the following, the notations $\pF$, $\mathbb{B}(\lambda)$, and  $\mathbb{B}^{\rm up}(\lambda)$ always stand for bases satisfying the properties \Grep, (G), (T), (M), (P1), and (P2). Moreover, let $\ptF{T}$ be the basis of $\cO(\mathcal{P}_{G, T})$ defined from $\pF$ as in \eqref{eq:triangle_basis}. 
\begin{thm}\label{t:triangle_GSpos}
The basis $\ptF{T}$ consists of \GSuniv\ positive Laurent elements. 
\end{thm}
\begin{proof}
	Recall that a decorated triangulation $\bD$ of $T$ is determined by the choice of a dot $m$ on $T$ and a reduced word $\bs$ of $w_0$. The associated GS coordinates on $\P_{G,T}$ are defined as $X_{s \choose i}^{(T,m,\bs)}:=f_{m}^*X_{s \choose i}^{\bs}$, where the right-hand side is 
	the pull-back of the GS coordinate on $\Conf_3 \P_G$ associated with $\bs$. By \cref{l:triangle_basis}, given any decorated triangulation $\bD$ of $T$, we may regard $\ptF{T}$ as
	$f_{m}^*(\widetilde{\mathbb{F}}_{\mathrm{pos}, 3})$, where $m$ is the dot of $\bD$ and
	\[
	\widetilde{\mathbb{F}}_{\mathrm{pos}, 3}:=\{(\wC_{3}^{\ast})^{-1}(e_1^{\mu}\otimes e_2^{\nu}\otimes F)\mid \mu,\nu\in X^{\ast}(H), F\in \pF \}.
	\]
	Therefore, it suffices to show that $\widetilde{\mathbb{F}}_{\mathrm{pos}, 3}$ is expressed as a Laurent polynomial with non-negative integral coefficients in terms of the GS coordinates on $\Conf_3 \P_G$ associated with $\bs$. Recall the map $\psi_\bs$ in \eqref{eq:GS coord} and the maps $\mathsf{h}_1,\mathsf{h}_2, \mathsf{u}_+$ in \eqref{eq:def_of_hu}, whose explicit descriptions are given in \cref{l:GScoord_standard}. For $\mu,\nu\in X^{\ast}(H)$ and $F\in \pF$, we have 
	\begin{align*}
		\psi_\bs^{\ast}\left((\wC_{3}^{\ast})^{-1}(e_1^{\mu}\otimes e_2^{\nu}\otimes F)\right)=(e_1^{\mu}\otimes e_2^{\nu}\otimes F)(\mathsf{h}_1(\mathbf{X}), \mathsf{h}_2(\mathbf{X}), \mathsf{u}_+(\mathbf{X})).
	\end{align*}
By \cref{l:GScoord_standard} and the property (P2) of $\pF$, the right-hand side is a Laurent polynomial with non-negative integral coefficients in $\{X_{s \choose i}^{\bs}\}_{(s,i)\in I_{\infty}(\bs)}$. This completes the proof. 
\end{proof}
\begin{rem}
	In the proof of \cref{t:triangle_GSpos}, we do not use the property (P1) of $\pF$. 
\end{rem}
\begin{lem}\label{l:Fpos_ast}
We have $F\circ \ast\in \pF$ for all $F\in \pF$. 
\end{lem}
\begin{proof}
	From \cref{l:tw-power}, we have $F\circ \ast=\Delta_{w_0, \beta}^+\cdot (\eta_{w_0}^{\ast})^3(F)$ for $F\in \pF\cap \cO(U_*^+)_{\beta}$. 
We have $(\eta_{w_0}^{\ast})^3(F)\in \pF$ by the property (T), and then $\Delta_{w_0, \beta}^+\cdot (\eta_{w_0}^{\ast})^3(F)\in \pF$ by the property (M). 
\end{proof}
Recall the basic Wilson lines $b_L: \Conf_3 \P_G \to B^+_*$ and $b_R: \Conf_3 \P_G \to B^-_*$ from \cref{d:basic-Conf3}. For a dot $m$ of $T$, we set
\begin{align*}
	&g_{m, L}: \P_{G,T}\xrightarrow{f_{m}}\Conf_3 \P_G \xrightarrow{b_L} B^+_* \overset{\iota}{\hookrightarrow} G,\\
	&g_{m, R}: \P_{G,T}\xrightarrow{f_{m}}\Conf_3 \P_G \xrightarrow{b_R} B^-_* \overset{\iota}{\hookrightarrow} G,
\end{align*}
where the last maps $\iota$ are the inclusion maps. 
\begin{thm}\label{t:matrixcoeff_pos}
	For $\lambda\in X^{\ast}(H)_+, \sfb, \sfb'\in \mathscr{B}(\lambda), i\in \{1,2,3\}$ and $\tau\in \{L, R\}$, the map $c^{\lambda}_{G_{\lambda}(\sfb)^{\vee}, G_{\lambda}^{\rm up}(\sfb')}\circ g_{m, \tau}\in \cO(\P_{G, T})$ is written as a $\mathbb{Z}_{\geq 0}$-linear combination of elements of $\ptF{T}$. In particular, it is \GSuniv\ positive Laurent, and $g_{m, \tau}$ is a \GSuniv\ positive  Laurent morphism. 
\end{thm}
\begin{proof}
We have 
	\begin{align*}
		c^{\lambda}_{G_{\lambda}(\sfb)^{\vee}, G_{\lambda}^{\rm up}(\sfb')}\circ g_{m, L}
		&=c^{\lambda}_{G_{\lambda}(\sfb)^{\vee}, G_{\lambda}^{\rm up}(\sfb')}\circ \iota \circ b_L\circ f_{m}\\
		&=c^{\lambda}_{G_{\lambda}(\sfb)^{\vee}, G_{\lambda}^{\rm up}(\sfb')}\circ \iota \circ b_L\circ \wC_3\circ \wC_3^{-1} \circ f_{m}\\
		&=(1\otimes e_2^{\weight \sfb'}\otimes c^{\lambda}_{G_{\lambda}(\sfb)^{\vee}, G_{\lambda}^{\rm up}(\sfb')}|_{U^+_{\ast}})\circ \wC_3^{-1} \circ f_{m}\ (	\text{by \cref{c:LR}})\\
		&\in \sum_{F\in \pF}\mathbb{Z}_{\geq 0} (\wC_{3, m}^{\ast})^{-1}(1\otimes e_2^{\weight \sfb'}\otimes F)\ (\text{by the property (P1)})\\
		&\subset \sum_{\widetilde{F}\in \ptF{T}}\mathbb{Z}_{\geq 0} \widetilde{F}. 
	\end{align*}
\begin{align*}
	c^{\lambda}_{G_{\lambda}(\sfb)^{\vee}, G_{\lambda}^{\rm up}(\sfb')}\circ g_{m, R}
	&=c^{\lambda}_{G_{\lambda}(\sfb)^{\vee}, G_{\lambda}^{\rm up}(\sfb')}
	\circ \iota \circ b_R\circ f_{m}\\
	&=c^{\lambda}_{G_{\lambda}(\sfb)^{\vee}, G_{\lambda}^{\rm up}(\sfb')}\circ \iota \circ b_R\circ \wC_3\circ \wC_3^{-1} \circ f_{m}\\
	&=(1\otimes e_2^{(\weight \sfb)^{\ast}}\otimes c^{\lambda}_{G_{\lambda}(\sfb)^{\vee}, G_{\lambda}^{\rm up}(\sfb')}\circ \mathsf{T}\circ \ast|_{U^+_{\ast}})\circ \wC_3^{-1} \circ f_{m}\ (\text{by \cref{c:LR}})\\
	&\in \sum_{F\in \pF}\mathbb{Z}_{\geq 0} (\wC_{3, m}^{\ast})^{-1}(1\otimes e_2^{(\weight \sfb)^{\ast}}\otimes (F\circ \ast))\ (\text{by the property (P1)})\\
	&=\sum_{F\in \pF}\mathbb{Z}_{\geq 0} (\wC_{3, m}^{\ast})^{-1}(1\otimes e_2^{(\weight \sfb)^{\ast}}\otimes F)\ (\text{by \cref{l:Fpos_ast}})\\
	&\subset \sum_{\widetilde{F}\in \ptF{T}}\mathbb{Z}_{\geq 0} \widetilde{F}. 
\end{align*}
The remaining statements immediately follow from \cref{r:GSpos} and \cref{t:triangle_GSpos}. 
\end{proof}
\subsection{A proof of \cref{t:Wilson_line_positivity}}
Let $\Sigma$ be a marked surface with non-empty boundary, and fix an arbitrary decorated triangulation $\bD = (\Delta_*, \bs_\Delta)$ of $\Sigma$ (recall our assumption on the marked surface in \cref{subsec:moduli}). Recall $\P_{G,\Sigma}^\Delta$ defined after \cref{t:GS decomposition}, where $\Delta$ is the underlying triangulation of $\Delta_*$. Fix an arc class $[c]: \widetilde{E}_\inn\to \widetilde{E}_\out$. 
For our purpose, it suffices to show that 
\[
c^{\lambda}_{G_{\lambda}(\sfb)^{\vee}, G_{\lambda}^{\rm up}(\sfb')}\circ g_{[c]}\in \cO(\P_{G,\Sigma})
\]
is \GSuniv\ positive Laurent for any $\lambda\in X^{\ast}(H)_+$ and $\sfb, \sfb'\in \mathscr{B}(\lambda)$ (see \cref{r:GSpos}).


Let $q_\Delta:\widetilde{\P_{G,\Sigma}^\Delta}=\prod_{T \in t(\Delta)} \P_{G,T}\to \P_{G,\Sigma}^\Delta$ be the gluing map in \cref{t:GS decomposition}, and $\pr_T: \widetilde{\P_{G,\Sigma}^\Delta} \to \P_{G,T}$ the projection for $T \in t(\Delta)$. Recall that 
\begin{align*}
	q_\Delta^{\ast}X_{s \choose i}^{(T;\bD)}=\pr_T^{\ast}(f_{m_T}^{\ast}X_{s \choose i}^{\bs_T})
\end{align*}
for $T \in t(\Delta),~(s,i) \in I_\uf(\bs)$, and 
\begin{align*}
	q_\Delta^{\ast}X^{(E;\bD)}_s=
	\begin{cases}
		\pr_{T^L}^{\ast}(f_{m_{T^L}}^{\ast}X_{s_{E^L}}^{\bs_{T^L}})\cdot \pr_{T^R}^{\ast}(f_{m_{T^R}}^{\ast}X_{s_{E^R}^\ast}^{\bs_{T^R}})&\text{if }E\ \text{is an interior edge,}\\
		\pr_{T}^{\ast}(f_{m_{T}}^{\ast}X_{s_{E}})&\text{if }E\ \text{is a boundary interval,}
	\end{cases}
\end{align*}
for $E \in e(\Delta)$ and $s \in S$. Here $T^L$ (resp. $T^R$) is the triangle containing $E$ and lies on the left (resp. right) side with respect to the orientation of $E$ in the first case, and $T$ is the unique triangle containing $E$ in the second case. 
By the correspondence above, it suffices to show that $q_\Delta^{\ast}(c^{\lambda}_{G_{\lambda}(\sfb)^{\vee}, G_{\lambda}^{\rm up}(\sfb')}\circ g_{[c]})$ is expressed as a Laurent polynomial with non-negative integral coefficients in any GS coordinate system on $\prod_{T \in t(\Delta)} \P_{G,T}$.  

Henceforth, we follow the notation in the beginning of \cref{subsec:regularity_of_Wilson_lines_and_loops}. For $\nu=1,\dots, M$, denote by $m_{\nu}$ the dot on $\widetilde{T}_\nu$ which is associated with the turning pattern $(\tau_1,\dots, \tau_M)$ of $\widetilde{c}$. Moreover, we have the commutative diagram

\begin{equation*}
    \begin{tikzcd}
    \prod_{T \in t(\Delta)} \P_{G,T} \ar[r,"\widetilde{\pi}_c^*"] \ar[d,"q_\Delta"'] & \prod_{\nu=1}^M \P_{G,\widetilde{T}_\nu} \ar[d,"q_{\Delta_c}"] \\
    \P_{G,\Sigma}^\Delta \ar[r,"\pi_c^*"'] & \P_{G,\Pi_{c;\Delta}}^{\Delta_c},
    \end{tikzcd}
\end{equation*}
where $\widetilde{\pi}_c^*:=\prod_{\nu=1}^M (\pi_c|_{\widetilde{T}_\nu})^*$ (see the proof of \cref{t:Wilson_line_regular}). Then, by \eqref{eq:Wilson_line_product_formula_2}, we have 
\begin{align*}
q_\Delta^{\ast}(c^{\lambda}_{G_{\lambda}(\sfb)^{\vee}, G_{\lambda}^{\rm up}(\sfb')}\circ g_{[c]})
&=c^{\lambda}_{G_{\lambda}(\sfb)^{\vee}, G_{\lambda}^{\rm up}(\sfb')}\circ \mu_M\circ 
\left(\prod_{\nu=1}^M g_{m_\nu, \tau_\nu}\right) \circ \widetilde{\pi}_c^*\\
&=(\widetilde{\pi}_c^*)^{\ast}
\left(
c^{\lambda}_{G_{\lambda}(\sfb)^{\vee}, G_{\lambda}^{\rm up}(\sfb')}\circ \mu_M\circ \prod_{\nu=1}^M g_{m_\nu, \tau_\nu}\right)\\
&=(\widetilde{\pi}_c^*)^{\ast}\left(
\sum_{\sfb_1,\dots, \sfb_{M-1}\in \mathscr{B}(\lambda)}\prod_{\nu=1}^{M}\left(c^{\lambda}_{G_{\lambda}(\sfb_{\nu-1})^{\vee}, G_{\lambda}^{\rm up}(\sfb_{\nu})}\circ g_{m_{\nu}, \tau_{\nu}}\circ \pr_{\widetilde{T_{\nu}}}\right)\right),
\end{align*}
where $\sfb_0:=\sfb$ and $\sfb_M:=\sfb'$. 
By \cref{t:matrixcoeff_pos}, each $c^{\lambda}_{G_{\lambda}(\sfb_{\nu-1})^{\vee}, G_{\lambda}^{\rm up}(\sfb_{\nu})}\circ g_{m_{\nu}, \tau_{\nu}}$ is \GSuniv\ positive Laurent. Moreover,  
\[
(\widetilde{\pi}_c^*)^{\ast} (\pr_{\widetilde{T_{\nu}}}^{\ast}(X_{s \choose i}^{(\widetilde{T}_\nu ,m,\bs)}))=
\pr_{\pi_c(\widetilde{T_{\nu}})}^{\ast} (X_{s \choose i}^{(\pi_c(\widetilde{T_{\nu}}), \pi_c(m),\bs)})
\]
for any dot $m$ on $\widetilde{T_{\nu}}$, any $(s, i)\in I_\infty(\bs)$, and any reduced word $\bs$ of $w_0$. Thus $q_\Delta^{\ast}(c^{\lambda}_{G_{\lambda}(\sfb)^{\vee}, G_{\lambda}^{\rm up}(\sfb')}\circ g_{[c]})$ is expressed as a Laurent polynomial with non-negative integral coefficients in the GS coordinate system on $\prod_{T \in t(\Delta)} \P_{G,T}$, which completes the proof of \cref{t:Wilson_line_positivity}. 
\appendix


\section{Some maps related to the twist automorphism}\label{sec:twist map}
In this Appendix, we collect some useful properties of the Berenstein--Fomin--Zelevinsky twist automorphism \cite{BFZ96,BZ97} 
\[
\eta_{w_0}: U_*^+ \to U_*^+, u_+ \mapsto [\overline{w_0}u_+^{\mathsf{T}}]_+,
\]
and its related maps. 

Let $\B_\pm:=\{u_\pm. B^{\mp} \mid u_\pm \in U^\pm \} \subset \B_G$ be the open Schubert cells. Consider the intersection $\B_*:=\B_+ \cap \B_-$. 

\begin{lem}[{\cite[Lemma 5.2]{FG06}}]
There are bijections
\[
\alpha_*^\pm: U_*^\pm \to \B_*, \quad u_\pm \mapsto u_\pm. B^\mp. 
\]
\end{lem}
Then we have a bijection $\phi':=(\alpha_*^-)^{-1} \circ \alpha_*^+: U_*^+ \to U_*^-$, which satisfies $\phi'(u_+).B^+ = u_+B^-$ for all $u_+ \in U_*^+$. 
Let us consider another map $\phi: U_*^+ \to U_*^-$ defined by $\phi(u_+):= (\phi'^{-1}(u_+^{\mathsf{T}}))^{\mathsf{T}}$ for $u_+ \in U^+$, which satisfies the following property:

\begin{lem}
We have $\phi(u_+)^{-1}.B^+ = u_+^{-1}.B^-$ for all $u_+ \in U_*^+$.
\end{lem}

\begin{proof}
Let $v_+:=\phi'^{-1}(u_+^{\mathsf{T}}) \in U_*^+$. Then 
\begin{align*}
\phi(u_+)^{-1}B^+\phi(u_+) 
	&= (v_+^{\mathsf{T}})^{-1}B^+ v_+^{\mathsf{T}} = (v_+B^-v_+^{-1})^{\mathsf{T}} \\
	&= (\phi'(v_+)B^+\phi'(v_+)^{-1})^{\mathsf{T}} = u_+^{-1}B^- u_+.
\end{align*}
\end{proof}
Using these maps, we get the following decomposition of an element of the unipotent cell $U_*^+$. Recall the triangular decomposition $G_0=U^- H U^+$, $g=[g]_-[g]_0[g]_+$. 

\begin{prop}\label{p:phi-decomp}
Let $u_+\in U_*^+$, which can be written as $u_+=u_- h\overline{w_0}^{-1} u'_-$ with $u_-, u'_-\in U^-$ and $h\in H$. Then
\begin{align*}
u_- &= \phi'(u_+), \\
u'_- &= \phi(u_+), \\
h &= [u_+\vw ]_0.
\end{align*}
In other words, we have $u_+ = \phi'(u_+) [u_+\vw ]_0\vw^{-1} \phi(u_+)$ for all $u_+ \in U_*^+$.
\end{prop}

\begin{proof}
The first equality follows from 
\begin{align*}
    u_+B^-u_+^{-1} = u_- h\overline{w_0}^{-1} B^- \overline{w_0} h^{-1}u_-^{-1} =u_- B^+ u_-^{-1}.
\end{align*} 
By the same argument, if we write $v_- \in U_*^-$ as $v_- = v_+\vw b'_+$ with $v_+\in U^+$ and $b'_+\in B^+$, then $v_+ = {\phi'}^{-1}(v_-)$. Using this for $u_+^{\mathsf{T}} = (u'_-)^{\mathsf{T}} \vw h u_-^{\mathsf{T}}$, we get $(u'_-)^{\mathsf{T}} = {\phi'}^{-1}(u_+^{\mathsf{T}})$. Hence $u'_- = ({\phi'}^{-1}(u_+^{\mathsf{T}}))^{\mathsf{T}} = \phi(u_+)$.

For the third equality, note that $\vw^{-1} u_+^{\mathsf{T}} = \vw^{-1} (u'_-)^{\mathsf{T}}\vw h u_-^{\mathsf{T}} \in U^-HU^+$. Thus we get $h = [\vw^{-1} u_+^{\mathsf{T}}]_0 = [u_+\vw]_0$.
\end{proof}

The maps $\phi$ and $\phi'$ are related to $\eta_{w_0}$ as follows:

\begin{lem}\label{l:phi-tw}
Let $\eta_{w_0}: U_*^+ \to U_*^+$, $u_+ \mapsto [\overline{w_0}u_+^{\mathsf{T}}]_+$ be the twist automorphism. Then we have $\phi(u_+)=(\eta_{w_0}^{-1}(u_+))^{\mathsf{T}}$ and $\phi'(u_+)=(\eta_{w_0}(u_+))^{\mathsf{T}}$.
\end{lem}

\begin{proof}
Let us write $u_+=b_- \overline{w_0} u'_-$ with $b_-\in B^-$ and $u'_-\in U^-$. Then 
\[
\eta_{w_0}(\phi(u_+)^{\mathsf{T}}) = \eta_{w_0}((u'_-)^{\mathsf{T}}) = [\overline{w_0}u'_-]_+=[b_-^{-1}u_+]_+ = u_+.
\]
The second equality immediately follows from the first one. 
\end{proof}


The $\ast$-involution $\ast\colon G\to G$ (\cref{l:Dynkininv}) restricts to an involution $\ast\colon U^+_{\ast}\to U^+_{\ast}$. 
\begin{lem}\label{l:tw-power}
	Let $\beta\in X^*(H)$. For $F\in \cO(U^+_*)_{\beta}$, we have 
	\[
	(\eta_{w_0}^{\ast})^3(F)=\Delta_{w_0, -\beta}^+\cdot (F\circ \ast).
	\] 
\end{lem}
\begin{proof}
See the proof of \cite[Theorem 8.1]{Kimura-Oya}. 
\end{proof}


\section{Cluster varieties, weighted quivers and their amalgamation}\label{sec:quivers}
Here we recall weighted quivers and their mutations, and the \emph{amalgamation} procedure which produces a new weighted quiver from a given one by \lq\lq gluing" some of its vertices. This procedure naturally fits into the gluing morphism (\cref{t:GS decomposition}) via Goncharov--Shen coordinates. We also recall the construction of weighted quivers from reduced words, following \cite{FG06} and \cite{GS19}.

\subsection{Weighted quivers and the cluster Poisson varieties}
We use the conventions for weighted quivers in \cite{IIO19}. 
Recall that a \emph{weighted quiver} $Q=(I,I_0,\sigma,d)$ is defined by the following data: 
\begin{itemize}
    \item $I_0 \subset I$ are finite sets. 
    \item $\sigma=(\sigma_{ij})_{i,j \in I}$ is a  skew-symmetric $\Z/2$-valued matrix such that $\sigma_{ij} \in \Z$ unless $(i,j) \in I_0 \times I_0$.
    \item $d=(d_i)_{i \in I} \in \Z^I_{>0}$ is a tuple of positive integers.
\end{itemize}
Diagrammatically, $I$ is the set of vertices of the quiver, $d$ is the tuple of weights assigned to vertices, and the data of arrows are encoded in the matrix $\sigma$ as 
\begin{align*}
    \sigma_{ij}:= \#\{\text{arrows from $i$ to $j$}\}- \#\{\text{arrows from $j$ to $i$}\}.
\end{align*}
Here we have \lq\lq half " arrows when $\sigma_{ij} \in \Z/2$ (shown by dashed arrows in figures). The quiver has no loops nor 2-cycles by definition. The subset $I_0$ is called the frozen set, and mutations will be allowed only at the vertices in the complement $I_\uf:= I \setminus I_0$. 
The \emph{ciral dual} of $Q$ is defined by $Q^\mathrm{op}:=(I,I_0,-\sigma,d)$. 

We define the \emph{exchange matrix} $\ve=(\ve_{ij})_{i,j \in I}$ of $Q$ to be $\ve_{ij}:= d_i \sigma_{ij}\gcd(d_i,d_j)^{-1}$. Since we can reconstruct the skew-symmetric matrix $\sigma$ from the pair $(\ve,d)$, we sometimes write $Q=(I,I_0,\ve,d)$. 
The following is a reformulation of the \emph{matrix mutation} (see \emph{e.g.}, ~\cite[(12)]{FG09}) in terms of the weighted quiver: 
\begin{dfn}
For $k \in I_\uf$, let $Q'=(I,I_0,\sigma',d)$ be the weighted quiver given by
\begin{align*}
    \sigma_{ij}'= \begin{cases}
    -\sigma_{ij} & \text{$i=k$ or $j=k$}, \\
    \displaystyle{\sigma_{ij} + \frac{|\sigma_{ik}| \sigma_{kj} + \sigma_{ik} |\sigma_{kj}|}{2} 
      \,\alpha_{ij}^k } & \text{otherwise},
  \end{cases} 
\end{align*}
where $\alpha_{ij}^k = d_k \gcd(d_i, d_j)\gcd(d_k, d_i)^{-1}\gcd(d_k,d_j)^{-1}$. The operation $\mu_k: Q \mapsto Q'$ is called the mutation at the vertex $k$. 
Then the exchange matrix $\ve'$ of $Q'$ is given by the matrix mutation.
\end{dfn}
Let $\cF$ be a field isomorphic to the field of rational functions on $|I|$ independent variables with coefficients in $\C$.  
An \emph{($\X$-)seed} is a pair $(Q,\mathbf{X})$, where $\mathbf{X}=(X_i)_{i \in I}$ is a tuple of algebraically independent elements of $\cF$ and $Q$ is a weighted quiver. For $k \in I \setminus I_0$, let $(Q',\mathbf{X}')$ be another seed where $Q'=\mu_k(Q)$ is obtained from $Q$ by the mutation at $k$, and $\mathbf{X}'=(X'_i)_{i \in I}$ is given by the \emph{cluster Poisson transformation} (or the \emph{cluster $\X$-transformation}):
\begin{align}\label{eq:cluster X-transf}
    X'_i= \begin{cases}
    X_k^{-1} & i=k, \\
    X_i (1+X_k^{-\mathrm{sgn}(\ve_{ik})})^{-\ve_{ik}} & i \neq k.
    \end{cases}
\end{align}
The operation $\mu_k: (Q,\mathbf{X}) \mapsto (Q',\mathbf{X}')$ is called the \emph{seed mutation} at $k$. It is not hard to see that seed mutations are involutive: $\mu_k\mu_k=\mathrm{id}$. We say that two seeds are \emph{mutation-equivalent} if they are connected by a sequence of seed mutations and seed permutations (bijections of $I$ preserving $I_0$ setwise).

Let $\bT_{I_\uf}$ be the regular $|I_\uf|$-valent tree, each of whose edge is labeled by an index in $I_\uf$ so that two edges sharing a vertex have different labels. An assignment $\sfS=(\sfS^{(t)})_{t \in \bT_{I_\uf}}$ of a seed $\sfS^{(t)}=(Q^{(t)},\mathbf{X}^{(t)})$ to each vertex $t$ of $\bT_{I_\uf}$ is called a \emph{seed pattern} if for two vertices $t$, $t'$ sharing an edge labeled by $k \in I_\uf$, the corresponding seeds are related as $\sfS^{(t')}=\mu_k\sfS^{(t)}$. 

The \emph{cluster Poisson variety} $\X_{\sfS}=\bigcup_{t \in \bT_{I_\uf}} \X_{(t)}$ is defined by patching the coordinate tori $\X_{(t)}=\bG_m^I$ corresponding to seeds $\sfS^{(t)}$ by the rational transformations 
\begin{align*}
    \mu_k^*: \cO(\X_{(t')})= \C[X^{(t')}_i \mid i \in I] \to \cO(\X_{(t)})= \C[X^{(t)}_i \mid i \in I]
\end{align*}
given by the formula \eqref{eq:cluster X-transf} whenever $t$ and $t'$ shares an edge labeled by $k \in I_\uf$. The cluster Poisson variety has a natural Poisson structure given by $\{X^{(t)}_i,X^{(t)}_j\}:=\varepsilon_{ij}X^{(t)}_iX^{(t)}_j$. 

The ring $\cO(\X_{\sfS}) = \bigcap_{t \in \bT_{I_\uf}} \cO(\X_{(t)})$ of regular functions is called the \emph{cluster Poisson algebra}, whose elements are called \emph{universally Laurent polynomials}. An element of the sub-semifield $\mathbb{L}_+(\X_{\sfS}):= \bigcap_{t \in \bT_{I_\uf}} \Z_{\geq 0}[X^{(t)}_i \mid i \in I] \subset \cO(\X_{\sfS})$ is called a \emph{universally positive Laurent polynomial}.

\begin{dfn}\label{d:cluster atlas}
A \emph{cluster Poisson atlas} on a variety (or  scheme, stack) $V$ over $\C$ is a collection $(\sfS_\alpha)_{\alpha \in A}$ of seeds (here $A$ is an index set) in the field $\Rat(V)$ of rational functions on $V$ such that
\begin{itemize}
    \item each seed $\sfS_\alpha=(Q_\alpha,\mathbf{X}_\alpha)$ gives rise to a birational isomorphism $\mathbf{X}_\alpha: V \dashrightarrow \bG_m^I$ which admits an open embedding $\psi_\alpha: \bG_m^I \hookrightarrow V$ as a birational inverse;
    \item the seeds $\sfS_\alpha$ for $\alpha \in A$ are mutation-equivalent to each other.
\end{itemize}
\end{dfn}
From the second condition, the collection $(\sfS_\alpha)_{\alpha \in A}$ can be extended to a unique seed pattern $\sfS=(\sfS^{(t)})_{t \in \bT_{I_\uf}}$. In particular we get a birational isomorphism $V \cong \X_{\sfS}$. We call the seed pattern $\sfS$ a \emph{cluster Poisson structure} on $V$, as it is a maximal cluster Poisson atlas. Note that the conditions do not imply an existence of an open embedding $\X_{\sfS} \hookrightarrow V$ when $(\sfS_\alpha)_{\alpha\in A} \subsetneq \sfS$.

A rational function on $V$ can be regarded as a rational function on $\X_{\sfS}$, and we can ask whether it is a universally (positive) Laurent polynomial.

\subsection{Amalgamations}\label{subsec:amalgamation}
We recall the \emph{amalgamation} procedure of weighted quivers, following \cite{FG06}.
\begin{dfn}\label{d:amalgamation}
Let $Q=(I,I_0,\sigma,d)$, $Q'=(I',I'_0,\sigma',d')$ be two weighted quivers. Fix two subsets $F \subset I_0$, $F' \subset I'_0$ and a bijection $\phi:F \to F'$ such that $d'(\phi(i))=d(i)$ for all $i \in I$, which we call the \emph{gluing data}. Then the \emph{amalgamation} of $Q$ and $Q'$ with respect to the gluing data $(F,F',\phi)$ produces the weighted quiver $Q \ast_\phi Q'=(J,J_0,\tau,c)$ defined as follows
\begin{itemize}
    \item $J:=I \cup_\phi I'$, $J_0 \subset I_0 \cup_\phi I'_0$. 
    \item $c(i):= 
    \begin{cases}
    d(i) & \text{ if $i \in I$}, \\
    d'(i) & \text{ if $i \in I' \setminus F'$}. 
    \end{cases}$
    \item The entry $\tau_{ij}$ is given by: 
    \begin{tabular}{c|ccc}
         & $j \in I \setminus F$ & $j \in I' \setminus F'$ & $j \in F$ \\ \hline
        $i \in I \setminus F$ & $\sigma_{ij}$ & $0$ & $\sigma_{ij}$ \\
        $i \in I' \setminus F'$ & $0$ & $\sigma'_{ij}$ & $\sigma'_{ij}$ \\
        $i \in F$ & $\sigma_{ij}$ & $\sigma'_{ij}$ & $\sigma_{ij}+\sigma'_{ij}$
    \end{tabular}
\end{itemize}
Here we can choose any subset $J_0$ of $I_0 \cup_\phi I'_0$
such that $\sigma_{ij}$ is integral unless $(i,j) \in J_0 \times J_0$. 
In this paper, we consider the minimal $J_0$ given by
$$
  J_0 = I_0 \cup_\phi I'_0 \setminus J_1, \quad
  J_1 := \{ i \in I_0 \cup_\phi I'_0 \mid 
  \sigma_{ij} \in \Z \text{ for all $j \in J$} \}.
$$
\end{dfn}
The amalgamation procedure can be upgraded to that for two seeds. Let $(Q,\mathbf{X})$ and $(Q',\mathbf{X}')$ be two seeds, $(F,F',\phi)$ a gluing data as above. Then we define a new seed $(Q \ast_\phi Q',\mathbf{Y})$, where the weighted quiver $Q \cup_\phi Q'$ is given as above and the variables $\mathbf{Y}=(Y_i)_{i \in J}$ is defined by
\begin{align*}
    Y_i:= \begin{cases}
    X_i & \mbox{if $i \in I \setminus F$} \\
    X'_i & \mbox{if $i \in I' \setminus F'$} \\
    X_i \cdot X'_{\phi(i)} & \mbox{if $i \in F$}.
    \end{cases}
\end{align*}
Then it is not hard to check that the amalgamation of seeds commutes with the mutation at any vertex $k \in (I \setminus I_0) \sqcup (I' \setminus I'_0)$. Thus for two seed patterns $\sfS$ and $\sfS'$ and a gluing data as above, we have a dominant morphism
\begin{align*}
    \alpha_\phi: \X_{\sfS} \times \X_{\sfS'} \to \X_{\sfS \ast_\phi \sfS'}.
\end{align*}
Here the seed pattern $\sfS \ast_\phi \sfS'$ is obtained by the amalgamation of the seeds $\sfS^{(t)}$ and ${\sfS'}^{(t)}$ for $t \in \bT_{I_\uf}$. 

\subsection{Weighted quivers from reduced words}
Let us fix a finite dimensional complex semisimple Lie algebra $\mathfrak{g}$. Let $C(\mathfrak{g})=(C_{st})_{s,t \in S}$ be the associated Cartan matrix. For $s \in S$, we define a weighted quiver $\bJ^+(s)=(J(s),J_0(s),\sigma(s),d(s))$ as follows.
\begin{itemize}
    \item $J(s)=J_0(s):=(S \setminus \{s\})\cup\{s^L,s^R\}$, where $s^L,s^R$ are new elements.
    \item The skew-symmetric matrix $\sigma(s)=(\sigma_{tu})_{t,u \in J(s)}$ is given by
    \begin{align*}
            &\sigma_{s^R,s^L}=1, \\
            &\sigma_{s^L,u}=\sigma_{u,s^R}= \begin{cases} 
            1/2 & \text{if $u \neq s$ and $C_{su} \neq 0$}, \\
            0 & \text{if $u \neq s$ and $C_{su}=0$}.
            \end{cases}
    \end{align*}
    Note that other entries are determined by the skew-symmetricity. 
    \item $d(s)$ is given by $d(s)_{s^{L,R}}:=d_s$, $d(s)_{t}:=d_t$ for $t \neq s$.
\end{itemize}
Let $\bJ^-(s):=\bJ^+(s)^{\mathrm{op}}$. We call $\bJ^\pm(s)$ the \emph{elementary quivers} associated with $\mathfrak{g}$. 

For each elementary quiver $\bJ^\epsilon(s)$ with $s \in S$ and $\epsilon \in \{+,-\}$, we define a function $\delta(s): J(s) \to S$ on the set of vertices by $\delta(s)_{s^{L,R}}:=s$ and $\delta(s)_t:=t$ for $t \in S \setminus \{s^L,s^R\}$. We call $\delta$ the \emph{Dynkin labeling} of vertices of $\bJ^\epsilon(s)$.

\begin{ex}
Here are some examples of the elementary quivers.
\begin{enumerate}
\item Type $A_3$: $S=\{1,2,3\}$ and the Cartan matrix is given by
\begin{align*}
    C(A_3) = \begin{pmatrix}
    2&-1&0\\
    -1&2&-1\\
    0&-1&2
    \end{pmatrix}.
\end{align*}
The elementary quivers $\bJ^+(1)$, $\bJ^+(2)$ and $\bJ^+(3)$ are given as follows:
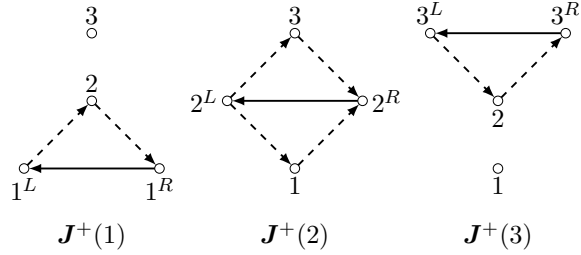
\[
\scalebox{0.9}{
\begin{tikzpicture}

\begin{scope}[>=latex]
\draw (3,0) circle(2pt) coordinate(B) node[below]{$1^R$};
\draw (1,0) circle(2pt) coordinate(C) node[below]{$1^L$};
\draw (2,1) circle(2pt) coordinate(D) node[above]{$2$};
\draw (2,2) circle(2pt) node[above]{$3$};
\qarrow{B}{C};
\qdarrow{C}{D};
\qdarrow{D}{B};
\draw (2,-1) node{$\bJ^+(1)$};
\draw (6,1) circle(2pt) coordinate(E) node[right]{$2^R$};
\draw (4,1) circle(2pt) coordinate(F) node[left]{$2^L$};
\draw (5,0) circle(2pt) coordinate(G) node[below]{$1$};
\draw (5,2) circle(2pt) coordinate(H) node[above]{$3$};
\qarrow{E}{F};
\qdarrow{F}{G};
\qdarrow{G}{E};
\qdarrow{F}{H};
\qdarrow{H}{E};
\draw (5,-1) node{$\bJ^+(2)$};
\draw (9,2) circle(2pt) coordinate(B) node[above]{$3^R$};
\draw (7,2) circle(2pt) coordinate(C) node[above]{$3^L$};
\draw (8,1) circle(2pt) coordinate(D) node[below]{$2$};
\draw (8,0) circle(2pt) node[below]{$1$};
\qarrow{B}{C};
\qdarrow{C}{D};
\qdarrow{D}{B};
\draw (8,-1) node{$\bJ^+(3)$};
\end{scope}
\end{tikzpicture}}
\] 
\item Type $C_3$: $S=\{1,2,3\}$ and the Cartan matrix is given by\footnote{Here we changed the convention from \cite{IIO19}: for type $C_n$, the long root is chosen to be $\alpha_n$. Similarly for type $B_n$, the short root is chosen to be $\alpha_n$.}
\begin{align*}
    C(C_3) = \begin{pmatrix}
    2&-1&0\\
    -1&2&-2\\
    0&-1&2
    \end{pmatrix}.
\end{align*}
The elementary quivers $\bJ^+(1)$, $\bJ^+(2)$, $\bJ^+(3)$ are given as follows.
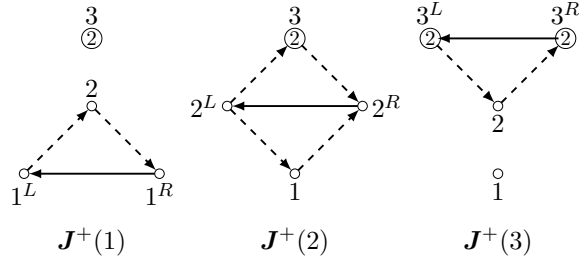
\[
\scalebox{0.9}{
\begin{tikzpicture}

\begin{scope}[>=latex]
\draw (3,0) circle(2pt) coordinate(B) node[below]{$1^R$};
\draw (1,0) circle(2pt) coordinate(C) node[below]{$1^L$};
\draw (2,1) circle(2pt) coordinate(D) node[above]{$2$};
\dnode{2,2} node[above=0.2em]{$3$};
\qarrow{B}{C};
\qdarrow{C}{D};
\qdarrow{D}{B};
\draw (2,-1) node{$\bJ^+(1)$};
\draw (6,1) circle(2pt) coordinate(E) node[right]{$2^R$};
\draw (4,1) circle(2pt) coordinate(F) node[left]{$2^L$};
\draw (5,0) circle(2pt) coordinate(G) node[below]{$1$};
\dnode{5,2} coordinate(H) node[above=0.2em]{$3$};
\qarrow{E}{F};
\qdarrow{F}{G};
\qdarrow{G}{E};
\qshdarrow{F}{H};
\qstdarrow{H}{E};
\draw (5,-1) node{$\bJ^+(2)$};
\dnode{9,2} coordinate(B) node[above=0.2em]{$3^R$};
\dnode{7,2} coordinate(C) node[above=0.2em]{$3^L$};
\draw (8,1) circle(2pt) coordinate(D) node[below]{$2$};
\draw (8,0) circle(2pt) node[below]{$1$};
\qsarrow{B}{C};
\qstdarrow{C}{D};
\qshdarrow{D}{B};
\draw (8,-1) node{$\bJ^+(3)$};
\end{scope}
\end{tikzpicture}}
\] 
\end{enumerate}
Note that the vertices with the same Dynkin label are drawn on the same level in the pictures.
\end{ex}
For a reduced word $\bs=(s_1 \ldots s_l)$ of $u \in W(\mathfrak{g})$ and $\epsilon \in \{+,-\}$, we construct a weighted quiver
$\bJ^\epsilon(\bs)=\bJ^\epsilon(s_1,\dots,s_l)$ by amalgamating the elementary quivers 
$\bJ^\epsilon(s_1),\cdots,\bJ^\epsilon(s_l)$ in the following way: 
for $k=1,\dots,l-1$, amalgamate $\bJ^\epsilon(s_k)$ and $\bJ^\epsilon(s_{k+1})$ by setting the gluing data in \cref{d:amalgamation} as 
\begin{align*}
    F &:= J(s_k) \setminus \{s_k^L\}, \quad
        F':= J(s_{k+1}) \setminus \{s_{k+1}^R\}, \\
    \phi&: F \to F',\quad s_k^R \mapsto s_k,~ s_{k+1} \mapsto s_{k+1}^L,~ t \mapsto t \mbox{ for $t \neq s_k,s_{k+1}$}.
\end{align*}
Note that the Dynkin labelings $\delta(s_k)$ are preserved under this amalgamation. Hence, these functions combine to give an $S$-valued function on the set of vertices of $J^\epsilon(\bs)$, which we call the \emph{Dynkin labeling} again. In the weighted quiver $\bJ^+(\bs)=\bJ^+(s_1)\ast\cdots\ast \bJ^+(s_l)$, let $v_i^s$ be the $(i+1)$-st vertex with Dynkin label $s$ from the left, for $s \in S$ and $i=0,\dots,n^s(\bs)$. Here $n^s(\bs)$ is the number of $s$ which appear in the word $\bs$. We also use the labelling $v_i^s=: v_{k(s,i)}$ for $s \in S$ and $i=1,\dots,n^s(\bs)$, where $k(s,i) \in \{1,\dots,l\}$ denotes the $i$-th number $k$ such that $s_k=s$ in the word $\bs$. Similarly, the vertices of $\bJ^-(\bs_{\mathrm{op}}^\ast)$ are labeled as $\bar{v}_i^s=\bar{v}_{k(s,i)}$ where the index $i$ runs from the left to the right. 

\begin{ex}
Here are some examples of the weighted quiver $\bJ^+(\bs)$ (left) and the corresponding quiver $\bJ^-(\bs_{\mathrm{op}}^\ast)$ (right).
\begin{enumerate}
\item Type $A_3$, $\bs=(1,2,3,1,2,1)$ and $\bs_{\mathrm{op}}^\ast=(3,2,3,1,2,3)$.
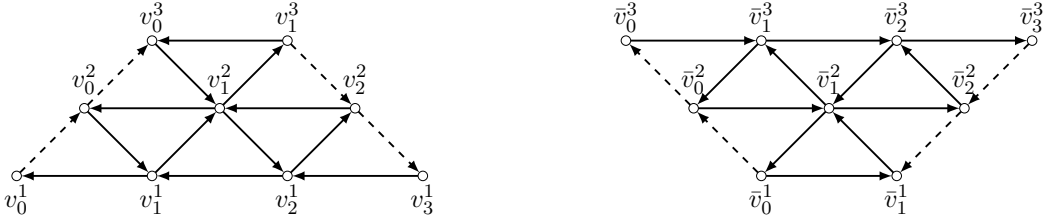
\[
\scalebox{0.9}{
\begin{tikzpicture}

\begin{scope}[>=latex]
\foreach \i in {0,1,2,3}
\draw (2*\i+1,0) circle(2pt) node[below]{$v_{\i}^1$};
\foreach \i in {0,1,2}
\draw (2*\i+2, 1) circle(2pt) node[above=0.2em]{$v_{\i}^2$};
\foreach \i in {0,1}
\draw (2*\i+3, 2) circle(2pt) node[above]{$v_{\i}^3$};
\foreach \i in {1,2,3}
\qarrow{2*\i+1,0}{2*\i-1,0};
\foreach \i in {1,2}
\qarrow{2*\i+2,1}{2*\i,1};
\qarrow{5,2}{3,2};
\foreach \y in {1,2}
	{
	\qdarrow{\y,\y-1}{\y+1,\y};
	\qarrow{\y+2,\y-1}{\y+3,\y};
	}
\qarrow{5,0}{6,1};
\foreach \y in {1,2}
	{
	\qarrow{-\y+5,\y}{-\y+6,\y-1};
	\qdarrow{-\y+7,\y}{-\y+8,\y-1};
	}
\qarrow{2,1}{3,0};
    {\begin{scope}[xshift=9cm]
    \foreach \i in {0,1,2,3}
    \draw (2*\i+1,2) circle(2pt) node[above]{$\bar{v}_{\i}^3$};
    \foreach \i in {0,1,2}
    \draw (2*\i+2, 1) circle(2pt) node[above=0.2em]{$\bar{v}_{\i}^2$};
    \foreach \i in {0,1}
    \draw (2*\i+3, 0) circle(2pt) node[below]{$\bar{v}_{\i}^1$};
    \foreach \i in {1,2,3}
    \qarrow{2*\i-1,2}{2*\i+1,2};
    \foreach \i in {1,2}
    \qarrow{2*\i,1}{2*\i+2,1};
    \qarrow{3,0}{5,0};
    \foreach \y in {1,2}
    	{
    	\qdarrow{-\y+4,\y-1}{-\y+3,\y};
    	\qarrow{-\y+6,\y-1}{-\y+5,\y};
    	}
    \qarrow{6,1}{5,2};
    \foreach \y in {1,2}
        {
        \qarrow{\y+3,\y}{\y+2,\y-1};
        \qdarrow{\y+5,\y}{\y+4,\y-1};
        }
    \qarrow{3,2}{2,1};
    \end{scope}}
\end{scope}
\end{tikzpicture}}
\] 
Here the vertices $v_1^1$, $v_2^1$, $v_1^2$, $\bar{v}_1^3$, $\bar{v}_1^2$ and $\bar{v}_2^3$ are mutable.

\item Type $C_3$, $\bs=(1,2,3,1,2,3,1,2,3)$ and $\bs^\ast_\mathrm{op}=(3,2,1,3,2,1,3,2,1)$.
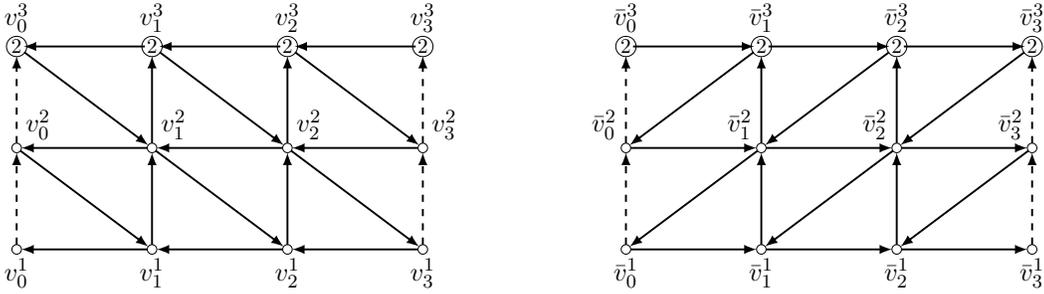
\[
\scalebox{0.9}{
\begin{tikzpicture}
\begin{scope}[>=latex]
\foreach \i in {0,1,2,3}
	\dnode{2*\i+1,3} node[above=0.2em]{$v_{\i}^3$};
\foreach \i in {0,1,2,3}
{
    \draw(2*\i+1,0) circle(2pt) node[below]{$v_{\i}^1$};
    \draw(2*\i+1,1.5) circle(2pt) node[above right]{$v_{\i}^2$};    
}
\foreach \i in {0,1,2}
    {\qarrow{2*\i+3,0}{2*\i+1,0};
     \qarrow{2*\i+3,1.5}{2*\i+1,1.5};
     \qsarrow{2*\i+3,3}{2*\i+1,3};
    }
\qdarrow{7,0}{7,1.5}
\qarrow{5,1.5}{7,0}
\qarrow{5,0}{5,1.5}
\qarrow{3,1.5}{5,0}
\qarrow{3,0}{3,1.5}
\qarrow{1,1.5}{3,0}
\qdarrow{1,0}{1,1.5}

\qshdarrow{7,1.5}{7,3}
\qstarrow{5,3}{7,1.5}
\qsharrow{5,1.5}{5,3}
\qstarrow{3,3}{5,1.5}
\qsharrow{3,1.5}{3,3}
\qstarrow{1,3}{3,1.5}
\qshdarrow{1,1.5}{1,3}

    {\begin{scope}[xshift=9cm]
    \foreach \i in {0,1,2,3}
    	\dnode{2*\i+1,3} node[above=0.2em]{$\bar{v}_{\i}^3$};
    \foreach \i in {0,1,2,3}
    {
        \draw(2*\i+1,0) circle(2pt) node[below]{$\bar{v}_{\i}^1$};
        \draw(2*\i+1,1.5) circle(2pt) node[above left]{$\bar{v}_{\i}^2$};    
    }
    \foreach \i in {0,1,2}
        {\qarrow{2*\i+1,0}{2*\i+3,0};
         \qarrow{2*\i+1,1.5}{2*\i+3,1.5};
         \qsarrow{2*\i+1,3}{2*\i+3,3};
        }
    \qdarrow{7,0}{7,1.5}
    \qarrow{7,1.5}{5,0}
    \qarrow{5,0}{5,1.5}
    \qarrow{5,1.5}{3,0}
    \qarrow{3,0}{3,1.5}
    \qarrow{3,1.5}{1,0}
    \qdarrow{1,0}{1,1.5}
    
    \qshdarrow{7,1.5}{7,3}
    \qstarrow{7,3}{5,1.5}
    \qsharrow{5,1.5}{5,3}
    \qstarrow{5,3}{3,1.5}
    \qsharrow{3,1.5}{3,3}
    \qstarrow{3,3}{1,1.5}
    \qshdarrow{1,1.5}{1,3}
    
    \end{scope}}
\end{scope}
\end{tikzpicture}}
\] 
Here the vertices $v_i^s$ and $\bar{v}_i^s$ for $s=1,2,3$, $i=1,2$ are mutable.
\end{enumerate}
\end{ex}
Note that the weighted quivers $\bJ^+(\bs)$ and $\bJ^-(\bs_{\mathrm{op}}^\ast)$ are isomorphic. 

Next we recall the quiver $\widetilde{J}^+(\bs)$, which is called the \emph{decorated quiver} in \cite{IIO19} in special cases. We follow a more general construction given in \cite{GS19}. 

Let $\bs=(s_1,\dots,s_l)$ be a reduced word of $u \in W(\mathfrak{g})$. For $k=1,\dots,l$, let $\alpha^\bs_k$ (resp. $\beta^\bs_k$) be the root (resp. coroot) defined by
\begin{align*}
    \alpha^\bs_k:= r_{s_l}\dots r_{s_{k+1}}\alpha_{s_k}, \quad 
    \beta^\bs_k:= r_{s_l}\dots r_{s_{k+1}}\alpha^\vee_{s_k}.
\end{align*}
Then for each $s \in S$, there exists a unique $k=k(s)$ such that $\beta^\bs_k=\alpha^\vee_s$. Note that from \cite[(3.10.3)]{Kac}, we have $\alpha^\bs_{k(s)}=\alpha_s$ at the same time. Let $\mathbf{H}(\bs)=(H(\bs),H_0(\bs),\ve(\bs),d(\bs))$ be the weighted quiver defined as follows.
\begin{itemize}
    \item $H(\bs)=H_0(\bs):=S$.
    \item The exchange matrix $\ve(\bs)=(\ve_{st})_{s,t \in S}$ is given by
    \begin{align*}
        \ve_{st}:=\frac{\mathrm{sgn}(k(t)-k(s))}{2}C_{st}.
    \end{align*}
    \item $d(\bs)=(d(\bs)_s)_{s \in S}$ is given by $d(\bs)_s:=d_s$. 
\end{itemize}
The vertex of $\mathbf{H}(\bs)$ corresponding to $s \in S$ is denoted by $y_s$. 
Then we define $\widetilde{\bJ}^+(\bs)$ to be the weighted quiver obtained from the disjoint union of the quivers $\bJ^+(\bs)$ and $\mathbf{H}(\bs)$ by adding the arrows $v_{k(s)-1} \to y_s$ and $y_s \to v_{k(s)}$ for each $s \in S$. As a convention, these additional arrows and the quiver $\mathbf{H}(\bs)$ are shown in blue in figures. 

\begin{ex}
Here are some examples of the quiver $\widetilde{\bJ}^+(\bs)$. 
\begin{enumerate}
    \item Type $A_3$, $\bs=(1,2,3,1,2,1)$. The sequence $(\alpha_k^\bs)_{k=1,\dots,6}$ is computed as 
    \begin{align*}
        &\alpha_6^\bs= \alpha_1, \quad \alpha_5^\bs=\alpha_1+\alpha_2, \quad \alpha_4^\bs=\alpha_2, \\
        &\alpha_3^\bs=\alpha_1+\alpha_2+\alpha_3, \quad \alpha_2^\bs=\alpha_2+\alpha_3, \quad \alpha_1^\bs=\alpha_3.
    \end{align*}
    Thus we get $k(1)=6,~k(2)=4,~k(3)=1$, and the quiver $\widetilde{\bJ}^+(1,2,3,1,2,1)$ is given by
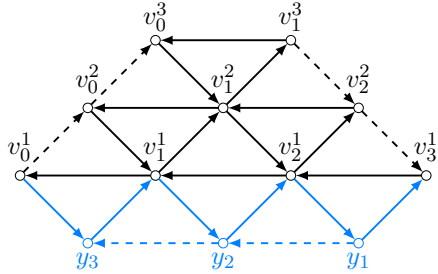
\[
\scalebox{0.9}{
\begin{tikzpicture}

\begin{scope}[>=latex]
\foreach \i in {0,1,2,3}
\draw (2*\i+1,0) circle(2pt) node[above=0.2em]{$v_{\i}^1$};
\foreach \i in {0,1,2}
\draw (2*\i+2, 1) circle(2pt) node[above=0.2em]{$v_{\i}^2$};
\foreach \i in {0,1}
\draw (2*\i+3, 2) circle(2pt) node[above]{$v_{\i}^3$};
    {\color{myblue}
    \foreach \j in {1,2,3}
    \draw (8-2*\j,-1) circle(2pt) node[below]{$y_\j$};
    }
\foreach \i in {1,2,3}
\qarrow{2*\i+1,0}{2*\i-1,0};
\foreach \i in {1,2}
\qarrow{2*\i+2,1}{2*\i,1};
\qarrow{5,2}{3,2};
\foreach \y in {1,2}
	{
	\qdarrow{\y,\y-1}{\y+1,\y};
	\qarrow{\y+2,\y-1}{\y+3,\y};
	}
\qarrow{5,0}{6,1};
\foreach \y in {1,2}
	{
	\qarrow{-\y+5,\y}{-\y+6,\y-1};
	\qdarrow{-\y+7,\y}{-\y+8,\y-1};
	}
\qarrow{2,1}{3,0};
{\color{myblue}
\qarrow{6,-1}{7,0};
\qarrow{5,0}{6,-1};
\qarrow{4,-1}{5,0};
\qarrow{3,0}{4,-1};
\qarrow{2,-1}{3,0};
\qarrow{1,0}{2,-1};
\qdarrow{4,-1}{2,-1};
\qdarrow{6,-1}{4,-1};
}
\end{scope}
\end{tikzpicture}}
\] 

    \item Type $C_3$, $\bs=(1,2,3,1,2,3,1,2,3)$. The sequence $(\alpha_k^\bs)_{k=1,\dots,9}$ is computed as
    \begin{align*}
        &\alpha_9^\bs= \alpha_3, \quad \alpha_8^\bs=\alpha_2+\alpha_3, \quad \alpha_7^\bs=\alpha_1+\alpha_2+\alpha_3, \\
        &\alpha_6^\bs= 2\alpha_2+\alpha_3, \quad \alpha_5^\bs=\alpha_1+2\alpha_2+\alpha_3, \quad \alpha_4^\bs=\alpha_2, \\
        &\alpha_3^\bs=2(\alpha_1+\alpha_2)+\alpha_3, \quad \alpha_2^\bs=\alpha_1+\alpha_2, \quad \alpha_1^\bs=\alpha_1.
    \end{align*}
    Thus we get $k(1)=1,~k(2)=4,=k(3)=9$, and the quiver $\widetilde{\bJ}^+((1,2,3)^3)$ is given by
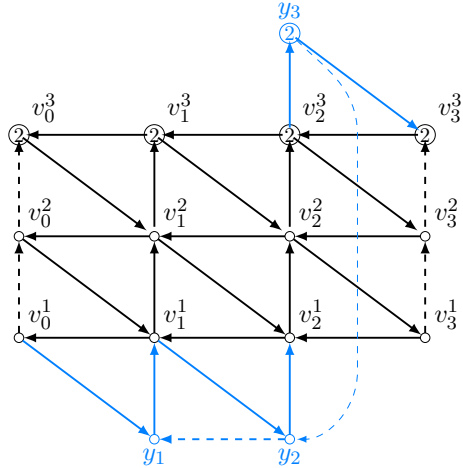
\[
\scalebox{0.9}{
\begin{tikzpicture}
\begin{scope}[>=latex]
\foreach \i in {0,1,2,3}
	\dnode{2*\i+1,3} node[above right=0.2em]{$v_{\i}^3$};
\foreach \i in {0,1,2,3}
{
    \draw(2*\i+1,0) circle(2pt) node[above right]{$v_{\i}^1$};
    \draw(2*\i+1,1.5) circle(2pt) node[above right]{$v_{\i}^2$};    
}
{\color{myblue}
    \foreach \j in {1,2}
        \draw(2*\j+1,-1.5) circle(2pt) node[below]{$y_\j$};
    \dnode{5,4.5} node[above=0.2em]{$y_3$};
}
\foreach \i in {0,1,2}
    {\qarrow{2*\i+3,0}{2*\i+1,0};
     \qarrow{2*\i+3,1.5}{2*\i+1,1.5};
     \qsarrow{2*\i+3,3}{2*\i+1,3};
    }
\qdarrow{7,0}{7,1.5}
\qarrow{5,1.5}{7,0}
\qarrow{5,0}{5,1.5}
\qarrow{3,1.5}{5,0}
\qarrow{3,0}{3,1.5}
\qarrow{1,1.5}{3,0}
\qdarrow{1,0}{1,1.5}

\qstdarrow{7,1.5}{7,3}
\qsharrow{5,3}{7,1.5}
\qstarrow{5,1.5}{5,3}
\qsharrow{3,3}{5,1.5}
\qstarrow{3,1.5}{3,3}
\qsharrow{1,3}{3,1.5}
\qstdarrow{1,1.5}{1,3}
{\color{myblue}
\foreach \j in {1,2}
    {
    \qarrow{2*\j-1,0}{2*\j+1,-1.5};
    \qarrow{2*\j+1,-1.5}{2*\j+1,0};
    }
\qdarrow{5,-1.5}{3,-1.5};
\qsarrow{5,3}{5,4.5};
\qsarrow{5,4.5}{7,3};
\draw[->,dashed,shorten >=3pt,shorten <=4pt] (5,4.5) to[out=-45, in=90] (6,3) to[out=270, in=90] (6,-0.5) to[out=270, in=0] (5,-1.5);
}
\end{scope}
\end{tikzpicture}}
\] 
\end{enumerate}
\end{ex}

\section{A short review on quotient stacks}\label{sec:stacks}
We shortly recall some basic facts on the \emph{quotient stacks}, to the minimal extent we need to recognize the moduli spaces $\P_{G,\Sigma}$ correctly. We refer the reader to \cite{Gomez,stacks_proj} for a self-contained presentation of the general theory of (algebraic) stacks. 
The lecture note \cite{Heinloth} will be also useful to get an intuition for stacks for the readers more familiar with the differential geometry or the algebraic topology than the algebraic geometry.  

Let $X$ be an algebraic variety (or more generally, a scheme), and $G$ an affine algebraic group acting on $X$ algebraically. 
In order to study the quotient of $X$ by $G$ from the viewpoint of algebraic geometry, a good way is to define it as a quotient stack $[X/G]$. Morally, the geometry of $[X/G]$ is the $G$-equivariant geometry of $X$. Several lemmas below will justify this slogan. 
When the action of $G$ is free, one can think of $[X/G]$ as an algebraic variety (\cref{lem:geometric_quotient}); in general, the quotient stack $[X/G]$ also contains the information on the stabilizers.

Let $X$ be a scheme over $\C$ and $G$ an affine algebraic group acting on $X$. Then the quotient stack $\X=[X/G]$ is a category fibered in groupoids (\cite[Definition 2.15]{Gomez}) where the objects over a scheme $B$ are pairs $(E,f)$ of a principal $G$-bundle $E \to B$ and a $G$-equivariant morphism $f:E \to X$; morphisms over $B$ are Cartesian diagrams of $G$-bundles which respect the equivariant morphisms to $X$. 

Note that an object $(E,f)$ over $B=\Spec\C$ can be viewed as a $G$-orbit in $X$. Thus the set $X/G$ of orbits is recovered as the set of images of $f:E \to X$, that is, the isomorphism classes of the objects of $\X(\Spec\C)$. Yoneda's lemma for stacks implies that a morphism $u: B \to \X$ from a scheme $B$ corresponds to an object of $\X(B)$. 

It is known that $\X$ is an Artin stack (\cite[Definition 2.26]{Gomez}): an atlas is defined by the morphism $X \to \X$ given by the pull-back of the trivial bundle $X \times G$ (see \cite[Example 2.29]{Gomez}). The structure sheaf on $\X$ and the $\C$-algebra $\cO(\X)$ of global functions (\emph{function algebra}) are defined, for example as in \cite[Section 4]{Heinloth}. One can verify the following: 

\begin{lem}
\label{lem:stack_functions}
$\cO(\X) \cong \cO(X)^G$.
\end{lem}
A scheme $V$ can be regarded as an Artin stack whose objects over $B$ are morphisms $B \to V$; the only morphism over $B$ is the identity. A stack is said to be \emph{representable} if it is isomorphic to a stack arising from a scheme. 
For two Artin stacks $\X$ and $\mathcal{Y}$, a morphism $\phi: \X \to \mathcal{Y}$ of stacks is said to be \emph{representable} if for any morphism $B \to \mathcal{Y}$, the fiber product $B \times_{\mathcal{Y}} \X$ is representable. Informally speaking, a morphism $B \to \mathcal{Y}$ can be viewed as a \lq\lq local chart'' on $\mathcal{Y}$, and the induced morphism $B \times_{\mathcal{Y}} \X \to B$ is the \lq\lq local expression'' of $\phi$. 
Here is an easy example:
\begin{lem}\label{lem:quotient_example}
Let $V$ be a scheme over $\C$ and $G$ an affine algebraic group. Consider the $G$-action on $V \times G$ given by the trivial action on the first factor and left multiplication on the second. Then the quotient stack $[V \times G/G]$ is representable by $V$.
\end{lem}

The following lemma tells us that we can obtain morphisms between quotient stacks from equivariant morphisms of varieties. 

\begin{lem}\label{lem:equivariant_morphism}
Let $\X=[X/H]$ and $\mathcal{Y}=[Y/G]$ be two quotient stacks. 
Let $\phi:X \to Y$ be a morphism equivariant with respect to an embedding $\tau:H \to G$. Then it induces a representable morphism $\phi_\ast: \X \to \mathcal{Y}$ of Artin stacks. More precisely, for any morphism $u: B \to \mathcal{Y}$ from a scheme $B$ which corresponds to an object $(E,f) \in \mathcal{Y}(B)$, the diagram
\begin{equation*}
    \begin{tikzcd}
    E \times_{G} (G/H) \ar[r,"u_H"] \ar[d,"\phi"'] & {\X} \ar[d,"\phi_\ast"] \\
    B \ar[r,"u"'] & {\mathcal{Y}}
    \end{tikzcd}
\end{equation*}
is Cartesian. Here $u_H$ is a morphism corresponding to an $H$-bundle $E \to E \times (G/H) \to E \times_G (G/H)$. 
\end{lem}
We give a proof for our convenience.

\begin{proof}
    For an object $(E,f)$ over $B$ in $\X$, the pair $\phi_*(E,f):=(E,\phi\circ f)$ is an object over $B$ in $\mathcal{Y}$. This correspondence is clearly compatible with pull-backs, and defines a morphism $\phi_\ast: \X \to \mathcal{Y}$. It is not hard to see that the fiber product $B \times_{\mathcal{Y}}\X$ is isomorphic to $E\times_G (G/H)$, with a notice that an $H$-sub-bundle of a $G$-bundle $P \to B$ is in one-to-one correspondence with a section of the associated bundle $P \times_G (G/H)$. 
    Thus $\phi_\ast$ it is representable. 
\end{proof}
We call $\phi$ a \emph{presentation} of the morphism $\phi_*$. 
\begin{rem}
When $\tau$ is not an embedding, the morphism $\phi_\ast$ is not representable in general. For instance, the morphism $(\mathrm{id}_{pt})_\ast: [pt/G] \to [pt/e] = pt$ is not representable for a non-trivial group $G$, where $pt$ denotes the point scheme and $e$ is the trivial group. 
\end{rem}

A property of morphisms of schemes that are local in nature on the target and stable under base-change can be defined for representable morphisms of stacks. For instance, $\phi_\ast$ is said to be an open embedding if $\phi:X=Y \times_{[Y/G]} [X/G] \to Y$ is an open embedding of varieties.

\begin{lem}\label{lem:geometric_quotient}
Suppose that $X$ is an affine algebraic variety, $G$ is reductive, and the $G$-action on $X$ is free. Then the quotient stack $\X=[X/G]$ is representable by the geometric quotient $X/G$.
\end{lem}
Indeed, the assumptions imply that all points of $X$ are $G$-stable, and the geometric quotient $X/G$ exists (see, for instance, \cite[Proposition 1.26]{Brion}).

\bibliographystyle{alpha}

\end{document}